\documentclass[12pt]{article}
\usepackage{amsmath, amsfonts, amssymb, amsthm, amscd}
\usepackage{graphicx}
\usepackage{caption, subcaption}
\usepackage{hyperref}

\newtheorem{thm}{Theorem}
\newtheorem{prop}[thm]{Proposition}
\newtheorem{lem}[thm]{Lemma}

\newtheorem{cor}[thm]{Corollary}
\newtheorem{rem}[thm]{Remark}
\newtheorem{df}[thm]{Definition}
\newtheorem{ex}[thm]{Example}

\renewcommand{\epsilon}{\varepsilon}
\renewcommand{\phi}{\varphi}

\newcommand{\BB}{\mathbb}
\newcommand{\g}{\mathfrak}
\newcommand{\separate}{\vskip5pt}
\newcommand{\supp}{\operatorname{supp}}
\newcommand{\im}{\operatorname{Im}}
\newcommand{\re}{\operatorname{Re}}
\newcommand{\tr}{\operatorname{Tr}}
\newcommand{\End}{\operatorname{End}}

\newcommand{\HC}{\operatorname{HC}}

\begin{document}

\LARGE

%\title{{\bf M607 Group Representations I} \\
%Representations of Non-Compact Real Semisimple Lie Groups}

\title{Introduction to Representations of Real Semisimple Lie Groups}
\author{Matvei Libine}
%\date{}
\maketitle

\begin{abstract}
These are lecture notes for a one semester introductory course I gave at
Indiana University. The goal was to make this exposition as clear and
elementary as possible.
A particular emphasis is given on examples involving $SU(1,1)$.
These notes are in part based on lectures given by my graduate advisor
Wilfried Schmid at Harvard University and PQR2003 Euroschool in Brussels
\cite{SchL} as well as other sources, such as
\cite{HT, Kn1, Kn2, Vi, Wal1, War}.
The text is formated for convenient viewing from iPad and other tablets.
\end{abstract}

\tableofcontents

\newpage

\section{Introduction}

Let $G$ be a real semisimple or reductive Lie group
(intuitively this means that its Lie algebra $\mathfrak{g}$ is semisimple
or reductive). \\
Examples: $SL(n,\BB R)$, $O(p,q)$, $SU(p,q)$, \dots are real semisimple
Lie groups and $GL(n,\BB R)$, $U(p,q)$, \dots are reductive real Lie groups. \\
$\hat G \: =$ set of isomorphism classes of irreducible unitary representations.
Thus, for each $i \in \hat G$, we get an irreducible unitary representation
$(\pi_i, V_i)$.

\separate

\noindent
\underline{$G$ -- compact}: $\dim V_i < \infty$,
unitarity is automatic (every irreducible representation is unitary),
$\pi_i: G \times V_i \to V_i$ is real  analytic, so get a representation of
$\mathfrak g$ on $V_i$.

\separate

\noindent
\underline{Non-compact}: Typically $\dim V_i = \infty$,
there are many interesting non-unitary representations,
analyticity fails miserably. To get around this problem and construct
representation of $\mathfrak g$, look at the space of ${\cal C}^{\infty}$ or
${\cal C}^{\omega}$ vectors; solution: Harish-Chandra modules.

\separate

\noindent
\underline{$G$ -- compact}: We have Peter-Weyl Theorem:
$$
L^2(G) \simeq \widehat{\bigoplus}_{i \in \hat G} V_i \otimes V_i^*
\simeq \widehat{\bigoplus}_{i \in \hat G} \End(V_i)
$$
$G \times G$ equivariant isomorphism of vector spaces,
the hat denotes the Hilbert space direct sum.
We also have a Fourier Inversion Theorem: Define
$$
{\cal C}^{\infty}(G) \ni \quad f \mapsto
\hat f(i) = \int_G f(g) \cdot \pi_i(g) \, dg \quad \in \End(V_i),
$$
then
$$
f(e)= \sum_{i \in \hat G} \tr \hat f(i) \cdot \dim V_i.
$$
Compare this with Fourier series:
$$
f(x) = \sum_{n \in \BB Z} \hat f(n) \cdot e^{inx}, \qquad
f(0) = \sum_{n \in \BB Z} \hat f(n),
$$
$$
\hat f(n) = \frac1{2\pi} \int_{-\pi}^{\pi} f(x) \cdot e^{-inx} \,dx, \qquad
f \in {\cal C}^{\infty}(\BB R / 2\pi \BB Z).
$$

\separate

\noindent
\underline{Non-compact}: Not every irreducible unitary representation occurs
in $L^2(G)$. In fact, it is still an open problem to describe $\hat G$ for
an arbitrary real semisimple Lie group.
We have an abstract Fourier Inversion Theorem:
$$
L^2(G) \simeq \int_{i \in \hat G} V_i \hat\otimes V_i^* \,d\mu(i)
\simeq \int_{i \in \hat G} \End_{HS}(V_i) \,d\mu(i),
$$
where $\mu(i)$ denotes a measure on $\hat G$, called Plancherel measure.
For compact $G$, $\mu(i)=\dim V_i$ and $\supp(\mu) =\hat G$.
For non-compact $G$, $\supp(\mu) \subsetneq \hat G$.
$$
f(e) = \int_{\hat G} \tr \hat f(i) \,d\mu(i), \qquad
f \in {\cal C}^{\infty}_c(G).
$$

\separate

\noindent
\underline{$G$ -- compact}: We have Borel-Weil-Bott Theorem for \linebreak
$H^i(G_{\BB C}/B_{\BB C}, {\cal L}_{\lambda})$, irreducible representations are
uniquely determined by their characters,Weyl and Kirillov's character formulas.

\separate

\noindent
\underline{Non-compact}: Na\"ive notion of character fails.
Harish-Chandra: get {\em character} as a conjugation-invariant distribution
on $G$, does determine representation up to equivalence.
Regularity theorem: characters are real-analytic functions with singularities.
No totally explicit formula for characters.
Kashiwara-Schmid \cite{KSch} constructed representations of $G$ in
$H^i(G_{\BB C}/B_{\BB C}, {\cal F})$, where ${\cal F}$ is a $G$-equivariant sheaf
on $G_{\BB C}/B_{\BB C}$.
This can be regarded as an analogue of Borel-Weil-Bott theorem.
Then Schmid-Vilonen \cite{SchV} gave analogues of Weil and Kirillov's character
formulas.

\separate

\noindent
\underline{$G$ -- compact}: If $H \subset G$ is a closed subgroup,
$$
L^2(G/H) \hookrightarrow L^2(G).
$$

\separate

\noindent
\underline{Non-compact}: If $H \subset G$ is a closed non-compact subgroup,
$$
L^2(G/H) \not\hookrightarrow L^2(G).
$$
For example, if $G=SL(2,\BB R)$, $\Gamma \subset G$ a discrete subgroup,
even then $L^2(G/\Gamma)$ is not known in general.

\separate

\noindent
\underline{Strategy for understanding representations}:
In order to understand representations of compact groups $K$, we restricted
them to a maximal torus $T$, and the representation theory for $T$ is trivial.
Equivalently, in order to understand finite-dimensional representations of a
semisimple Lie algebra $\mathfrak k$, we restrict them to a Cartan subalgebra
$\mathfrak h$. Now, to understand representations of a real semisimple
Lie group $G$ we restrict them to a maximal compact subgroup $K$,
and the representation theory for $K$ is assumed to be understood.
(A particular choice of a maximal compact subgroup $K \subset G$
does not matter, just like a particular choice of a maximal torus $T \subset K$
or a Cartan subalgebra $\mathfrak h \subset \mathfrak k$ is not essential.)

\section{Some Geometry and Examples of Representations}

\subsection{Complex Projective Line}

The complex projective line, usually denoted by $\BB CP^1$ or $\BB P^1(\BB C)$,
is $\BB C^2 \setminus \{0\}$ modulo equivalence, where
$$
\begin{pmatrix} z_1 \\ z_2 \end{pmatrix} \sim
\begin{pmatrix} \lambda z_1 \\ \lambda z_2 \end{pmatrix},
\qquad \forall \lambda \in \BB C^{\times}.
$$
Note that
$$
\begin{pmatrix} z_1 \\ z_2 \end{pmatrix} \sim
\begin{pmatrix} z_1/z_2 \\ 1 \end{pmatrix} \: \text{ if $z_2 \ne 0$}, \quad
\begin{pmatrix} z_1 \\ z_2 \end{pmatrix} \sim
\begin{pmatrix} 1 \\ z_2/z_1 \end{pmatrix} \: \text{ if $z_2 \ne 0$}.
$$
We identify points
$\bigl(\begin{smallmatrix} z \\ 1 \end{smallmatrix}\bigr) \in \BB CP^1$
with $z \in \BB C$, then the complement of $\BB C$ in $\BB CP^1$ is
$\bigl(\begin{smallmatrix} 1 \\ 0 \end{smallmatrix}\bigr)$,
and we think of this point as $\infty$. Thus
$$
\BB CP^1 = \BB C \cup \{\infty\} \approx S^2
\quad \text{(as ${\cal C}^{\infty}$ manifolds).}
$$

The group $SL(2,\BB C)$ acts on $\BB CP^1$ by fractional linear transformations:
$$
\begin{pmatrix} a & b \\ c & d \end{pmatrix}
\begin{pmatrix} z_1 \\ z_2 \end{pmatrix}
= \begin{pmatrix} az_1 + bz_2 \\ cz_1 + dz_2 \end{pmatrix},
$$
$$
\begin{pmatrix} a & b \\ c & d \end{pmatrix}
\begin{pmatrix} z \\ 1 \end{pmatrix}
= \begin{pmatrix} az+b \\ cz+d \end{pmatrix}
\sim \begin{pmatrix} \frac{az+b}{cz+d} \\ 1 \end{pmatrix}.
$$
The map $z \mapsto \frac{az+b}{cz+d}$ from $\BB C \cup \{\infty\}$ to itself
is also called a M\"obius or fractional linear transformation.
Note that the center $\{\pm Id\} \subset SL(2,\BB C)$ acts on $\BB CP^1$
trivially, so we actually have an action on $\BB CP^1$ of
$$
PGL(2,\BB C) = GL(2,\BB C)/\{\text{center}\} = SL(2,\BB C)/\{\pm Id\}.
$$

Recall that an {\em automorphism} of $\BB CP^1$ or any complex manifold $M$ is
a complex analytic map $M \to M$ which has an inverse that is also complex
analytic.

\begin{thm}
All automorphisms of $\BB CP^1$ are of the form
$$
z \mapsto \frac{az+b}{cz+d}, \qquad
\begin{pmatrix} a & b \\ c & d \end{pmatrix} \in SL(2,\BB C).
$$
Moreover, two matrices $M_1, M_2 \in SL(2,\BB C)$ produce the same
automorphism on $\BB CP^1$ if and only if $M_2 = \pm M_1$.
\end{thm}

Next we turn our attention to the upper half-plane
$$
\BB H = \{ z \in \BB C ;\: \im z >0 \}.
$$
Note that $SL(2,\BB R)$ preserves $\BB H$:
If $a,b,c,d \in \BB R$ and $\im z>0$, then
\begin{multline*}
\im \biggl( \frac{az+b}{cz+d} \biggr)
= \frac1{2i} \biggl( \frac{az+b}{cz+d} - \frac{a\bar z+b}{c\bar z+d} \biggr) \\
= \frac1{2i} \frac{(az+b)(c\bar z+d)-(a\bar z+b)(cz+d)}{|cz+d|^2} \\
= \frac1{2i} \frac{(ad-bc)(z-\bar z)}{|cz+d|^2}
= \frac{\im z}{|cz+d|^2} \quad>0.
\end{multline*}
Thus the group $SL(2,\BB R)$ acts on $\BB H$ by automorphisms.

\begin{thm}
All automorphisms of the upper half-plane $\BB H$ are of the form
$$
z \mapsto \frac{az+b}{cz+d}, \qquad
\begin{pmatrix} a & b \\ c & d \end{pmatrix} \in SL(2,\BB R).
$$
Moreover, two matrices $M_1, M_2 \in SL(2,\BB R)$ produce the same
automorphism on $\BB H$ if and only if $M_2 = \pm M_1$.
\end{thm}

The group $PGL(2,\BB R)=SL(2,\BB R)/\{\pm Id\}$ is also the automorphism group
of the lower half-plane $\bar{\BB H} = \{ z \in \BB C ;\: \im z <0 \}$
and preserves $\BB R \cup \{\infty\}$.

Note that the Cayley transform
$$
z \mapsto \frac{z-i}{z+i},
$$
which is the fractional linear transformation associated to the matrix
$C = \bigl(\begin{smallmatrix} 1 & -i \\ 1 & i \end{smallmatrix}\bigr)$,
takes $\BB H$ to the unit disk
$$
\BB D = \{ z \in \BB C ;\: |z|<1 \}.
$$
Of course, the group of automorphisms of $\BB D$ is obtained from
the group of automorphisms of $\BB H$ by conjugating by $C$:
$$
C \cdot PGL(2,\BB R) \cdot C^{-1} \quad \subset \quad PGL(2,\BB C).
$$

\begin{lem}
We have $C \cdot SL(2,\BB R) \cdot C^{-1} = SU(1,1)$, where
$$
SU(1,1) = \biggl\{ \begin{pmatrix} a & b \\ \bar b & \bar a \end{pmatrix}
\in SL(2,\BB C);\: a,b \in \BB C,\: |a|^2-|b|^2=1 \biggr\}.
$$
In particular, the groups $SL(2,\BB R)$ and $SU(1,1)$ are isomorphic.
\end{lem}

We immediately obtain:

\begin{thm}
All automorphisms of the unit disk $\BB D$ are of the form
$$
z \mapsto \frac{az+b}{cz+d}, \qquad
\begin{pmatrix} a & b \\ c & d \end{pmatrix} \in SU(1,1).
$$
Moreover, two matrices $M_1, M_2 \in SU(1,1)$ produce the same
automorphism on $\BB D$ if and only if $M_2 = \pm M_1$.
\end{thm}

The group $SU(1,1)/\{\pm Id\}$ is also the automorphism group of
$\{ z \in \BB C ;\: |z|>1 \} \cup \{\infty\}$
and preserves the unit circle $S^1 = \{ z \in \BB C ;\: |z|=1 \}$.

\subsection{Examples of Representations}

In general, if $G$ is any kind of group acting on a set $X$,
then we automatically get a representation $\pi$ of $G$ in the vector space
$V$ consisting of complex (or real) valued functions on $X$:
$$
\bigl( \pi(g) f \bigr)(x) = f(g^{-1} \cdot x),
\qquad g \in G, \: f \in V, \: x \in X.
$$
Note that the inverse in $f(g^{-1} \cdot x)$ ensures
$\pi(g_1) \pi(g_2) = \pi(g_1g_2)$, without it we would have
$\pi(g_1) \pi(g_2) = \pi(g_2g_1)$.

\separate

Now, let
$$
V = \{ \text{holomorphic functions on the upper half-plane $\BB H$} \},
$$
clearly, $V$ is a vector space over $\BB C$ of infinite dimension.
Define a representation of $SL(2,\BB R)$ on $V$ by
$$
\bigl( \pi(g) f \bigr)(z) = f(g^{-1} \cdot z),
\qquad g \in SL(2,\BB R), \: f \in V, \: z \in \BB H.
$$
Explicitly, if
$g = \bigl(\begin{smallmatrix} a & b \\ c & d \end{smallmatrix}\bigr)$, then
$g^{-1} = \bigl(\begin{smallmatrix} d & -b \\ -c & a \end{smallmatrix}\bigr)$,
and
$$
\bigl( \pi(g) f \bigr)(z) = f(g^{-1} \cdot z)
= f \biggl( \frac{dz-b}{-cz+a} \biggr).
$$

Note that $V$ is not irreducible -- it contains a subrepresentation $V_0$
spanned by the constant functions, which form the one-dimensional trivial
representation. On the other hand, $V$ is indecomposable, i.e. cannot
be written as a direct sum of two subrepresentations
(this part is harder to see). Once we learn about Harish-Chandra modules,
it will be easy to see that $V_0$ does not have an $SL(2,\BB R)$-invariant
direct sum complement, in fact, $V_0$ is the only closed invariant subspace
of $V$, and the quotient representation on $V/V_0$ is irreducible.
Note also that $(\pi, V)$ is {\em not} unitary
(if it were unitary, then one could take the orthogonal complement of $V_0$
and write $V$ as a direct sum of subrepresentations $V_0 \oplus V_0^{\perp}$).

This example is in many ways typical to representations of real semisimple
Lie groups: the vector spaces are infinite-dimensional and have some sort
of topological structure, the representations may not have unitary structure,
need not decompose into direct sum of irreducible subrepresentations and may
have interesting components appearing as quotients.
This is quite different from the representations of compact Lie groups.

Let us modify the last example by letting
$$
V = \{ \text{holomorphic functions on the unit disk $\BB D$} \},
$$
and defining a representation of $SU(1,1)$ on $V$ by
$$
\bigl( \pi(g) f \bigr)(z) = f(g^{-1} \cdot z),
\qquad g \in SU(1,1), \: f \in V, \: z \in \BB D.
$$
Explicitly, if $g = \bigl(\begin{smallmatrix} a & b \\ \bar b & \bar a
\end{smallmatrix}\bigr)$, then
$g^{-1} = \bigl(\begin{smallmatrix} \bar a & -b \\ -\bar b & a
\end{smallmatrix}\bigr)$, and
$$
\bigl( \pi(g) f \bigr)(z) = f(g^{-1} \cdot z)
= f \biggl( \frac{\bar az-b}{-\bar bz+a} \biggr).
$$

Of course, this is the same example of representation as before.
However, this $SU(1,1)$ acting on the unit disk $\BB D$ model has
a certain advantage. We already mentioned that the strategy for understanding
representations of real semisimple Lie groups is by reduction to maximal
compact subgroup.
In the case of $SU(1,1)$, a maximal compact subgroup can be chosen to be
the subgroup of diagonal matrices
$$
K = \biggl\{ k_{\theta} = \begin{pmatrix} e^{i\theta} & 0 \\ 0 & e^{-i\theta}
\end{pmatrix} ;\: \theta \in [0,2\pi) \biggr\} \quad \subset SU(1,1).
$$
In the case of $SL(2,\BB R)$, a maximal compact subgroup can be chosen to be
the subgroup of rotation matrices
$$
K' = \biggl\{ k'_{\theta} = \begin{pmatrix} \cos\theta & -\sin\theta \\
\sin\theta & \cos\theta \end{pmatrix} ;\:
\theta \in [0,2\pi) \biggr\} \quad \subset SL(2,\BB R).
$$
The action of $k_{\theta}$ is particularly easy to describe:
$$
k_{\theta} \cdot z = e^{2i\theta} z, \qquad
\bigl( \pi(k_{\theta}) f \bigr)(z) = f(e^{-2i\theta} z)
$$
-- these are rotations by $2\theta$.
On the other hand, the action of $k'_{\theta}$ is much harder to visualize:
$$
k'_{\theta} \cdot z = \frac{\cos\theta z - \sin\theta}{\sin\theta z +\cos\theta},
\quad
\bigl( \pi(k'_{\theta}) f \bigr)(z) = f\biggl( \frac{\cos\theta z + \sin\theta}
{-\sin\theta z +\cos\theta} \biggr),
$$
but one can identify $K'$ with the isotropy subgroup of $i \in \BB H$.

Let us find the eigenvectors in
$$
V = \{ \text{holomorphic functions on the unit disk $\BB D$} \}
$$
with respect to the action of $K$.
These are functions $z^m$, $m=0,1,2,3,\dots$
$$
\bigl( \pi(k_{\theta}) z^m \bigr)(z) = (e^{-2i\theta} z)^m = e^{-2mi\theta} \cdot z^m.
$$
The algebraic direct sum of eigenspaces $\bigoplus_{m \ge 0} \BB C \cdot z^m$
is just the subspace of polynomial functions on $\BB C$.
These are dense in $V$:
$$
V = \overline{\bigoplus_{m \ge 0} \BB C \cdot z^m}.
$$
Later we will see that $\bigoplus_{m \ge 0} \BB C \cdot z^m$ is the
underlying Harish-Chandra module of $V$.

\separate

Fix an integer $n \ge 0$ and let
$$
V = \{ \text{holomorphic functions on the unit disk $\BB D$} \}.
$$
Define a representation of $SU(1,1)$ by
\begin{multline*}
\bigl( \pi_n(g) f \bigr)(z)
= (-\bar bz+a)^{-n} \cdot f \biggl( \frac{\bar az-b}{-\bar bz+a} \biggr), \\
g = \begin{pmatrix} a & b \\ \bar b & \bar a \end{pmatrix} \in SU(1,1), \:
f \in V, \: z \in \BB D.
\end{multline*}
You can verify that $\pi_n(g_1) \pi_n(g_2) = \pi_n(g_1g_2)$ is true,
so we get a representation.
If $n=0$ we get the same representation as before.
If $n \ge 2$ we get the so-called
{\em holomorphic discrete series} representation,
and if $n=1$ we get the so-called
{\em limit of the holomorphic discrete series} representation.
To get the so-called {\em antiholomorphic discrete series and its limit}
representations one uses the antiholomorphic functions on $\BB D$.

\separate

Finally, we describe the (not necessarily unitary) {\em principal series}
representations of $SU(1,1)$.
These require a parameter $\lambda \in \BB C$.
Let $S^1 = \{z \in \BB C;\: |z|=1 \}$ be the unit circle and define
$$
\tilde V = \{ \text{smooth functions on the unit circle $S^1$} \},
$$
\begin{multline*}
\bigl( \pi^+_{\lambda}(g) f \bigr)(z)
%= \bigl( \sgn(-\bar bz+a) \bigr)^{\epsilon} \cdot
= |-\bar bz+a|^{-1-\lambda} \cdot f \biggl( \frac{\bar az-b}{-\bar bz+a} \biggr),\\
g = \begin{pmatrix} a & b \\ \bar b & \bar a \end{pmatrix} \in SU(1,1), \:
f \in \tilde V, \: z \in S^1.
\end{multline*}
It is still true that
$\pi^+_{\lambda}(g_1) \pi^+_{\lambda}(g_2) = \pi^+_{\lambda}(g_1g_2)$,
so we get a representation.
Note that $\pi^+_{\lambda}(-Id)=Id_{\tilde V}$, there is another principal series
with the property $\pi^-_{\lambda}(-Id)=-Id_{\tilde V}$.

Later we will construct the principal series representations for every
real semisimple Lie group $G$. These representations are very important
because, in a certain sense, every representation of $G$ occurs as a
subrepresentation of some principal series representation.

\section{The Universal Enveloping Algebra}

In this section, let $\Bbbk$ be any field (not necessarily $\BB R$ or $\BB C$)
and let $\g g$ be a Lie algebra over $\Bbbk$.
We give a brief overview of properties of the universal enveloping algebra
associated to $\g g$, proofs and details can be found in, for example,
\cite{Va1, Kir}.

\subsection{The Definition}

We would like to embed $\g g$ into a (large) associative algebra ${\cal U}$
so that
$$
[X,Y]=XY-YX, \qquad \forall X,Y \in \g g,
$$
and ${\cal U}$ is ``as close to the free associative algebra as possible''.
First we consider the tensor algebra
$$
\bigotimes \g g = \bigoplus_{n=0}^{\infty} (\otimes^n \g g), \qquad
\otimes^n \g g = \underbrace{\g g \otimes \dots \otimes \g g}_\text{$n$ times},
\quad \otimes^0 \g g = \Bbbk.
$$
Then $\bigotimes \g g$ is an algebra over $\Bbbk$ which is associative,
non-commutative, generated by $\g g$ and has a unit.
It is universal in the sense that any other $\Bbbk$-algebra
with those properties is a quotient of $\bigotimes \g g$.

Similarly, we can define
$$
S(\g g) = \bigoplus_{n=0}^{\infty} S^n(\g g),
$$
where $S^0(\g g)= \Bbbk$ and
$$
S^n(\g g)= \text{$n^{\text{th}}$ symmetric power of $\g g$.}
$$
In other words, $S(\g g)$ is the algebra of polynomials over $\g g^*$
and $S^n(\g g)$ is the subspace of homogeneous polynomials over $\g g^*$
of homogeneity degree $n$, where $\g g^*$ is the vector space dual of $\g g$.
Note that $S(\g g)$ contains $\g g$ as a linear subspace and is generated by it.
Then $S(\g g)$ is an algebra over $\Bbbk$ which is associative, commutative,
generated by $\g g$ and has unit.
It is universal in the sense that any other $\Bbbk$-algebra with those
properties is a quotient of $S(\g g)$.
The algebra $S(\g g)$ can be realized as a quotient of $\bigotimes \g g$:
$$
S(\g g) = \bigotimes \g g / I
$$
where $I$ is the two-sided ideal generated by $XY-YX$, $\forall X,Y \in \g g$.
If $\operatorname{char} \Bbbk =0$, we can also realize $S(\g g)$ as a subalgebra
of $\bigotimes \g g$.
By construction, both $\bigotimes \g g$ and $S(\g g)$ are graded algebras:
$$
\bigotimes \g g = \bigoplus_{n=0}^{\infty} (\otimes^n \g g), \qquad
(\otimes^m \g g) \cdot (\otimes^n \g g) \subset \otimes^{m+n} \g g,
$$
$$
S(\g g) = \bigoplus_{n=0}^{\infty} S^n(\g g), \qquad
S^m(\g g) \cdot S^n(\g g) \subset S^{m+n}(\g g).
$$

\begin{df}
The {\em universal enveloping algebra} of $\g g$ is
$$
{\cal U}(\g g) = \bigotimes \g g /J,
$$
where $J$ is the two-sided ideal generated by
$$
XY-YX-[X,Y], \qquad \forall X,Y \in \g g.
$$
\end{df}

Then ${\cal U}(\g g)$ is an associative algebra over $\Bbbk$ with unit
equipped with a $\Bbbk$-linear map $i: \g g \to {\cal U}(\g g)$ such that
\begin{enumerate}
\item
$i$ is a Lie algebra homomorphism, where ${\cal U}(\g g)$ is given
structure of a Lie algebra by $[a,b]=_{\text{def}} ab-ba$;
\item
${\cal U}(\g g)$ is generated by $i(\g g)$ as a $\Bbbk$-algebra.
\end{enumerate}

We will see soon that $i: \g g \to {\cal U}(\g g)$ is an injection,
but it is not clear at this point.

\noindent
\underline{Caution}:
Let
$E= \bigl(\begin{smallmatrix} 0 & 1 \\ 0 & 0 \end{smallmatrix}\bigr)
\in \mathfrak{sl}(2,\BB C)$,
then $E^2=0$ in the sense of multiplication of $2 \times 2$ matrices, but
$E^2 \ne 0$ as an element of ${\cal U}(\mathfrak{sl}(2,\BB C))$.

The algebra ${\cal U}(\g g)$ is {\em universal} in the following sense:

\begin{thm}
Let $A$ be any associative algebra over $\Bbbk$ with unit, and let
$\rho: \g g \to A$ be a $\Bbbk$-linear map such that
$$
\rho(X)\rho(Y) - \rho(Y)\rho(X) = \rho([X,Y]), \qquad \forall X,Y \in \g g.
$$
Then $\rho$ can be uniquely extended to a morphism of associative algebras
with unit $\rho: {\cal U}(\g g) \to A$.
\end{thm}

\begin{cor}
Any representation of $\g g$ (not necessarily finite-dimensional)
has a canonical structure of a ${\cal U}(\g g)$-module.
Conversely, every ${\cal U}(\g g)$-module has a canonical structure of a
representation of $\g g$.
\end{cor}

\noindent
\underline{Restatement}: The categories of representations of $\g g$ and
${\cal U}(\g g)$-modules are equivalent.

\subsection{Some Properties of the Universal Enveloping Algebra}

Note that ${\cal U}(\g g)$ is a {\em filtered} algebra: Let
$$
{\cal U}_n(\g g) =_{\text{def}} \text{ image of }
\bigoplus_{k=o}^n (\otimes^k \g g)
\text{ in } {\cal U}(\g g)=\bigotimes \g g /J,
$$
then
$$
\Bbbk = {\cal U}_0(\g g) \subset {\cal U}_1(\g g) \subset \dots \subset
{\cal U}_n(\g g) \subset \dots \subset
{\cal U}(\g g) = \bigcup_{n=0}^{\infty} {\cal U}_n(\g g)
$$
$$
\text{and} \qquad
{\cal U}_m(\g g) \cdot {\cal U}_n(\g g) \subset {\cal U}_{m+n}(\g g).
$$

\noindent
\underline{Caution}: ${\cal U}(\g g)$ is {\em not} a graded algebra.

\begin{prop}
\begin{enumerate}
\item
If $X \in {\cal U}_m(\g g)$ and $Y \in {\cal U}_n(\g g)$, then
$$
XY-YX \in {\cal U}_{m+n-1}(\g g).
$$
\item
Let $\{X_1,\dots,X_r\}$ be an ordered basis of $\g g$, the monomials
$$
(X_1)^{k_1} \cdot (X_2)^{k_2} \cdot \ldots \cdot (X_r)^{k_r},  \quad
\sum_{i=1}^r k_i \le n, \quad \text{span } {\cal U}_n(\g g).
$$
(Note that we have fixed the order of basis elements.)
\end{enumerate}
\end{prop}

\begin{proof}
We prove the first part by induction on $m$. So suppose first that
$m=1$ and $X \in i(\g g)$, then
\begin{multline*}
XY_1 \dots Y_n - Y_1 \dots Y_nX \\
= XY_1Y_2 \dots Y_n - Y_1XY_2 \dots Y_n \\
+ Y_1XY_2 \dots Y_n - Y_1Y_2XY_3 \dots Y_n \\
\dots \\
+ Y_1 \dots Y_{n-1}XY_n - Y_1 \dots Y_nX \\
= \sum_{i=1}^n Y_1 \dots [X,Y_i] \dots Y_n \quad \in {\cal U}_n(\g g).
\end{multline*}
This implies if $X \in {\cal U}_1(\g g)$, $Y \in {\cal U}_n(\g g)$, then
$[X,Y] \in {\cal U}_n(\g g)$, i.e.
$$
XY \equiv YX \mod {{\cal U}_n(\g g)}.
$$
Then
\begin{multline*}
X_1 \dots X_m X_{m+1} Y \equiv X_1 \dots X_m Y X_{m+1} \\
\equiv Y X_1 \dots X_mX_{m+1} \mod {{\cal U}_{m+n}(\g g)}
\end{multline*}
and the first part follows.

To prove the second part we also use induction on $n$.
If $n=1$ the statement amounts to
$$
{\cal U}_1(\g g) = \Bbbk\text{-}\operatorname{Span} \{ 1,X_1,\dots,X_r \},
$$
which is true. Note that
${\cal U}_{n+1}(\g g) = {\cal U}_1(\g g) \cdot {\cal U}_n(\g g)$.
By induction hypothesis,
$$
{\cal U}_n(\g g) = \Bbbk\text{-}\operatorname{Span} \Bigl\{
(X_1)^{k_1} \cdot (X_2)^{k_2} \cdot \ldots \cdot (X_r)^{k_r} ;\:
\sum_{i=1}^r k_i \le n \Bigr \}.
$$
By the first part,
$$
X_j \cdot (X_1)^{k_1} \cdot \ldots \cdot (X_r)^{k_r}
- (X_1)^{k_1} \cdot \ldots \cdot (X_j)^{k_j+1} \cdot \ldots \cdot (X_r)^{k_r}
$$
lies in ${\cal U}_n(\g g)$.
Hence $X_j \cdot (X_1)^{k_1} \cdot \ldots \cdot (X_r)^{k_r}$ lies in
$$
\Bbbk\text{-}\operatorname{Span} \Bigl\{
(X_1)^{k_1} \cdot (X_2)^{k_2} \cdot \ldots \cdot (X_r)^{k_r} ;\:
\sum_{i=1}^r k_i \le n+1 \Bigr \}.
$$
\end{proof}

\begin{cor}
The associated graded algebra
$$
Gr {\cal U}(\g g) =_{\text{def}}
\bigoplus_{n=0}^{\infty} {\cal U}_n(\g g) / {\cal U}_{n-1}(\g g)
$$
is commutative.
\end{cor}

Since $S(\g g)$ is a universal associative commutative algebra,
we get a unique map $S(\g g) \to Gr {\cal U}(\g g)$ such that
$$
S(\g g) \ni \quad X \mapsto i(X) \quad \in Gr {\cal U}(\g g),
\qquad \forall X \in \g g.
$$

\begin{thm}[Poincar\'e-Birkhoff-Witt]
This map $S(\g g) \to Gr {\cal U}(\g g)$ is an isomorphism of algebras.
\end{thm}

The Poincar\'e-Birkhoff-Witt Theorem is valid over {\em any} field
$\Bbbk$, even if $\operatorname{char} \Bbbk \ne 0$.
In Subsection \ref{U(g)-geom} we will outline a proof of
this theorem for the special case when $\Bbbk$ is $\BB R$ or $\BB C$.

\begin{cor}
The ordered monomials
$$
\Bigl\{ (X_1)^{k_1} \cdot (X_2)^{k_2} \cdot \dots \cdot (X_r)^{k_r}; \:
\sum_{i=1}^r k_i \le n \Bigr\}
$$
form a vector space basis of ${\cal U}_n(\g g)$.
\end{cor}

\begin{cor}
The map $i: \g g \to {\cal U}(\g g)$ is injective,
hence $\g g$ can be regarded as a vector subspace of ${\cal U}(\g g)$.
\end{cor}

\begin{cor}
If $\g h \subset \g g$ is a Lie subalgebra, then the inclusion
$\g h \hookrightarrow \g g$ induces
${\cal U}(\g h) \hookrightarrow {\cal U}(\g g)$.
Moreover, ${\cal U}(\g g)$ is free as left ${\cal U}(\g h)$-module.
\end{cor}

\begin{cor}
If $\g h_1, \g h_2 \subset \g g$ are Lie subalgebras such that
$\g = \g h_1 \oplus \g h_2$ as vector spaces
($\g h_1$ and $\g h_2$ need not be ideals), then the multiplication map
${\cal U}(\g h_1) \otimes {\cal U}(\g h_2) \to {\cal U}(\g g)$
is a vector space isomorphism.
\end{cor}

\begin{cor}
${\cal U}(\g g)$ has no zero divisors.
\end{cor}

\begin{ex}
Let
$E= \bigl(\begin{smallmatrix} 0 & 1 \\ 0 & 0 \end{smallmatrix}\bigr)
\in \mathfrak{sl}(2,\BB C)$,
then $E^n \ne 0$ in ${\cal U}(\mathfrak{sl}(2,\BB C))$ for all $n$.
\end{ex}

\begin{prop}
Assume that $\operatorname{char} \Bbbk = 0$, then the map
$S(\g g) \to {\cal U}(\g g)$ defined on monomials by
$$
\operatorname{Sym}(X_1 \dots X_n)
= \frac1{n!} \sum_{\sigma \in S_n} X_{\sigma(1)} \dots X_{\sigma(n)}
$$
is an isomorphism of $\g z(\g g)$-modules
($\g z(\g g)$ is the center of $\g g$).
\end{prop}

\noindent
\underline{Caution}: This is {\em not} an algebra isomorphism
unless $\g g$ is commutative.

\subsection{Geometric Realization of ${\cal U}(\g g)$}  \label{U(g)-geom}

We now assume that the field $\Bbbk$ is $\BB R$ or $\BB C$.
Just as the Lie algebra $\g g$ can be identified with left-invariant
vector fields on $G$, the universal enveloping algebra
${\cal U}(\g g)$ can be identified with
$$
{\cal D}^l = \{ \text{left-invariant differential operators on $G$} \}.
$$
Indeed, ${\cal D}^l$ is an associative algebra over $\BB R$ or $\BB C$ with
unit, it is filtered by the order of differential operators.
We have a map $\g g \to {\cal D}^l_1$ sending an element of Lie algebra into
the corresponding left-invariant vector field, which is a differential
operator of order 1.
By the universality of ${\cal U}(\g g)$, we get an algebra homomorphism
$$
j: {\cal U}(\g g) \to {\cal D}^l
$$
which respects the filtration: $j({\cal U}_n(\g g)) \subset {\cal D}^l_n$,
$n=0,1,2,3,\dots$.
Hence we get a map
$$
j: Gr{\cal U}(\g g) \to Gr{\cal D}^l.
$$
Note that for an operator $D \in {\cal D}^l_n$, its symbol
$\sigma_n(D)$ is left-invariant under the action of $G$ and hence uniquely
determined by the value at $e \in G$, so we get an injective map
$$
\sigma: Gr{\cal D}^l \hookrightarrow S(T_eG)=S(\g g)
$$
which is an algebra homomorphism.

Let us consider the following composition of algebra homomorphisms
$$
S(\g g) \to Gr{\cal U}(\g g) \xrightarrow{j}  Gr{\cal D}^l
\xrightarrow{\sigma} S(\g g),
$$
where the first map is the map whose existence is guaranteed by the
universal property of the associative commutative algebra $S(\g g)$.
This composition is the identity map, since it is the identity map for
elements of degree 1, i.e. elements in $\g g$, which generate the algebra.
Thus we obtain a commutative diagram:
$$
\begin{CD}
Gr{\cal U}(\g g)   @>{j}>>  Gr{\cal D}^l \\
@A{\text{onto}}AA  @VV{\sigma, \text{ injective}}V \\
S(\g g) @= S(\g g)
\end{CD}
$$
Each map in the diagram must be an isomorphism.
In particular, we have proved the Poincar\'e-Birkhoff-Witt Theorem
when the field $\Bbbk$ is $\BB R$ or $\BB C$.
We conclude:

\begin{prop}  \label{UDiso}
We have an isomorphism of filtered algebras ${\cal U}(\g g) \simeq {\cal D}^l$.
Under this isomorphism, the center of ${\cal U}(\g g)$ is identified with
the subalgebra of bi-invariant differential operators on $G$.
\end{prop}

\section{Irreducible Representations of $\mathfrak{sl}(2,\BB R)$}
\label{sl(2,R)-rep-section}

In this section we classify all irreducible representations of the Lie algebra
$\mathfrak{sl}(2,\BB R)$, these are the same as the irreducible representations
of $\mathfrak{sl}(2,\BB C)$.
%This is done in a way very similar to the classification of  all irreducible
%finite-dimensional representations of $\mathfrak{sl}(2,\BB C)$.
The results of this section will be used to classify all irreducible
representations of $SL(2,\BB R)$.

\subsection{Preliminaries}

Recall that a {\em representation} of a Lie algebra $\mathfrak g$ is a
complex vector space $V$ (possibly of infinite dimension) together with a map
$\pi: \mathfrak g \to \End(V)$ such that
$$
\pi([X,Y]) = \pi(X) \pi(Y) - \pi(Y) \pi(X), \qquad \forall X,Y \in \mathfrak g.
$$
Note that, unlike the Lie group representation, vector space $V$ is not
required to have any topology whatsoever.
If $\mathfrak g$ is a real Lie algebra, $V$ may be taken a real vector space,
but we prefer to work with complex vector spaces.
Representations of $\mathfrak g$ are often called {\em $\mathfrak g$-modules}.

When the Lie algebra is $\mathfrak{sl}(2,\BB R)$ or $\mathfrak{sl}(2,\BB C)$,
it is convenient to fix the following set of generators:
$$
H = \begin{pmatrix} 1 & 0 \\ 0 & -1 \end{pmatrix}, \qquad
E = \begin{pmatrix} 0 & 1 \\ 0 & 0 \end{pmatrix}, \qquad
F = \begin{pmatrix} 0 & 0 \\ 1 & 0 \end{pmatrix}.
$$
These generators satisfy the following relations:
$$
[H,E]=2E, \qquad [H,F]=-2F, \qquad [E,F]=H.
$$

We classify all irreducible representations $(\pi, V)$ of
$\mathfrak{sl}(2,\BB C)$ that satisfy the following additional assumption:
{\em There exists a $\lambda \in \BB C$ such that the $\lambda$-eigenspace
for $H$ in $V$ has dimension one, i.e.
$$
\pi(H) v_0 = \lambda v_0, \qquad \text{for some $v_0 \in V$, $v_0 \ne 0$,
and $\lambda \in \BB C$}
$$
and
$$
\forall v \in V, \quad \pi(H) v = \lambda v \quad \Longrightarrow \quad
\text{$v$ is a scalar multiple of $v_0$.}
%\dim_{\BB C} \ker (\pi(H)-\lambda) = 1.
$$}

We will see soon that this assumption forces $V$ to decompose into a direct
sum of eigenspaces for $H$ and each eigenspace to have dimension one.
The $\mathfrak{sl}(2,\BB C)$-modules associated to irreducible
representations of $SL(2,\BB R)$ and all irreducible finite-dimensional
$\mathfrak{sl}(2,\BB C)$-modules satisfy this assumption automatically.

\begin{lem}  \label{eigenvectors}
Let $v \in V$ be an eigenvector for $H$ with eigenvalue $\nu \in \BB C$,
then
\begin{enumerate}
\item
$\pi(E) v \in V$ is either zero or an eigenvector for $H$
with eigenvalue $\nu+2$;
\item
$\pi(F) v \in V$ is either zero or an eigenvector for $H$
with eigenvalue $\nu-2$.
\end{enumerate}
\end{lem}

\begin{proof}
We have:
\begin{multline*}
\pi(H) \bigl( \pi(E) v \bigr) = \pi([H,E]) v + \pi(E) \pi(H) v \\
= 2\pi(E) v + \nu \pi(E) v = (\nu+2) \pi(E) v;
\end{multline*}
\begin{multline*}
\pi(H) \bigl( \pi(F) v \bigr) = \pi([H,F]) v + \pi(F) \pi(H) v \\
= -2\pi(F) v + \nu \pi(F) v = (\nu-2) \pi(F) v.
\end{multline*}
\end{proof}

Define
$$
v_k = \bigl( \pi(E) \bigr)^k v_0, \qquad v_{-k} = \bigl( \pi(F) \bigr)^k v_0,
\qquad k=1,2,3,\dots.
$$
Then
$$
\pi(E) v_k = v_{k+1} \quad \text{if $k \ge 0$}, \qquad
\pi(F) v_k = v_{k-1} \quad \text{if $k \le 0$}.
$$
%\begin{align*}
%\pi(E) v_k &= v_{k+1} \quad \text{if $k \ge 0$}; \\
%\pi(F) v_k &= v_{k-1} \quad \text{if $k \le 0$}.
%\end{align*}

\begin{cor}
Each vector $v_k \in V$, $k\in \BB Z$, is either zero or an eigenvector for $H$
with eigenvalue $\lambda+2k$.
\end{cor}

Let us consider the {\em Casimir element}
$$
\Omega =_{\text{def}}  H^2+2EF+2FE \quad \in {\cal U}(\mathfrak{sl}(2,\BB C)).
$$
This element is distinguished by the property that it generates
${\cal ZU}(\mathfrak{sl}(2,\BB C))$
-- the center of the universal enveloping algebra of $\mathfrak{sl}(2,\BB C)$.
%We will discuss its properties later.

\begin{lem}
Let $v_0 \in V$ be an eigenvector of $H$ with eigenvalue $\lambda$ such that
$\dim_{\BB C} \ker (\pi(H)-\lambda) = 1$.
Then
$$
\pi(\Omega)v_0 =  \pi(H)^2v_0 + 2\pi(E)\pi(F)v_0+2\pi(F)\pi(E)v_0 = \mu v_0
$$
for some $\mu \in \BB C$.
\end{lem}

\begin{proof}
First of all, $\pi(H)^2v_0 = \lambda^2 v_0$.
By Lemma \ref{eigenvectors}, $\pi(F)v_0$ is either zero or an eigenvector
for $H$ with eigenvalue $\lambda-2$ and $\pi(E)\pi(F)v_0$ is either zero
or an eigenvector for $H$ with eigenvalue $\lambda$.
Then our assumption $\dim_{\BB C} \ker (\pi(H)-\lambda) = 1$ implies
$\pi(E)\pi(F)v_0$ is a scalar multiple of $v_0$.
Similarly, $\pi(F)\pi(E)v_0$ is a scalar multiple of $v_0$.
Therefore, $\pi(\Omega)v_0$ is a scalar multiple of $v_0$.
\end{proof}

\begin{cor}  \label{omega-mu}
Under the assumptions of the Lemma, \\
$\pi(\Omega) v_k = \mu v_k$ for all $k \in \BB Z$.
\end{cor}

\begin{proof}
This is a manifestation of the fact that $\Omega$ is a central element of
${\cal U}(\mathfrak{sl}(2,\BB C))$:
%Hence $\pi(\Omega) \pi(X) = \pi(X) \pi(\Omega)$ for all
%$X \in {\cal U}(\mathfrak{sl}(2,\BB C))$.
\begin{multline*}
\pi(\Omega) v_k = \pi(\Omega) \bigl( \pi(E) \bigr)^k v_0
= \bigl( \pi(E) \bigr)^k \pi(\Omega) v_0 \\
= \mu \bigl( \pi(E) \bigr)^k v_0 = \mu v_k, \qquad \text{if $k \ge 0$};
\end{multline*}
\begin{multline*}
\pi(\Omega) v_k = \pi(\Omega) \bigl( \pi(F) \bigr)^{-k} v_0
= \bigl( \pi(F) \bigr)^{-k} \pi(\Omega) v_0 \\
= \mu \bigl( \pi(F) \bigr)^{-k} v_0 = \mu v_k, \qquad \text{if $k \le 0$}.
\end{multline*}
\end{proof}

Next we would like to prove that
$$
V = \bigoplus_{k \in \BB Z} \BB C \cdot v_k.
$$
Since $V$ is irreducible, it is sufficient to prove that the vector space
$$
V_0 =_{\text{def}} \bigoplus_{k \in \BB Z} \BB C \cdot v_k
$$
is invariant under the action of $\mathfrak{sl}(2,\BB C)$.
Obviously, $V_0$ is invariant under $H$, and it would suffice to prove that
$\pi(F)v_k$ is proportional to $v_{k-1}$ for $k>0$ and that
$\pi(E)v_k$ is proportional to $v_{k+1}$ for $k<0$.

\begin{lem}  \label{classification-lemma}
Let $v_0 \in V$ be an eigenvector for $H$ with eigenvalue $\lambda$,
define $V_0 = \bigoplus_{k \in \BB Z} \BB C \cdot v_k$ and suppose that
$\pi(\Omega)v_0=\mu v_0$ for some $\mu \in \BB C$.
Then the set of nonzero $v_k$ forms a basis for $V_0$ with the following
relations:
$$
\pi(H) v_k = (\lambda +2k) v_k, \qquad k \in \BB Z,
$$
$$
\pi(E) v_k = v_{k+1} \quad \text{if $k \ge 0$}, \qquad
\pi(F) v_k = v_{k-1} \quad \text{if $k \le 0$},
$$
$$
\pi(E) v_k =
\frac14 \bigl( \mu - (\lambda+2k+2)^2 + 2(\lambda+2k+2) \bigr) v_{k+1}
%\frac14 \bigl( \mu - (\lambda+2k+2)((\lambda+2k+4) \bigr) v_{k+1}
\quad \text{if $k<0$},
$$
$$
\pi(F) v_k =
\frac14 \bigl( \mu - (\lambda+2k-2)^2 - 2(\lambda+2k-2) \bigr) v_{k-1}
%\frac14 \bigl( \mu - (\lambda+2k-2)((\lambda+2k-4) \bigr) v_{k-1}
\quad \text{if $k>0$}.
$$
In particular, $V_0$ is $\mathfrak{sl}(2,\BB C)$-invariant and all nontrivial
$H$ eigenspaces of $V_0$ are one-dimensional.
\end{lem}

\begin{proof}
By Corollary \ref{omega-mu}, $\pi(\Omega)v=\mu v$ for all $v \in V_0$.
We can rewrite $\Omega$ as
$$
\Omega = H^2+2EF+2FE = H^2+2H+4FE = H^2-2H+4EF.
$$
To simplify notations, let us suppose that $u_{\nu} \in V_0$ is such that
$\pi(H)u_{\nu}=\nu u_{\nu}$ for some $\nu \in \BB C$. Of course,
we will let $u_{\nu}=v_k$ with $\nu=\lambda+2k$, $k \in \BB Z$. Then
$$
\pi(E)\pi(F)u_{\nu} = \frac14 \bigl(\pi(\Omega) -\pi(H)^2 + 2\pi(H) \bigr) u_{\nu}
= \frac14 (\mu - \nu^2 +2\nu) u_{\nu},
$$
$$
\pi(F)\pi(E)u_{\nu} = \frac14 \bigl(\pi(\Omega) -\pi(H)^2 - 2\pi(H) \bigr) u_{\nu}
= \frac14 (\mu - \nu^2 -2\nu) u_{\nu}.
$$
If $k<0$, letting $u_{\nu}=v_{k+1}$ with $\nu=\lambda+2k+2$, we get
\begin{multline*}
\pi(E) v_k = \pi(E) \pi(F) v_{k+1} \\
= \frac14 \bigl( \mu - (\lambda+2k+2)^2 + 2(\lambda+2k+2) \bigr) v_{k+1}.
\end{multline*}
If $k>0$, letting $u_{\nu}=v_{k-1}$ with $\nu=\lambda+2k-2$, we get
\begin{multline*}
\pi(F) v_k = \pi(F) \pi(E) v_{k-1} \\
= \frac14 \bigl( \mu - (\lambda+2k-2)^2 - 2(\lambda+2k-2) \bigr) v_{k-1}.
\end{multline*}
\end{proof}

\begin{cor}  \label{equiv-ond-cor}
Let $V$ be an irreducible $\mathfrak{sl}(2,\BB C)$-module.
Then the following conditions are equivalent:
\begin{enumerate}
\item
There exists a $\lambda \in \BB C$ such that
$\dim_{\BB C} \ker (\pi(H)-\lambda) = 1$;
\item
There exist a non-zero $v_0 \in V$ that is simultaneously an eigenvector
for $\pi(H)$ and $\pi(\Omega)$, i.e.
$$
\pi(H)v_0 = \lambda v_0 \quad \text{and} \quad \pi(\Omega) v_0 = \mu v_0
\quad \text{for some $\lambda,\mu \in \BB C$.}
$$
\end{enumerate}
\end{cor}

(Some authors start with one condition and some with the other.
Now we know that they are equivalent.)
In the context of representation theory, it is customary to call the
eigenvalues of $\pi(H)$ {\em weights} of $V$ and the eigenspaces of $\pi(H)$
{\em weight spaces} of $V$.
We summarize the results of this subsection as follows:

\begin{prop}  \label{classification-prop}
Let $V$ be an irreducible $\mathfrak{sl}(2,\BB C)$-module with such that
one of the two equivalent conditions from Corollary \ref{equiv-ond-cor}
is satisfied. Then $V$ is a direct sum of weight spaces,
all weight spaces of $V$ are one-dimensional, and the weights of $V$ are
of the form $\lambda+2k$ with $k$ ranging over an ``interval of integers''
$$
\BB Z \cap [a,b], \qquad \BB Z \cap [a,\infty), \qquad \BB Z \cap (-\infty, b]
\qquad \text{or} \qquad \BB Z.
$$
Moreover, there exists a $\mu \in \BB C$ such that $\pi(\Omega)v=\mu v$
for all $v \in V$.
\end{prop}

\subsection{Classification of Irreducible $\mathfrak{sl}(2,\BB C)$-Modules}

First, we consider the case when the irreducible
$\mathfrak{sl}(2,\BB C)$-module $V$ is finite-dimensional.

\begin{prop}  \label{finite-dim}
Let $V$ be an irreducible $\mathfrak{sl}(2,\BB C)$-module of dimension $d+1$.
Then $V$ has a basis $\{v_0,v_1,\dots,v_d\}$ such that
\begin{align*}
\pi(H) v_k &= (-d+2k) v_k, \quad 0 \le k \le d, \\
\pi(E) v_k &= v_{k+1}, \quad 0 \le k < d, \qquad \pi(E) v_d =0, \\
\pi(F) v_k &= k(d+1-k) v_{k-1}, \quad 0 < k \le d, \qquad \pi(F) v_0 =0.
\end{align*}
Moreover,
$$
\pi(\Omega) v = d(d+2) v, \qquad \forall v \in V,
$$
and $V$ is determined up to isomorphism by its dimension.
\end{prop}

We denote the irreducible $\mathfrak{sl}(2,\BB C)$-module of dimension $d+1$
by $F_d$.

\begin{proof}
The operator $\pi(\Omega)$ commutes with all $\pi(X)$,
$X \in \mathfrak{sl}(2,\BB C)$. Hence, by Schur's Lemma, there exists
a $\mu \in \BB C$ such that $\pi(\Omega) v = \mu v$ for all $v \in V$.
(This is one of the places where we use the finite-dimensionality of $V$.)
Since $V$ is finite-dimensional, $\pi(H)$ has at least one eigenvalue.
Let $\lambda \in \BB C$ be an eigenvalue for $\pi(H)$ with the least
real part and $v_0 \in V$ a corresponding eigenvector.
By Lemma \ref{eigenvectors}, $\pi(F)v_0=0$.
By Lemma \ref{classification-lemma}, the set of nonzero
$v_k= \bigl( \pi(E) \bigr)^k v_0$, $k=0,1,2,3,\dots$, form a basis for $V$.
Since $\dim V = d+1$, $\{v_0,\dots,v_d\}$ is a basis such that
$$
\pi(E) v_k = v_{k+1}, \quad 0 \le k < d, \qquad \pi(E) v_d =0
$$
and
$$
\pi(H) v_k = (\lambda+2k) v_k, \quad 0 \le k \le d.
$$
Since $\pi(F)v_0=0$,
\begin{multline*}
\mu v_0 = \pi(\Omega)v_0 = \bigl( \pi(H)^2 -2\pi(H)+4\pi(E)\pi(F) \bigr) v_0 \\
= (\lambda^2-2\lambda)v_0 = \lambda(\lambda-2)v_0
\end{multline*}
and $\mu=\lambda(\lambda-2)$.
Similarly, from $\pi(E)v_d=0$ we obtain
\begin{multline*}
\mu v_d = \pi(\Omega)v_d = \bigl( \pi(H)^2 +2\pi(H)+4\pi(F)\pi(E) \bigr) v_d \\
= ((\lambda+2d)^2+2(\lambda+2d))v_d = (\lambda+2d)(\lambda+2d+2)v_d
\end{multline*}
and $\mu=(\lambda+2d)(\lambda+2d+2)$.
Solving
$$
\lambda(\lambda-2) = (\lambda+2d)(\lambda+2d+2)
$$
we obtain
$$
\lambda = -d \qquad \text{and} \qquad \mu = \lambda(\lambda-2) = d(d+2).
$$
The property
$$
\pi(F) v_k = k(d+1-k) v_{k-1}, \qquad 0 < k \le d,
$$
follows from Lemma \ref{classification-lemma}:
\begin{multline*}
\frac14 \bigl( \mu - (\lambda+2k-2)^2 - 2(\lambda+2k-2) \bigr)  \\
= \frac14 \bigl( d(d+2) - (-d+2k-2)^2 - 2(-d+2k-2) \bigr)  \\
= \frac14 \bigl( d^2+2d - d^2 - (2k-2)^2 + 4dk -4d +2d - 2(2k-2) \bigr)  \\
= -(k-1)^2 + dk -k+1
%= -k^2+2k-1 +dk-k+1
= k(d+1-k).
\end{multline*}
This finishes our proof of the theorem.
\end{proof}

\begin{rem}
It is also true that any indecomposable finite-dimensional
$\mathfrak{sl}(2,\BB C)$-module is irreducible.
Hence every finite-dimensional $\mathfrak{sl}(2,\BB C)$-module
is a direct sum of irreducible submodules.
See the book \cite{HT} for details.
\end{rem}

\begin{ex}  \label{F_d-construction}
The irreducible finite-dimensional representations of $SL(2,\BB C)$ and
$\mathfrak{sl}(2,\BB C)$ can be constructed as follows.
Let $SL(2,\BB C)$ act on $\BB C^2$ by matrix multiplication:
$$
\begin{pmatrix} a & b \\ c & d \end{pmatrix}
\begin{pmatrix} z_1 \\ z_2 \end{pmatrix}
= \begin{pmatrix} az_1+bz_2 \\ cz_1+dz_2 \end{pmatrix}, \:
\begin{pmatrix} a & b \\ c & d \end{pmatrix} \in SL(2,\BB C), \:
\begin{pmatrix} z_1 \\ z_2 \end{pmatrix} \in \BB C^2.
$$
Let
$$
V = \{ \text{polynomial functions on $\BB C^2$} \},
$$
$$
V_d = \{ \text{homogeneous polynomials on $\BB C^2$ of degree $d$} \},
$$
and define a representation $\pi$ of $SL(2,\BB C)$ in $V$ by
$$
\bigl( \pi(g) f \bigr)(z) = f(g^{-1} \cdot z),
\quad g \in SL(2,\BB C), \: f \in V, \: z \in \BB C^2.
$$
It is easy to see that the subspaces $V_d$ remain invariant under this action.
Thus
$$
(\pi,V) = \bigoplus_{d \ge 0} (\pi_d,V_d),
$$
where $\pi_d$ denotes the restriction of $\pi$ to $V_d$. Differentiating,
one obtains representations $(\pi_d,V_d)$ of $\mathfrak{sl}(2,\BB C)$.
Then each $\mathfrak{sl}(2,\BB C)$-module $V_d$ has dimension $d+1$,
is irreducible, and hence isomorphic to $F_d$ (homework).
\end{ex}

\begin{rem}
One can extend this construction to any closed subgroup of
$G \subset GL(n,\BB C)$.
However, the resulting representations are not necessarily irreducible.
If $G$ is a complex analytic subgroup of $GL(n,\BB C)$ and the Lie algebra
of $G$ is simple, each irreducible representation of $G$
appears as a subrepresentation of the space of homogeneous polynomials on
$\BB C^n$ of degree $d$, for some $d$.
\end{rem}

Now we look at the irreducible infinite-dimensional
$\mathfrak{sl}(2,\BB C)$-modules. By Proposition \ref{classification-prop},
the weights of $V$ are of the form $\lambda+2k$ with $k$ ranging over an
infinite ``interval of integers''
$$
\BB Z \cap [a,\infty), \qquad \BB Z \cap (-\infty, b]
\qquad \text{or} \qquad \BB Z.
$$
In the first case we call $V$ a {\em lowest weight module} and in the second
-- a {\em highest weight module}.
The irreducible modules of the last type do not have a universally accepted
name, but one could call them {\em irreducible principal series modules}.
First we classify the irreducible lowest weight modules.

\begin{prop}
Let $V$ be an irreducible infinite-dimensional lowest weight
$\mathfrak{sl}(2,\BB C)$-module.
Then $V$ has a basis of $H$-eigenvectors $\{v_0,v_1,v_2,\dots \}$
and a $\lambda \in \BB C$, $\lambda \ne 0, -1, -2, \dots$, such that
\begin{align*}
\pi(H) v_k &= (\lambda+2k) v_k, \quad k \ge 0, \\
\pi(E) v_k &= v_{k+1}, \quad k \ge 0, \\
\pi(F) v_k &= -k(\lambda+k-1) v_{k-1}, \quad k > 0 , \qquad \pi(F) v_0 =0.
\end{align*}
Moreover,
$$
\pi(\Omega) v = \lambda(\lambda-2) v, \qquad \forall v \in V,
$$
and $V$ is determined up to isomorphism by its lowest weight $\lambda$.
\end{prop}

We denote the irreducible $\mathfrak{sl}(2,\BB C)$-module of lowest weight
$\lambda$ by $V_{\lambda}$.

\begin{proof}
Let $\lambda \in \BB C$ be an eigenvalue for $\pi(H)$ with the least
real part and $v_0 \in V$ a corresponding eigenvector.
By Lemma \ref{eigenvectors}, $\pi(F)v_0=0$.
By Lemma \ref{classification-lemma}, the set of nonzero
$v_k= \bigl( \pi(E) \bigr)^k v_0$, $k=0,1,2,3,\dots$, form a basis for $V$.
By assumption, $V$ is infinite-dimensional, hence we get a basis
$\{v_0,v_1,v_2,\dots\}$ such that
$$
\pi(E) v_k = v_{k+1}, \quad k \ge 0, \qquad \pi(F)v_0=0
$$
and
$$
\pi(H) v_k = (\lambda+2k) v_k, \quad k \ge 0.
$$
Since $\pi(F)v_0=0$,
\begin{multline*}
\pi(\Omega)v_0 = \bigl( \pi(H)^2 -2\pi(H)+4\pi(E)\pi(F) \bigr) v_0 \\
= (\lambda^2-2\lambda)v_0 = \lambda(\lambda-2)v_0
\end{multline*}
and it follows that $\pi(\Omega) v = \lambda(\lambda-2) v$, for all $v \in V$.
The property
$$
\pi(F) v_k = -k(\lambda+k-1) v_{k-1}, \qquad k > 0,
$$
follows from Lemma \ref{classification-lemma}:
\begin{multline*}
\frac14 \bigl( \mu - (\lambda+2k-2)^2 - 2(\lambda+2k-2) \bigr)  \\
= \frac14
\bigl( \lambda(\lambda-2) - (\lambda+2k-2)^2 - 2(\lambda+2k-2) \bigr)  \\
= \frac14 \bigl( \lambda^2 - 2\lambda - \lambda^2 - (2k-2)^2 - 4k\lambda
+ 4\lambda - 2\lambda - 2(2k-2) \bigr)  \\
= -(k-1)^2 - k\lambda -k+1
= -k(\lambda+k-1).
\end{multline*}

It remains to show $\lambda \ne 0, -1, -2, \dots$.
If $\lambda$ is an integer and $\lambda \le 0$, consider $k=-\lambda+1$, then
$$
\pi(F) v_k = -k(\lambda+k-1) v_{k-1} =0,
$$
so $\{v_{-\lambda+1}, v_{-\lambda+2}, v_{-\lambda+3}, \dots \}$ form a basis
for a proper $\mathfrak{sl}(2,\BB C)$-submodule of $V$
(isomorphic to $V_{-\lambda+2}$). This contradicts $V$ being irreducible.
\end{proof}

The classification of irreducible highest weight modules is similar.

\begin{prop}
Let $V$ be an irreducible infinite-dimensional highest weight
$\mathfrak{sl}(2,\BB C)$-module.
Then $V$ has a basis of $H$-eigenvectors $\{\bar v_0,\bar v_1,\bar v_2,\dots \}$
and a $\lambda \in \BB C$, $\lambda \ne 0, 1, 2, \dots$, such that
\begin{align*}
\pi(H) \bar v_k &= (\lambda-2k) \bar v_k, \quad k \ge 0, \\
\pi(F) \bar v_k &= \bar v_{k+1}, \quad k \ge 0, \\
\pi(E) \bar v_k &= k(\lambda-k+1) \bar v_{k-1}, \quad k > 0, \qquad
\pi(E) \bar v_0 =0.
\end{align*}
Moreover,
$$
\pi(\Omega) v = \lambda(\lambda+2) v, \qquad \forall v \in V,
$$
and $V$ is determined up to isomorphism by its highest weight $\lambda$.
\end{prop}

We denote the irreducible infinite-dimensional $\mathfrak{sl}(2,\BB C)$-module
of highest weight $\lambda$ by $\bar V_{\lambda}$.
Finally, we turn our attention to what we call
the irreducible principal series modules.

\begin{prop}  \label{P-class}
Let $V$ be an irreducible infinite-dimensional $\mathfrak{sl}(2,\BB C)$-module
that satisfies one of the two equivalent conditions from
Corollary \ref{equiv-ond-cor} and is not a highest nor lowest weight module.
Let $v_0 \in V$ be an eigenvector for $H$ with eigenvalue $\lambda$,
then there is a $\mu \in \BB C$ such that
$$
\pi(\Omega) v = \mu v, \qquad \forall v \in V,
$$
and $V$ has a basis of eigenvectors
$\{ \dots, v_{-2},v_{-1},v_0,v_1,v_2,\dots \}$ such that
$$
\pi(H) v_k = (\lambda +2k) v_k, \qquad k \in \BB Z,
$$
$$
\pi(E) v_k = v_{k+1} \quad \text{if $k \ge 0$}, \qquad
\pi(F) v_k = v_{k-1} \quad \text{if $k \le 0$},
$$
$$
\pi(E) v_k =
\frac14 \bigl( \mu - (\lambda+2k+1)^2 + 1 \bigr) v_{k+1}
\quad \text{if $k<0$},
$$
$$
\pi(F) v_k =
\frac14 \bigl( \mu - (\lambda+2k-1)^2 +1 \bigr) v_{k-1}
\quad \text{if $k>0$}.
$$
The constants $\lambda, \mu \in \BB C$ are subject to the constraint
\begin{equation}  \label{constraint}
\lambda \pm \sqrt{\mu+1} \ne \text{odd integer}.
\end{equation}

We denote such an irreducible infinite-dimensional
$\mathfrak{sl}(2,\BB C)$-module by $P(\lambda, \mu)$.
Two $\mathfrak{sl}(2,\BB C)$-modules $P(\lambda, \mu)$ and $P(\lambda', \mu')$
are isomorphic if and only if $\mu'=\mu$ and $\lambda'=\lambda+2k$,
for some $k \in \BB Z$.
\end{prop}

\begin{proof}
The statements about the basis and $\mathfrak{sl}(2,\BB C)$-action follow
from Lemma \ref{classification-lemma}.
It is easy to see that $V$ is irreducible if and only if
$$
\pi(E) v_k \ne 0 \quad \forall k \in \BB Z
\qquad \text{and} \qquad \pi(F) v_k \ne 0 \quad \forall k \in \BB Z,
$$
hence
$$
\mu+1 \ne (\lambda+ \text{odd integer})^2
$$
and the constraint (\ref{constraint}) follows.

If two $\mathfrak{sl}(2,\BB C)$-modules $P(\lambda, \mu)$ and
$P(\lambda', \mu')$ are isomorphic, then $\mu'=\mu$ and $\lambda'=\lambda+2k$,
which follows by examinations of the weights.
Conversely, if $\mu'=\mu$ and $\lambda'=\lambda+2k$,
the module $P(\lambda', \mu')$ has $\lambda$ as a weight, so let
$v_0 \in P(\lambda', \mu')$ be an eigenvector for $H$ with eigenvalue $\lambda$,
then Lemma \ref{classification-lemma} implies
$P(\lambda, \mu) \simeq P(\lambda', \mu')$.
\end{proof}

We summarize the results of this section as follows:

\begin{thm}  \label{classification-thm}
There are precisely four types of irreducible $\mathfrak{sl}(2,\BB C)$-modules
such that one of the two equivalent conditions from
Corollary \ref{equiv-ond-cor} is satisfied:
\begin{itemize}
\item
Finite-dimensional modules $F_d$ of dimension $d+1$ with weights
$$
\begin{matrix}
[\circ & & \circ & & \dots & & \circ & & \circ] \\
-d & & -d+2 & & \dots & & d-2 & & d
\end{matrix}
$$
$\pi(\Omega)v=d(d+2)v$ for all $v \in F_d$.
\item
Lowest weight modules $V_{\lambda}$ of lowest weight $\lambda \in \BB C$,
$\lambda \ne 0,-1,-2,\dots$, with weights
$$
\begin{matrix}
[\circ & & \circ & & \circ & & \dots  \\
\lambda & & \lambda+2 & & \lambda+4 & & \dots
\end{matrix}
$$
$\pi(\Omega)v=\lambda(\lambda-2)v$ for all $v \in V_{\lambda}$.
\item
Highest weight modules $\bar V_{\lambda}$ of highest weight $\lambda \in \BB C$,
$\lambda \ne 0,1,2,\dots$, with weights
$$
\begin{matrix}
\dots & & \circ & & \circ & & \circ] \\
\dots & & \lambda-4 & & \lambda-2 & & \lambda
\end{matrix}
$$
$\pi(\Omega)v=\lambda(\lambda+2)v$ for all $v \in \bar V_{\lambda}$.
\item
Irreducible principal series modules $P(\lambda,\mu)$ with
$\lambda,\mu \in \BB C$ subject to the constraint
$$
\lambda \pm \sqrt{\mu+1} \ne \text{odd integer}
$$
and weights
$$
\begin{matrix}
\dots & & \circ & & \circ & & \circ & & \dots \\
\dots & & \lambda-2 & & \lambda & & \lambda+2 & & \dots
\end{matrix}
$$
$\pi(\Omega)v=\mu v$ for all $v \in P(\lambda,\mu)$.
\end{itemize}
\end{thm}

\begin{proof}
We already know that any irreducible $\mathfrak{sl}(2,\BB C)$-module has
to be of the type $F_d$, $V_{\lambda}$, $\bar V_{\lambda}$ or $P(\lambda,\mu)$.
It remains to prove the existence part. I.e., for example, that for any
choice of the parameters $\lambda,\mu \in \BB C$ such that
$$
\lambda \pm \sqrt{\mu+1} \ne \text{odd integer},
$$
there {\em exists} an irreducible $\mathfrak{sl}(2,\BB C)$-module of type
$P(\lambda,\mu)$. To do that, let
$$
V = \bigoplus_{k \in \BB Z} \BB C \cdot v_k
$$
and {\em define} the action of $\mathfrak{sl}(2,\BB C)$ by
$$
\pi(H) v_k = (\lambda +2k) v_k, \qquad k \in \BB Z,
$$
$$
\pi(E) v_k = v_{k+1} \quad \text{if $k \ge 0$}, \qquad
\pi(F) v_k = v_{k-1} \quad \text{if $k \le 0$},
$$
$$
\pi(E) v_k =
\frac14 \bigl( \mu - (\lambda+2k+1)^2 + 1 \bigr) v_{k+1}
\quad \text{if $k<0$},
$$
$$
\pi(F) v_k =
\frac14 \bigl( \mu - (\lambda+2k-1)^2 +1 \bigr) v_{k-1}
\quad \text{if $k>0$}.
$$
Then check
$$
[\pi(H),\pi(E)]=2\pi(E), \qquad [\pi(H),\pi(F)]=-2\pi(F),
$$
$$
[\pi(E),\pi(F)]=\pi(H),
$$
which proves that the action of $\mathfrak{sl}(2,\BB C)$ on $V$ is well-defined.
Finally, to show that $(\pi,V)$ is irreducible, note that any non-zero
$\mathfrak{sl}(2,\BB C)$-invariant subspace $W \subset V$ must contain one
of the vectors $v_k$ (homework).
Then the constraint (\ref{constraint}) ensures that this vector $v_k$ generates
all of $V$, hence $W=V$ and $V$ is irreducible.
\end{proof}

\begin{rem}
There are many infinite-dimensional $\mathfrak{sl}(2,\BB C)$-modules that are
indecomposable, but not irreducible.
The book \cite{HT} describes them in detail.
We will see examples of those when we discuss the principal series of
$SL(2,\BB R)$.

Of these irreducible $\mathfrak{sl}(2,\BB C)$-modules, only the
finite-{\linebreak}dimensional ones ``lift'' to representations of
$SL(2,\BB C)$.
Precisely those modules with $\lambda \in \BB Z$ occur as underlying
$\mathfrak{sl}(2,\BB C)$-modules of irreducible representations
of $SU(1,1)$.
\end{rem}

\section{The Complexification, the Cartan Decomposition and
Maximal Compact Subgroups}

In this section we state without proofs the properties of maximal compact
subgroups. For more details and proofs see the books \cite{Kn2} and \cite{He}.

\subsection{Assumptions on Groups}

\begin{df}
A Lie algebra $\g g$ is {\em simple} if it has no proper ideals
and $\dim \g g >1$. A Lie algebra $\g g$ is {\em semisimple}
if it can be written as a direct sum of simple ideals $\g g_i$,
$$
\g g = {\bigoplus}_{1\leq i\leq N}\ \g g_i\,.
$$
One calls a Lie algebra $\g g$ {\em reductive} if it can be written
as a direct sum of ideals
$$
\g g = \g s \oplus \g z\,,
$$
with $\g s$ semisimple and $\g z=$ center of $\g g$.
A Lie group is {\em simple}, respectively {\em semisimple},
if it has finitely many connected components and if its Lie algebra is simple,
respectively semisimple.
A {\em closed linear group} is a closed Lie subgroup $G \subset GL(n,\BB R)$
or $G \subset GL(n, \BB C)$.
\end{df}

In this course we always assume that our group is a
{\em connected linear real semisimple Lie group} and denote such a group by
$G_{\BB R}$. Thus $G_{\BB R}$ denotes a closed connected Lie subgroup of some
$GL(n,\BB R)$ with semisimple Lie algebra.
We skip the definition of a reductive Lie group\footnote{The definition
of a reductive Lie group can be found in, for example, \cite{Kn2}.}
-- it is technical and requires more than just the Lie algebra being reductive.
Most results that are true for linear connected semisimple real
Lie groups also extend to linear semisimple real Lie groups that are not
necessarily connected as well as linear reductive real Lie groups,
but the proofs may become more involved.

\begin{ex}
The Lie groups $SL(n, \BB R)$, $SL(n, \BB C)$ are simple;
the Lie groups $GL(n, \BB R)$, $GL(n, \BB C)$ are reductive.
Any compact real Lie group is a closed linear group, as can be deduced from
the Peter-Weyl Theorem (requires some effort), and is moreover reductive.
Any compact real Lie group with discrete center is closed linear semisimple.
On the other hand, the universal covering
$\widetilde{SL(n, \BB R)}$ of $SL(n,\BB R)$, $n \ge 2$, is not a linear group
(see Example \ref{not-linear}).
\end{ex}

We conclude this subsection with a list of equivalent characterizations
of semisimple Lie algebras.

\begin{prop}
The following conditions on a Lie algebra $\g g$ are equivalent to
$\g g$ being semisimple:
\begin{enumerate}
\item[\rm{(a)}]
$\g g$ can be written as a direct sum of simple ideals $\g g_i$,
$$
\g g = {\bigoplus}_{1\leq i\leq N}\ \g g_i\,;
$$
\item[\rm{(b)}]
$\g g$ contains no nonzero solvable ideals;
\item[\rm{(c)}]
The radical of $\g g$, $\operatorname{rad}(\g g)$, is zero;
\item[\rm{(d)}]
The Killing form of $\g g$
$$
K(X,Y) = \tr (\operatorname{ad}X \operatorname{ad}Y), \qquad X,Y \in \g g,
$$
is non-degenerate.
\end{enumerate}
\end{prop}

\subsection{Complexifications and Real Forms}

Let $G_{\BB R}$ be a closed linear connected semisimple real Lie group.
Like any linear Lie group, $G_{\BB R}$ has a {\em complexification}
-- a complex Lie group $G$, with Lie algebra
$$
\g g \ =_{\text{def}}\ \BB C {\otimes}_{\BB R}\, \g g_{\BB R}\,,
$$
containing $G_{\BB R}$ as a Lie subgroup, connected and such that
the inclusion $G_{\BB R} \hookrightarrow G$ induces
$\g g_{\BB R} \hookrightarrow \g g$, $X \mapsto 1\otimes X$.

To construct a complexification, one regards $G_{\BB R}$ as subgroup of
$GL(n,\BB R)$, so that $\g g_{\BB R} \subset \g {gl}(n,\BB R)$.
That makes $\g g$ a Lie subalgebra of
$\g {gl}(n,\BB C) = \BB C \otimes_{\BB R}\, \g {gl}(n,\BB R)$.
Then let $G$ be the connected Lie subgroup of $GL(n,\BB C)$ with Lie algebra
$\g g$. Since $G_{\BB R}$ is connected, it is contained in $G$.
When $G$ is a complexification of $G_{\BB R}$, one calls $G_{\BB R}$ a
{\em real form} of $G$

\begin{ex}
The groups $SL(2,\BB R) \subset GL(2,\BB R)$,
$SU(1,1) \subset GL(2,\BB C) \subset GL(4,\BB R)$ and
$SU(2) \subset GL(2,\BB C) \subset GL(4,\BB R)$
have complexifications isomorphic to $SL(2,\BB C)$.
\end{ex}

%Note that if $G_{\BB R}$ is presented as a linear group
%$G_{\BB R}\subset GL(m,\BB C)$, one uses the usual inclusion
%$GL(m,\BB C)\hookrightarrow GL(2m,\BB R)$ to exhibit $G_{\BB R}$
%as subgroup of $GL(n,\BB R)$, with $n=2m$.

In general, the complexification of a linear Lie group depends on the
realization of that group as a linear group.

\begin{ex}
The group of real numbers under addition $(\BB R, +)$ can be realized as a
closed subgroup of $GL(2,\BB R)$ in at least two different ways:
$$
\left\{ \bigl(\begin{smallmatrix} e^t & 0 \\ 0 & e^{-t} \end{smallmatrix}\bigr)
;\: t \in \BB R \right\}
\qquad \text{and} \qquad
\left\{ \bigl(\begin{smallmatrix} 1 & t \\ 0 & 1 \end{smallmatrix}\bigr)
;\: t \in \BB R \right\}.
$$
These realizations lead to complexifications
$$
\left\{ \bigl(\begin{smallmatrix} z & 0 \\ 0 & z^{-1} \end{smallmatrix}\bigr)
;\: z \in \BB C \right\}
\qquad \text{and} \qquad
\left\{ \bigl(\begin{smallmatrix} 1 & z \\ 0 & 1 \end{smallmatrix}\bigr)
;\: z \in \BB C \right\}
$$
respectively.
The first complexification is isomorphic to $(\BB C^{\times}, \cdot)$
and the second to $(\BB C,+)$.
\end{ex}

In our situation, $G_{\BB R}$ is closed connected linear semisimple, and
the complexification $G$ is determined by $G_{\BB R}$ up to isomorphism, but
the {\em embedding} $G \subset GL(n,\BB C)$ still depends on the realization
$G_{\BB R} \subset GL(n,\BB R)$ of $G_{\BB R}$ as linear group.

\begin{thm}
Let $G$ be a connected complex semisimple Lie group. Then $G$ has a
{\em compact real form} $U_{\BB R}$, i.e. a compact connected real subgroup
$U_{\BB R} \subset G$ such that $G$ is a complexification of $U_{\BB R}$.
Moreover, any two compact real forms $U_{\BB R}$ and $U'_{\BB R}$ are conjugate
to each other by an element of $G$, i.e. there exists an element $g \in G$
such that $U'_{\BB R} = g \cdot U_{\BB R} \cdot g^{-1}$.
\end{thm}

\begin{ex}
A compact real form of $SL(n,\BB C)$ is $SU(n)$.
\end{ex}

\subsection{The Cartan Decomposition}

Suppose that $\g g = \BB C \otimes \g g_{\BB R}$, define a complex conjugation
$\sigma$ of $\g g$ with respect to $\g g_{\BB R}$ by
$$
\sigma(X+iY) = X-iY, \qquad X,Y \in \g g_{\BB R}.
$$
Then
$$
\sigma(\lambda Z) = \bar\lambda \sigma(Z),
\qquad \forall \lambda \in \BB C,\: Z \in \g g,
$$
(homework). For example, let $\g g = \g{sl}(n,\BB C)$ and
$\g g_{\BB R}= \g{sl}(n,\BB R)$, then
$$
\sigma \begin{pmatrix} a_{11} & \ldots & a_{1n} \\ \vdots & \ddots & \vdots \\
a_{n1} & \ldots & a_{nn} \end{pmatrix}
= \begin{pmatrix} \bar a_{11} & \ldots & \bar a_{1n} \\
\vdots & \ddots & \vdots \\ \bar a_{n1} & \ldots & \bar a_{nn} \end{pmatrix}
$$
is entry-by-entry conjugation. On the other hand, if
$\g g_{\BB R}= \g{su}(n) \subset \g{sl}(n,\BB C)$,
$$
\sigma(X) = - X^*, \qquad \forall X \in \g{sl}(n,\BB C),
$$
(negative conjugate transpose). In particular, the conjugation $\sigma$
depends on the choice of a real form $\g g_{\BB R} \subset \g g$.

We start with a closed connected linear semisimple real Lie group $G_{\BB R}$,
take its complexification $G$ and select a compact real form
$U_{\BB R} \subset G$:
$$
G_{\BB R} \quad \subset \quad G \quad \supset \quad U_{\BB R}.
$$
Let $\sigma, \tau : \g g \to \g g$ denote the complex conjugations with respect
to $\g g_{\BB R}$ and $\g u_{\BB R}$.

\begin{prop}
By replacing the compact real form $U_{\BB R} \subset G$ by an appropriate
conjugate, one can arrange that $\tau\sigma=\sigma\tau$.
\end{prop}

From now on we {\em fix} a $U_{\BB R} \subset G$ such that
$\tau\sigma=\sigma\tau$.
{\em Define}
$$
\theta: \g g \to \g g, \quad \theta = \tau\sigma=\sigma\tau,
$$
this is a complex linear Lie algebra automorphism of $\g g$ (homework).
Since $\theta^2 = \tau\sigma\tau\sigma = \tau^2\sigma^2 = Id_{\g g}$,
$\theta$ is an involution, called {\em Cartan involution}.

Since $\theta$ is an involution, its eigenvalues are $\pm 1$. Let
\begin{align*}
\g k &= \text{$+1$ eigenspace of $\theta: \g g \to \g g$,} \\
\g p &= \text{$-1$ eigenspace of $\theta: \g g \to \g g$.}
\end{align*}
Then
$$
\g g = \g k \oplus \g p
$$
as complex vector spaces. Now let
$$
\g k_{\BB R} = \g g_{\BB R} \cap \g k, \qquad \g p_{\BB R} = \g g_{\BB R} \cap \g p,
$$
then we have a direct sum decomposition of real vector spaces
$$
\g g_{\BB R} = \g k_{\BB R} \oplus \g p_{\BB R}
$$
called {\em Cartan decomposition} of $\g g_{\BB R}$ (note that we would not have
such a decomposition if $\tau$ and $\sigma$ did not commute).

\begin{lem}
\begin{enumerate}
\item
$\sigma = \tau$ on $\g k$ and $\sigma = -\tau$ on $\g p$;
\item
$[\g k, \g k] \subset \g k$, $[\g p, \g p] \subset \g k$, and
$[\g k, \g p] \subset \g p$;
\item
$\g u_{\BB R} = \g k_{\BB R} \oplus i \g p_{\BB R}$
(direct sum of real vector spaces).
\end{enumerate}
\end{lem}

\begin{proof}
\begin{enumerate}
\item
This part follows from $\theta=\tau\sigma$ and $\theta=+Id$ on $\g k$,
$\theta=-Id$ on $\g p$.

\item
This part follows from the fact that the Cartan involution $\theta$
is a Lie algebra isomorphism: $[\theta(X),\theta(Y)]=\theta([X,Y])$
for all $X,Y \in \g g$.
In particular, if $X$ belongs to the $\epsilon_1$-eigenspace of $\theta$
and $Y$ belongs to the $\epsilon_2$-eigenspace, where
$\epsilon_1,\epsilon_2=\pm1$, then $[X,Y]$ belongs to the
$(\epsilon_1\epsilon_2)$-eigenspace.

\item
From the definition of $\tau$,
$$
\g u_{\BB R} = \{ X \in \g g ;\: \tau(X)=X \}.
$$
It follows from part 1 that $\g k_{\BB R}$ and $ i \g p_{\BB R}$ are subspaces of
$\g u_{\BB R}$. Finally, $\g k_{\BB R} \cap i \g p_{\BB R} =\{0\}$ and
$\dim (\g k_{\BB R}) + \dim (i \g p_{\BB R}) = \dim(\g g_{\BB R}) = \dim(\g u_{\BB R})$
imply $\g u_{\BB R} = \g k_{\BB R} \oplus i \g p_{\BB R}$.
\end{enumerate}
\end{proof}

\begin{cor}
$\g k_{\BB R}$ is a real Lie subalgebra of $\g g_{\BB R}$ and
$\g p_{\BB R}$ is a $\g k_{\BB R}$-module under the $\operatorname{ad}$ action.
However, $\g p_{\BB R}$ is {\em not} a Lie algebra, unless
$[\g p_{\BB R}, \g p_{\BB R}]=0$.
\end{cor}

\begin{ex}  \label{SL(n)-ex}
Returning to our previous examples, let us take $G_{\BB R}=SL(n,\BB R)$,
its complexification $G=SL(2,\BB C)$ and a compact real form $U_{\BB R}=SU(n)$.
Then the conjugations
$$
\sigma, \tau: \g{sl}(n,\BB C) \to \g{sl}(n,\BB C) \quad \text{are} \quad
\sigma(X)=\bar X, \quad \tau(X)=-X^*,
$$
$\sigma$ and $\tau$ commute, so the Cartan involution $\theta$ is
$\theta(X)=-X^T$ (negative transpose).
$$
\g k = \g{so}(n,\BB C), \quad
\g p = \{ \text{symmetric matrices in $\g{sl}(n,\BB C)$} \},
$$
$$
\g k_{\BB R} = \g{so}(n,\BB R), \quad
\g p_{\BB R} = \{ \text{symmetric matrices in $\g{sl}(n,\BB R)$} \},
$$
$\g u_{\BB R} = \g k_{\BB R} \oplus i \g p_{\BB R}$ amounts to
$$
\g{su}(n) = \g{so}(n,\BB R) \oplus
i \{ \text{symmetric matrices in $\g{sl}(n,\BB R)$} \}.
$$
\end{ex}

Note that every $X \in \g k_{\BB R} \subset \g u_{\BB R}$ is diagonalizable,
with purely imaginary eigenvalues.
Indeed, $\{t \mapsto \exp tX ;\: t \in \BB R \}$ is a one-parameter subgroup
of $U_{\BB R}$, and must therefore have bounded matrix entries;
that is possible only when $X$ is diagonalizable over $\BB C$,
with purely imaginary eigenvalues.
By the same reasoning, since $i \g p_{\BB R} \subset \g u_{\BB R}$,
every $X \in \g p_{\BB R} \subset \g g_{\BB R}$ is diagonalizable,
with real eigenvalues.

Let $K_{\BB R} = (U_{\BB R} \cap G_{\BB R})^0$ -- the connected component of
the identity element of $U_{\BB R} \cap G_{\BB R}$.
Clearly, $K_{\BB R}$ is a compact connected subgroup of $G_{\BB R}$
with Lie algebra $\g k_{\BB R}$. The following result describes what is called
the {\em global Cartan decomposition of $G_{\BB R}$}.

\begin{prop}  \label{Cartan-decomp}
The map
$$
K_{\BB R} \times \g p_{\BB R} \to G_{\BB R}, \quad (k,X) \mapsto k \cdot \exp X,
$$
is a diffeomorphism of ${\cal C}^{\infty}$-manifolds.
\end{prop}

\begin{cor}
$K_{\BB R} \hookrightarrow G_{\BB R}$ is a strong deformation retract.
In particular, $G_{\BB R}$ and $K_{\BB R}$ are homotopic to each other,
and this inclusion induces isomorphisms of homology and homotopy groups.
\end{cor}

\begin{ex}
$SL(2,\BB R) \simeq SU(1,1)$ is diffeomorphic to $S^1 \times \BB R^2$ and
homotopic to $S^1$.
\end{ex}

This diffeomorphism allows us to define the {\em global Cartan involution}
(also denoted by $\theta$)
$$
\theta: G_{\BB R} \to G_{\BB R} \quad \text{by} \quad
\text{$\theta: (k,X) \mapsto (k,-X)$ on $K_{\BB R} \times \g p_{\BB R}$.}
$$
This global Cartan involution satisfies $\theta^2=Id_{G_{\BB R}}$ and
its differential at the identity is the Cartan involution
$\theta: \g g_{\BB R} \to \g g_{\BB R}$.
%(which is a Lie algebra automorphism).
%It follows that the global Cartan involution is a Lie group automorphism.
The group $K_{\BB R}$ can be characterized as the fixed point set of $\theta$,
i.e., $K_{\BB R} = \{ g \in G_{\BB R} ;\: \theta g =g \}$.

\begin{lem}
The global Cartan involution is a Lie group automorphism.
\end{lem}

\begin{proof}
Let $\tilde G_{\BB R}$ be the universal covering group of $G_{\BB R}$.
By Proposition \ref{Cartan-decomp}, it is diffeomorphic to
$\tilde K_{\BB R} \times \g p_{\BB R}$, where $\tilde K_{\BB R}$ is the
universal covering group of $K_{\BB R}$.
Since the (local) Cartan involution $\theta: \g g_{\BB R} \to \g g_{\BB R}$
is a Lie algebra automorphism, it ``lifts'' to a Lie group automorphism
$\tilde \theta: \tilde G_{\BB R} \to \tilde G_{\BB R}$.
It is easy to check that $\tilde\theta$ is the identity map on $\tilde K_{\BB R}$
and the $X \mapsto -X$ map on $\g p_{\BB R}$, i.e.
$$
\tilde \theta: (k,X) \mapsto (k,-X) \quad \text{on} \quad
\tilde K_{\BB R} \times \g p_{\BB R}.
$$
Then the group automorphism $\tilde\theta$ descends to a group automorphism
$\theta: G_{\BB R} \to G_{\BB R}$, which is equal to
$\theta: (k,X) \mapsto (k,-X)$ on $K_{\BB R} \times \g p_{\BB R}$.
\end{proof}

\begin{ex}
Returning to the setting of Example \ref{SL(n)-ex}, we have $K_{\BB R} = SO(n)$,
the Cartan involutions are
$$
\theta: \g{sl}(n,\BB R) \to \g{sl}(n,\BB R), \quad \theta(X)=-X^T,
$$
and
$$
\theta: SL(n,\BB R) \to SL(n,\BB R), \quad \theta(g)=(g^{-1})^T.
$$
In the setting of this example, Proposition \ref{Cartan-decomp} is
essentially equivalent to the assertion that any invertible real square
matrix can be expressed uniquely as the product of an orthogonal matrix
and a positive definite symmetric matrix.
\end{ex}

\subsection{Maximal Compact Subgroups}

Recall that $K_{\BB R} = (U_{\BB R} \cap G_{\BB R})^0$.

\begin{prop}
Let $G_{\BB R}$ be a closed linear connected semisimple real Lie group, then
\begin{enumerate}
\item[\rm{(a)}]
Any compact subgroup of $G_{\BB R}$ is conjugate to a subgroup of $K_{\BB R}$
by an element of $G_{\BB R}$;

\item[\rm{(b)}]
$K_{\BB R}$ is a maximal compact subgroup of $G_{\BB R}$;

\item[\rm{(c)}]
Any two maximal compact subgroups of $G_{\BB R}$ are conjugate by an
element of $G_{\BB R}$;

\item[\rm{(d)}]
The normalizer of $K_{\BB R}$ in $G_{\BB R}$ coincides with $K_{\BB R}$.
\end{enumerate}
\end{prop}

Note that the properties of maximal compact subgroups of real semisimple
closed linear Lie groups are very similar to those of maximal tori of compact
groups.
Since the maximal compact subgroups are all conjugate,
the choice of any one of them is non-essential. At various points,
we shall choose a maximal compact subgroup;
the particular choice will not matter.

\begin{cor}  \label{max-compact-cor}
\begin{enumerate}
\item[\rm{(a)}]
All maximal compact subgroups of $G_{\BB R}$ are connected;

\item[\rm{(b)}]
$K_{\BB R}$ is the only subgroup of $G_{\BB R}$, connected or not, that has
Lie algebra $\g k_{\BB R}$;

\item[\rm{(c)}]
$K_{\BB R}$ contains the center of $G_{\BB R}$.
\end{enumerate}
\end{cor}

Define:
$$
K = \text{the complexification of $K_{\BB R}$,}
$$
$K$ is a connected complex Lie subgroup of $G$
-- the complexification of $G_{\BB R}$ -- with Lie algebra $\g k$.
Observe that $K$ cannot be compact unless $K_{\BB R}=\{e\}$,
which  does not happen unless $G_{\BB R}=\{e\}$.
Indeed, $\g k_{\BB R} \ne \{0\}$ (otherwise $\g g_{\BB R} = \g p_{\BB R}$ is not
semisimple) and any non-zero $X \in \g k_{\BB R}$ is diagonalizable over $\BB C$,
with purely imaginary eigenvalues, not all zero, so the complex one-parameter
subgroup $\{z \mapsto \exp(zX);\: z \in \BB C \}$ of $K$ is unbounded.
By construction, the Lie algebras $\g g_{\BB R}$, $\g k_{\BB R}$, $\g g$, $\g k$
and the corresponding groups satisfy the following containments:
$$
\begin{matrix}
\g g_{\BB R}  & \subset & \g g & \qquad\qquad & G_{\BB R}  & \subset & G \\
\cup & \qquad & \cup & \qquad\qquad & \cup & \qquad & \cup \\
\g k_{\BB R} & \subset & \g k & \qquad\qquad & K_{\BB R} & \subset & K
\end{matrix}
$$

In general, we have a bijection
$$
\left\{ \begin{matrix} \text{isomorphism classes of} \\
\text{connected real compact} \\
\text{semisimple Lie groups $U_{\BB R}$} \end{matrix}\right\}
\simeq
\left\{ \begin{matrix} \text{isomorphism classes} \\
\text{of connected complex} \\
\text{semisimple Lie groups $G$}\end{matrix}\right\}
$$
\begin{align*}
U_{\BB R} &\mapsto \text{its complexification $G$}, \\
G &\mapsto \text{its compact real form $U_{\BB R}$}
\end{align*}
(in other words, $U_{\BB R}$ is a maximal compact subgroup of $G$).
The statement remains true if one drops ``connected'' and replaces
``semisimple'' with ``reductive''.

Since $\g g = \BB C\otimes_{\BB R}\,\g u_{\BB R}$, these two Lie algebras have
the same representations over $\BB C$ -- representations can be restricted
from $\g g$ to $\g u_{\BB R}$, and in the opposite direction, can be extended
complex linearly.
On the global level, these operations induce a canonical bijection
$$
\left\{ \begin{matrix} \text{finite-dimensional} \\ \text{continuous complex} \\
\text{representations of $U_{\BB R}$} \end{matrix}\right\}
\simeq
\left\{ \begin{matrix} \text{finite-dimensional} \\ \text{holomorphic} \\
\text{representations of $G$} \end{matrix}\right\}
$$
From a representation of $G$ one gets a representation of $U_{\BB R}$
by restriction. In the other direction, one uses the fact that continuous
finite-dimensional representations of Lie groups are necessarily real analytic,
and are determined by the corresponding infinitesimal representations of the
Lie algebra. Hence, starting from a finite-dimensional representation of
$U_{\BB R}$, one obtains a representation $\pi$ of Lie algebras $\g u_{\BB R}$
and $\g g$.
This representation then lifts to a representation $\tilde\pi$ of $\tilde G$
-- the universal cover of $G$, which contains $\tilde U_{\BB R}$
-- the universal cover of $U_{\BB R}$ (here we used $U_{\BB R} \hookrightarrow G$
induces an isomorphism of the fundamental group).
Since we started with a representation of $U_{\BB R}$, the kernel of
$\tilde U_{\BB R} \twoheadrightarrow U_{\BB R}$ acts trivially,
hence $\tilde\pi$ descends to
$\tilde G/ \ker(\tilde U_{\BB R} \twoheadrightarrow U_{\BB R}) =G$.

In particular, every finite-dimensional complex representation of $K_{\BB R}$
extends to a holomorphic representation of $K$.

We had mentioned earlier that the universal covering group of
$G=SL(n, \BB R)$, $n \ge 2$, is not a linear group.
We can now sketch the argument:

\begin{ex}  \label{not-linear}
Let $\widetilde{SL(n,\BB R)}$ be the universal covering group of $SL(n,\BB R)$,
$n \ge 2$. Since
$$
\pi_1(SL(n,\BB R)) = \pi_1(SO(n)) =
\begin{cases}
\BB Z,        & \text{if $n=2$;}  \\
\BB Z/2\BB Z, & \text{if $n \ge 3$,}
\end{cases}
$$
the universal covering $\widetilde{SL(n,\BB R)} \rightarrow SL(n,\BB R)$
is a principal $\BB Z$-bundle when $n=2$ and a principal $\BB Z/2\BB Z$-bundle
when $n \ge 3$. If $\widetilde{SL(n,\BB R)}$ were linear,
its complexification would have to be a covering group of
$SL(n,\BB C)=$ complexification of $SL(n, \BB R)$,
of infinite order when $n=2$ and of order (at least) two when $n \ge 3$.
But $SU(n)$, $n \ge 2$, is simply connected, as can be shown by induction
on $n$. But then $SL(n,\BB C)=$ complexification of $SU(n)$ is also simply
connected, and therefore cannot have a non-trivial covering.
We conclude that $\widetilde{SL(n,\BB R)}$ is not a linear group.
\end{ex}

\section{Definition of a Representation}

Interesting representations of non-compact groups are typically
infinite-dimensional. To apply analytic and geometric methods,
it is necessary to have a topology on the representation space and
to impose an appropriate continuity condition on the representation
in question. In the finite-dimensional case, there is only one
``reasonable" topology and continuity hypothesis,
but in the infinite dimensional case choices must be made.
One may want to study both complex and real representations.
There is really no loss in treating only the complex case,
since one can complexify  a real representation and regard the original
space as an $\BB R$-linear subspace of its complexification.

\subsection{A Few Words on Topological Vector Spaces}

We shall consider representations on {\em complete locally convex\footnote{
    A topological vector space is called {\em locally convex} if there is
    a base for the topology consisting of convex sets.}
Hausdorff topological vector spaces over} $\BB C$.
That includes complex Hilbert spaces, of course.
Unitary representations are of particular interest, and one might think
that Hilbert spaces constitute a large enough universe of representation
spaces. It turns out, however, that even the study of unitary representations
naturally leads to the consideration of other types of topological spaces.
So, before we give the definition of a representation,
we review topological vector spaces.

First of all, all vector spaces are required to be over $\BB C$ and
{\em complete}, which means every Cauchy sequence converges.
(There is a way to define Cauchy sequences for topological vector spaces
without reference to any metric\footnote{Let $V$ be a topological vector space.
A sequence $\{x_n\}_{n=1}^{\infty}$ in $V$ is a {\em Cauchy sequence}
if, for any non-empty open set $U \subset V$ containing $0$,
there is some number $N$ such that
$$
(x_n-x_m) \in U, \qquad \text{whenever $m,n >N$.}
$$
Then it is easy to check that it is sufficient to let the open sets $U$
range over a {\em local base} for $V$ about $0$.})

\separate

\noindent
\underline{Hilbert Spaces}: These are complete complex vector spaces with
positive-definite inner product, their topology is determined by the metric
induced by the inner product.

Example: $L^2$-spaces $L^2(X,\mu)$, where $(X,\mu)$ is a measure space,
the inner product is defined by
$$
\langle f, g \rangle = \int f \bar g \,d\mu.
$$

\separate

\noindent
\underline{Banach Spaces}:
These are complete complex vector spaces with a norm\footnote{Recall,
if $V$ is a vector space, a {\em norm} on $V$ is a function
$\|\,.\,\|: V \to [0, \infty)$ such that
\begin{itemize}
\item
$\|x\|=0$ if and only if $x=0$;
\item
$\|\lambda x\| = |\lambda| \, \|x\|$ for all $x \in V$ and $\lambda \in \BB C$;
\item
(triangle inequality)
$\|x+y\| \le \|x\|+\|y\|$ for all $x,y \in V$.
\end{itemize}},
their topology is determined by the metric induced by the norm.

Example: $L^p$-spaces $L^p(X,\mu)$, $1 \le p \le \infty$,
where $(X,\mu)$ is a measure space, the norm is defined by
$$
\| f\|_p = \left( \int |f|^p\,d\mu \right)^{1/p},
$$
$$
\|f\|_{\infty} = \inf \{ a \ge 0 ;\: \mu( \{ x \in X ;\: f(x)>a \})=0 \},
\quad \inf \varnothing = \infty.
$$
(We exclude $p<1$ because the triangle inequality may fail.)

\separate

\noindent
\underline{Complete Metric Spaces}:
These are complete complex vector spaces with a metric,
their topology is determined by the metric.

Example: $L^p$-spaces $L^p(X,\mu)$, $0 < p \le \infty$,
where $(X,\mu)$ is a measure space, the metric is defined by
$$
d(f,g) = \int |f-g|^p \,d\mu.
$$

\separate

\noindent
\underline{Fr\'echet Spaces}:
These are complex vector spaces with a countable family of seminorms.
Since they occur so often in representation theory,
let us review these spaces in more detail.

Let $V$ be a vector space. A {\em seminorm} on $V$ is a function
$p: V \to [0, \infty)$ such that
\begin{itemize}
\item
$p(\lambda x) = |\lambda| \, p(x)$ for all $x \in V$ and $\lambda \in \BB C$;
\item
(triangle inequality)
$p(x+y) \le p(x)+p(y)$ for all $x,y \in V$.
\end{itemize}
If $p(x)=0$ implies $x=0$, then $p$ is a {\em norm}.

Let $V$ be a vector space with a family of seminorms
$\{ p_{\alpha} \}_{\alpha \in A}$. We declare the ``balls''
$$
U_{x\alpha\epsilon} = \{ v \in V ;\: p_{\alpha}(v-x)<\epsilon \},
\quad x \in V,\: \alpha \in A,\: \epsilon >0,
$$
to be open sets and equip $V$ with the topology generated by these sets
$U_{x\alpha\epsilon}$.
This topology is Hausdorff if and only if for each $x \in V$, $x \ne 0$,
there exists $\alpha \in A$ such that $p_{\alpha}(x) \ne 0$.
If $V$ is Hausdorff and $A$ is countable, then $V$ is metrizable with a
translation-invariant metric (i.e. $d(x,y)=d(x+z,y+z)$ for all $x,y,z \in V$).
Finally, a {\em Fr\'echet space} is a complete Hausdorff topological vector
space whose topology is defined by a countable family of seminorms.
(Of course, every Fr\'echet space is metrizable -- take, for example,
$d(x,y) = \sum_{k=1}^{\infty} 2^{-k} \frac{p_k(x-y)}{1+p_k(x-y)}$.)

Examples: Let $U \subset \BB R^n$ be an open set, and let
$$
V_0 = \{ \text{continuous functions $f: U \to \BB C$} \}.
$$
For each compact set $K \subset U$ we can define a seminorm
$$
p_K(f) = \sup_{x \in K} |f(x)|.
$$
Of course, this family of seminorms is uncountable, but we can find a sequence
of compact sets
$$
K_1 \subset K_2 \subset K_3 \subset \dots \subset U
\quad \text{such that} \quad
\bigcup_{i=1}^{\infty} K_i = U
$$
and use seminorms $p_i = p_{K_i}$ to define topology on $V_0$.
Then $V_0$ becomes a Fr\'echet space (one needs to check the completeness).

We modify the example by letting
$$
V_k = \{ \text{${\cal C}^k$ functions $f: U \to \BB C$} \}.
$$
For each compact set $K$ and multiindex $\alpha$, $|\alpha| \le k$,
we can define a seminorm
$$
p_{K,\alpha}(f) = \sup_{x \in K} \left|
\frac{\partial^{|\alpha|}f}{\partial x^{\alpha}}(x) \right|.
$$
Then the countable family of seminorms $p_{K_i,\alpha}$ turns $V_k$ into a
Fr\'echet space (again, completeness requires verification).

Finally, let
$$
V_{\infty} = \{ \text{${\cal C}^{\infty}$ functions $f: U \to \BB C$} \}.
$$
For each compact set $K$ and multiindex $\alpha$ we have a seminorm
$$
p_{K,\alpha}(f) = \sup_{x \in K} \left|
\frac{\partial^{|\alpha|}f}{\partial x^{\alpha}}(x) \right|.
$$
Then the countable family of seminorms $p_{K_i,\alpha}$ turns $V_{\infty}$ into a
Fr\'echet space (of course, completeness is not entirely obvious).

\separate

As was mentioned before, we study representations on
{\em complete locally convex Hausdorff topological vector spaces over} $\BB C$.
But for most purposes Fr\'echet spaces will suffice.
(Recall that a topological vector space is called {\em locally convex} if
there is a base for the topology consisting of convex sets.)
Hilbert, Banach and Fr\'echet spaces are automatically locally convex,
but not all metric spaces are locally convex.
Local convexity is required to define the integral of vector-valued functions,
which is a crucial tool in the study of representations of reductive groups.

Finally, we prefer to work with topological spaces that are {\em separable},
i.e. have a countable dense subset.

\subsection{The Definition}

In this subsection $G_{\BB R}$ can be any real Lie group.
Recall that $V$ is a complete locally convex Hausdorff topological vector space
over $\BB C$ such as, for example, a Fr\'echet space.
Let $\operatorname{Aut}(V)$ denote the group of continuous, continuously
invertible, linear maps from $V$ to itself; we do not yet specify a topology
on this group.
There are at least four reasonable notions of continuity one could impose on
a homomorphism $\pi: G_{\BB R} \to \operatorname{Aut}(V)$:
\begin{itemize}
\item[\rm{a)}]
{\em continuity}: the action map $G_{\BB R} \times V \to V$ is continuous,
relative to the product topology on $G_{\BB R} \times V$;
\item[\rm{b)}]
{\em strong continuity}: for every $v \in V$, $g \mapsto \pi(g)v$
is continuous as map from $G_{\BB R}$ to $V$;
\item[\rm{c)}]
{\em weak continuity}: for every $v \in V$ and every $l$ in the
continuous linear dual space $V^*$, the complex-valued function
$g \mapsto \langle l, \pi(g)v \rangle$ is continuous;
\item[\rm{d)}]
{\em continuity in the operator norm}, which makes sense only if $V$ is
a Banach space; in that case, $\operatorname{Aut}(V)$ can be equipped with
the norm topology\footnote{If $V$ is a Banach space, then the space of
all bounded linear operators $A: V \to V$ is a Banach space with a norm
$$
\|A\| = \sup \{ \|Av\|;\: v \in V,\: \|v\| \le 1 \}.
$$},
and continuity in the operator norm means that
$\pi: G_{\BB R} \to \operatorname{Aut}(V)$ is a continuous homomorphism of
topological groups.
\end{itemize}

It is easy to see that
\begin{equation*}
\text{ continuity } \Longrightarrow
\text{ strong continuity } \Longrightarrow
\text{ weak continuity},
\end{equation*}
and if $V$ is a Banach space,
\begin{equation*}
\text{ continuity in the operator norm } \Longrightarrow
\text{ continuity}.
\end{equation*}

\begin{rem}
By the group property,
\begin{itemize}
\item[\rm{a)}]
is equivalent to $G_{\BB R} \times V \to V$ being continuous at $(e,v)$,
for every $v \in V$;
\item[\rm{b)}]
is equivalent to $g \mapsto \pi(g)v$ being continuous at $g=e$,
for every $v \in V$.
\end{itemize}
\end{rem}

\begin{ex}  \label{translation-ex}
The translation action of $(\BB R, +)$ on $L^p(\BB R)$ is continuous for
$1 \le p < \infty$, but {\em not} continuous in the operator norm;
for $p =\infty$ the translation action fails to be continuous,
strongly continuous, even weakly continuous.

%The translation action of $(\BB R, +)$ on ${\cal C}^0(\BB R)$ is
%strongly continuous (essentially because every continuous function
%is uniformly continuous on a compact subset), but not continuous.

%The translation action of $(\BB R, +)$ on ${\cal C}^{\infty}(\BB R)$
%is continuous.

The translation action of $(\BB R, +)$ on each ${\cal C}^k(\BB R)$,
$0 \le k \le \infty$, is continuous.

Let $V$ be the Banach space of bounded continuous functions on $\BB R$
with norm $\|f\| = \sup_{x \in \BB R} |f(x)|$. Then the translation action
of $(\BB R, +)$ on $V$ fails to be strongly continuous.
\end{ex}

This example shows that requiring continuity in the operator norm is too
restrictive -- most of the representations of interest involve translation.
Thus, from now on, ``representation" shall mean a continuous
-- in the sense of part (a) above -- linear action
$\pi: G_{\BB R} \to \operatorname{Aut}(V)$ on a complete, locally convex
Hausdorff space $V$.

\begin{ex}
Let $H_{\BB R} \subset G_{\BB R}$ be a closed subgroup, suppose that the
homogeneous space $G_{\BB R}/H_{\BB R}$ has a $G_{\BB R}$-invariant measure.
Then $G_{\BB R}$ acts on $L^2(G_{\BB R}/H_{\BB R})$ continuously and unitarily.
More generally, $G_{\BB R}$ acts continuously on
$$
L^p(G_{\BB R}/H_{\BB R}), \quad 1 \le p < \infty,
$$
$$
{\cal C}^{\infty}(G_{\BB R}/H_{\BB R}), \qquad
{\cal C}^{\infty}_c(G_{\BB R}/H_{\BB R}), \qquad
{\cal C}^{-\infty}(G_{\BB R}/H_{\BB R}),
$$
where ${\cal C}^{\infty}_c$ denotes the space of smooth functions with
compact support and ${\cal C}^{-\infty}$ denotes the space of distributions.
Later these will be declared infinitesimally equivalent representations.
\end{ex}

\begin{rem}
If both $G_{\BB R}$ and $H_{\BB R}$ are unimodular, then $G_{\BB R}/H_{\BB R}$ has a
$G_{\BB R}$-invariant measure.

A Lie group is called {\em unimodular} if it has a bi-invariant Haar measure.
The following groups are always unimodular: abelian, compact,
connected semisimple, nilpotent.
\end{rem}

\begin{prop}
Let $V$ be a Banach space, $\pi: G_{\BB R} \to \operatorname{Aut}(V)$ a
group homomorphism with no continuity hypothesis,
then the following are equivalent:
$$
\text{ continuity } \Longleftrightarrow
\text{ strong continuity } \Longleftrightarrow
\text{ weak continuity}.
$$
\end{prop}

In this chain of implications,
{\it strong continuity}~$\Longrightarrow$~{\it continuity}
follows relatively easily from the uniform boundedness principle\footnote{
{\bf Theorem} (Uniform Boundedness Principle).
Let $V$ be a Banach space and $W$ be a normed vector space.
Suppose that ${\cal F}$ is a collection of continuous linear operators from
$V$ to $W$ such that
$$
\sup_{T \in {\cal F}} \|T(v)\| < \infty, \quad \forall v \in V.
\qquad \text{Then} \qquad
\sup_{T \in {\cal F}} \|T\| < \infty.
$$},
but the implication
{\it weak continuity}~$\Longrightarrow$~{\it strong continuity}
is more subtle.

If $V$ is a topological vector space, one can equip its continuous linear dual
space $V^*$ with something called the strong dual topology.
In the case $V$ is a Banach space, this is the topology on $V^*$ defined by
the norm
$$
\| l \| =_{\text{def}} \sup_{v\in V, \: \|v\|=1} |l(v)|, \qquad l \in V^*.
$$
In general, if $\pi$ is continuous, the induced dual linear action on $V^*$
need not be continuous.
However, when $V$ is a reflexive Banach space (i.e. $V \simeq (V^*)^*$),
$V$ and $V^*$ play symmetric roles in the definition of weak continuity;
in this case, the dual action is also continuous, so there exists a
``dual representation'' $\pi^*$ of $G_{\BB R}$ on the dual Banach space $V^*$.

An infinite dimensional representation $(\pi, V)$ typically has numerous
invariant subspaces $W \subset V$, but the induced linear action of
$G_{\BB R}$ on $V/W$ is a purely algebraic object unless $V/W$ is Hausdorff,
i.e., unless $W \subset V$ is a closed subspace.
For this reason, the existence of a non-closed invariant subspace should not
be regarded as an obstacle to irreducibility: $(\pi, V)$ is {\em irreducible}
if $V$ has no proper {\em closed} $G_{\BB R}$-invariant subspaces.

In the same spirit, a {\em subrepresentation} of $(\pi, V)$ is a {\em closed}
$G_{\BB R}$-invariant subspace $W \subset V$.

A representation $(\pi, V)$ has {\em finite length} if every increasing chain
of closed $G_{\BB R}$-invariant subspaces breaks off after finitely many steps.
Equivalently, $(\pi, V)$ has finite length if there is a finite chain of
closed invariant subspaces
$$
0 = V_0 \subset V_1 \subset \dots \subset V_N = V
$$
such that each $V_i/V_{i-1}$ is irreducible.

\subsection{Admissible Representations}

As usual, we assume that $G_{\BB R}$ is a connected closed linear real Lie group
with semisimple Lie algebra. Recall that $K_{\BB R}$ denotes a maximal compact
subgroup of $G_{\BB R}$.

One calls a representation $(\pi, V)$ {\em admissible} if
$$
\dim_{\BB R} \operatorname{Hom}_{K_\BB R}(U,V) < \infty
$$
for every finite-dimensional
irreducible representation $(\tau, U)$ of $K_{\BB R}$.
Informally speaking, admissibility means that the restriction of
$(\pi, V)$ to $K_{\BB R}$ contains any irreducible $K_{\BB R}$-representation
only finitely often.

\begin{thm} [Harish-Chandra] \label{admissible}
Every irreducible unitary representation $(\pi,V)$ of $G_{\BB R}$  is admissible.
\end{thm}

Heuristically, admissible representations of finite length constitute the
smallest class that is invariant under ``standard constructions''
(in a very broad sense!) and contains the irreducible unitary representations.
One should regard inadmissible irreducible representations as exotic.
All irreducible representations which have come up naturally in geometry,
differential equations, physics, and number theory are admissible.

\section{Harish-Chandra Modules}

The goal of this section is, starting with a representation $(\pi, V)$
of $G_{\BB R}$, to construct a representation of $\g g_{\BB R}$.
In Example \ref{translation-ex}, where $G_{\BB R} = \BB R$ acts by translations
on $L^p(\BB R)$, $1 \le p < \infty$, we get an action of the Lie algebra on
$f \in L^p(\BB R)$ if and only if $f$ is differentiable.
Thus we cannot expect $\g g_{\BB R}$ act on all of $V$,
and we want to find a ``good'' vector subspace of $V$ on which $\g g_{\BB R}$
can act and which contains enough information about $V$.

\subsection{$K_{\BB R}$-finite and ${\cal C}^{\infty}$ Vectors}

Let $U \subset \BB R^n$ be an open set. Let $V$ be a topological
vector space, and consider a function $f: U \to V$.
We can define partial derivatives (with values in $V$)
$$
\frac{\partial f}{\partial x_i}
= \lim_{\epsilon \to 0} \frac{f(x_1, \dots, x_i+\epsilon, \dots, x_n) -
f(x_1, \dots, x_i, \dots, x_n)}{\epsilon}
$$
if the limit exists.
Then we can make sense of partial derivatives of higher order and
${\cal C}^{\infty}$ functions $f: U \to V$.

If $V$ is a Banach space, we say that a function $f: U \to V$ is
${\cal C}^\omega$ (or real analytic) if near every point in $U$
it can be represented by an absolutely convergent $V$-valued power series.
($V$ is required to be a Banach space so we can use the norm to make sense
of absolute convergence.)
Since $G_{\BB R}$ ``locally looks like $\BB R^n$'', one can make sense of
${\cal C}^{\infty}$ and ${\cal C}^\omega$ functions $G_{\BB R} \to V$.

\begin{df}
Let $(\pi, V)$ be a representation of $G_{\BB R}$. A vector $v \in V$ is
\begin{itemize}
\item[{\rm a)}]
{\em $K_{\BB R}$-finite} if $v$ lies in a finite-dimensional
$K_{\BB R}$-invariant subspace;
\item[{\rm b)}]
a {\em ${\cal C}^{\infty}$ vector} if $g \mapsto \pi(g)v$ is a
${\cal C}^{\infty}$ map from $G_{\BB R}$ to $V$;
\item[{\rm c)}]
in the case of a Banach space $V$ only, an {\em analytic vector}
if $g \mapsto \pi(g)v$ is a ${\cal C}^\omega$ (or real analytic) map from
$G_{\BB R}$ to $V$;
\item[{\rm d)}]
{a weakly analytic vector} if, for every $l \in V^*$, the complex-valued
function $g \mapsto \langle l, \pi(g)v \rangle$ is real analytic.
\end{itemize}
\end{df}

All reasonable notions of a real analytic $V$-valued map agree when
$V$ is a Banach space, but not for other locally convex topological vector
spaces. That is the reason for defining the notion of an analytic vector only
in the Banach case. Surprisingly perhaps, even weakly real analytic functions
with values in a Banach space are real analytic (i.e., locally representable
by absolutely convergent vector valued power series) -- see the appendix in
\cite{La} for an efficient argument.
In the Banach case, then, the notions of an analytic vector and
of a weakly analytic coincide. For other representations,
the former is not defined, but the latter still makes sense.

As a matter of self-explanatory notation, we write $V_{fini}$, $V^{\infty}$
and $V^{\omega}$ for the spaces of $K_{\BB R}$-finite, smooth and weakly analytic
vectors in $V$. We do not equip these spaces with any topology.

\begin{thm} [Harish-Chandra] \label{densityof}
If $(\pi, V)$ is an admissible representation
(which may or may not be of finite length),
\begin{itemize}
\item[{\rm a)}]
$V_{fini}$ is a dense subspace of $V$;
\item[{\rm b)}]
Every $v \in V_{fini}$ is both a ${\cal C}^{\infty}$
vector and a weakly analytic vector;
\item[{\rm c)}]
In particular, $V^{\infty}$ and $V^{\omega}$ are dense in $V$.
\end{itemize}
\end{thm}

The theorem applies in particular to $K_{\BB R}$, considered as maximal compact
subgroup of itself. Finite-dimensional subspaces are automatically closed,
so the density of $K_{\BB R}$-finite vectors forces any infinite dimensional
representation $(\pi,V)$ of $K_{\BB R}$ to have proper closed invariant
subspaces. In other words,

\begin{cor}
Every irreducible representation of $K_{\BB R}$ is finite dimensional. 
\end{cor}

\subsection{Proof of Theorem \ref{densityof}}

In this subsection we prove Theorem \ref{densityof}. More precisely,
we prove that $V_{fini}$ is dense in $V$ and that $V_{fini} \subset V^{\infty}$,
but we omit the $V_{fini} \subset V^{\omega}$ part.
The proof is important because it illustrates essential techniques of
representation theory.

Since the group $G_{\BB R}$ is semisimple, it is unimodular, i.e. has a
bi-invariant Haar measure. So fix such a measure $dg$.
It is not unique, since, for example, one can scale it by a positive scalar,
but a particular choice of such measure is not essential.

We start with a representation $(\pi, V)$ and for the moment make no
assumptions on admissibility or finite length.
Let $f \in {\cal C}^0_c(G_{\BB R})$
(where ${\cal C}^0_c(G_{\BB R})$ denotes the space of continuous complex-valued
functions with compact support), define $\pi(f) \in \operatorname{End}(V)$ by
$$
\pi(f) v = \int_{g \in G_{\BB R}} f(g)\pi(g)v \,dg.
$$
Note that the integrand $f(g)\pi(g)v$ is a continuous compactly supported
function on $G_{\BB R}$ with values in $V$, its integral is defined as the limit
of Riemannian sums. Since $V$ is complete and locally convex, the limit of
these Riemannian sums exists and the integral is well-defined.

\separate

\noindent
\underline{Notation}: Let $l$ and $r$ denote the linear action of $G_{\BB R}$
on the spaces of functions on $G_{\BB R}$
(such as $L^p(G_{\BB R})$, ${\cal C}^k(G_{\BB R})$, ${\cal C}^k_c(G_{\BB R})$,
$0 \le k \le \infty$) induced by left and right multiplication:
$$
\bigl( l(g)f \bigr)(h) = f(g^{-1}h), \qquad
\bigl( r(g)f \bigr)(h) = f(hg).
$$

\begin{lem}  \label{comp-lem}
\begin{itemize}
\item[{\rm a)}]
If $f \in {\cal C}^0_c(G_{\BB R})$ and $g \in G_{\BB R}$, then
$$
\pi(g) \circ \pi(f) = \pi\bigl( l(g) f \bigr),
\qquad
\pi(f) \circ \pi(g) = \pi\bigl( r(g^{-1}) f \bigr);
$$

\item[{\rm b)}]
For $f_1, f_2 \in {\cal C}^0_c(G_{\BB R})$ we have
$$
\pi(f_1) \circ \pi(f_2) = \pi\bigl( l(f_1) f_2 \bigr) = \pi(f_1 * f_2),
$$
where $f_1 * f_2$ denotes the convolution product:
%$$
%(f_1 * f_2)(h) = \int_{g \in G_{\BB R}} f_1(g) f_2(g^{-1}h) \,dg;
%$$
\begin{align*}
(f_1 * f_2)(h) &= \int_{g \in G_{\BB R}} f_1(g) f_2(g^{-1}h) \,dg \\
&= \int_{g \in G_{\BB R}} f_1(hg^{-1}) f_2(g) \,dg;
\end{align*}

\item[{\rm c)}]
Let $\{f_n\} \subset {\cal C}^0_c(G_{\BB R})$ be an approximate identity
(this means $f_n \ge 0$, $\int_{G_{\BB R}} f_n(g) \,dg=1$ for all $n$
and $\{ \supp(f_n) \} \searrow \{e\}$), then $\pi(f_n) \to Id_V$ strongly,
i.e. $\pi(f_n)v \to v$ for all $v \in V$.
\end{itemize}
\end{lem}

The proof of this lemma is left as a homework.

\begin{lem}[G{\aa}rding]
$V^{\infty}$ is dense in $V$.
\end{lem}

\begin{proof}
Let $\{f_n\} \subset {\cal C}^{\infty}_c(G_{\BB R})$ be an approximate identity.
Then $\pi(f_n)v \to v$ for all $v \in V$.
Thus it suffices to show that $\pi(f)v \in V^{\infty}$ if
$f \in {\cal C}^{\infty}_c(G_{\BB R})$, i.e. that the map
$$
G_{\BB R} \to V, \qquad g \mapsto \pi(g)\circ\pi(f)v,
$$
is smooth. By Lemma \ref{comp-lem}, $\pi(g)\circ\pi(f)v = \pi(l(g)f)v$,
and we can break this map as a composition
$$
G_{\BB R} \to {\cal C}^{\infty}_c(G_{\BB R}) \to V, \qquad
g \mapsto l(g)f \quad \text{and} \quad \tilde f \mapsto \pi(\tilde f)v.
%\mapsto \pi(l(g)f)v.
$$
For a fixed $v \in V$, the second map
$$
{\cal C}^{\infty}_c(G_{\BB R}) \to V, \qquad \tilde f \mapsto \pi(\tilde f)v,
$$
is linear and continuous.
Thus it remains to show that, for each $f \in {\cal C}^{\infty}_c(G_{\BB R})$,
the map
$$
G_{\BB R} \to {\cal C}^{\infty}_c(G_{\BB R}), \qquad g \mapsto l(g)f,
$$
is smooth. But this is true essentially because the function $f$ is smooth.
\end{proof}

Fix a maximal compact subgroup $K_{\BB R} \subset G_{\BB R}$.
Let $\hat K_{\BB R}$ denote the set of isomorphism classes of finite-dimensional
irreducible representations of $K_{\BB R}$
(each of these representations is automatically unitary).
Thus, for each $i \in \hat K_{\BB R}$, we have a representation $(\tau_i, U_i)$
of $K_{\BB R}$. Let
$$
V(i) =_{\text{def}} \text{ the linear span of all images of
$T \in \operatorname{Hom}_{K_{\BB R}}(U_i,V)$},
$$
and call $V(i)$ the {\em $i$-isotypic subspace of $V$}.
Then the space of $K_{\BB R}$-finite vectors
$$
V_{fini} = \bigoplus_{i \in \hat K_{\BB R}} V(i)
\qquad \text{(algebraic direct sum)}.
$$
Note that the representation $(\pi,V)$ is admissible if and only if
the dimension of each $V(i)$ is finite.

Let $\chi_i$ be the character of $(\tau_i, U_i)$, so
$$
\chi_i \in {\cal C}^{\infty}(K_{\BB R}), \qquad \chi_i(k)=\tr(\tau_i(k)),
\quad k \in K_{\BB R},
$$
and introduce notations
$$
\phi_i = \dim U_i \cdot \overline{\chi_i}, \quad
\pi_{K_{\BB R}} = \pi \bigr|_{K_{\BB R}}, \quad
l_{K_{\BB R}} = l \bigr|_{K_{\BB R}}, \quad r_{K_{\BB R}} = r \bigr|_{K_{\BB R}}.
$$
We consider $\pi_{K_{\BB R}}(\phi_i) \in \operatorname{End}(V)$,
\begin{multline*}
\pi_{K_{\BB R}}(\phi_i)v =_{\text{def}}
\int_{K_{\BB R}} \phi_i(k) \pi_{K_{\BB R}}(k)v \,dk \\
= \dim U_i \cdot \int_{K_{\BB R}} \overline{\chi_i}(k) \pi_{K_{\BB R}}(k)v \,dk,
\end{multline*}
where $dk$ denotes the bi-invariant Haar measure on $K_{\BB R}$ normalized by
the requirement $\int_{K_{\BB R}} 1\,dk=1$.

We recall Schur Orthogonality Relations:

\begin{thm}
Let $(\tau_1,U_1)$, $(\tau_2,U_2)$ be irreducible finite-\\dimensional unitary
representations of a compact Lie group $K_{\BB R}$ with $K_{\BB R}$-invariant
inner products $(\,\cdot\,,\,\cdot\,)_1$ and $(\,\cdot\,,\,\cdot\,)_2$.
If $x_1,x_2 \in U_1$ and $y_1,y_2 \in U_2$,
\begin{multline*}
\int_{K_{\BB R}} (\tau_1(k)x_1,x_2)_1 \cdot \overline{(\tau_2(k)y_1,y_2)_2} \,dk \\
= \begin{cases}
0 & \text{if $(\tau_1,U_1) \not\simeq (\tau_2,U_2)$;} \\
%0 & \text{if $(\tau_1,U_1) \simeq \hskip-23pt \diagup \hskip7pt (\tau_2,U_2)$;}\\
\frac1{\dim U_1} (x_1,y_1)_1 \cdot \overline{(x_2,y_2)_1}
& \text{if $(\tau_1,U_1) = (\tau_2,U_2)$.}
\end{cases}
\end{multline*}
\end{thm}

\begin{cor}
\begin{itemize}
\item[{\rm a)}]
$\tau_i(\phi_i) = Id_{U_i}$;
\item[{\rm b)}]
$$
\chi_i*\chi_j = \begin{cases}
0 & \text{if $i \ne j$;} \\
\frac{\chi_i}{\dim U_i} & \text{if $i = j$,}
\end{cases}
\qquad
\phi_i*\phi_j = \begin{cases}
0 & \text{if $i \ne j$;} \\
\phi_i & \text{if $i = j$.}
\end{cases}
$$
\end{itemize}
\end{cor}

\begin{proof}
\begin{itemize}
\item[{\rm a)}]
Homework.
%Let $\{x_m\}$ be an orthonormal basis for $U_i$ relative to the
%$K_{\BB R}$-invariant inner product $(\,\cdot\,,\,\cdot\,)_i$.
%By Schur Orthogonality Theorem,
%\begin{multline*}
%(\tau_i(\phi_i)x_m, x_n)_i
%= \int_{K_{\BB R}} \bigl( \phi_i(k) \tau_i(k)x_m, x_n \bigr)_i \,dk \\
%= \dim U_i \sum_p \int_{K_{\BB R}} (\tau_i(k)x_m, x_n)_i \cdot
%\overline{(\tau_i(k)x_p, x_p)_i} \,dk \\
%= \sum_p (x_m, x_p)_i \cdot \overline{(x_n, x_p)_i}
%= \begin{cases} 0 & \text{if $m \ne n$;} \\
%1 & \text{if $m=n$.}\end{cases}
%\end{multline*}
%This proves $\tau_i(\phi_i) = Id_{U_i}$.

\item[{\rm b)}]
Let $\{x_m\}$ and $\{y_n\}$ be orthonormal bases for $U_i$ and $U_j$
relative to the $K_{\BB R}$-invariant inner products
$(\,\cdot\,,\,\cdot\,)_i$ and $(\,\cdot\,,\,\cdot\,)_j$.
Write
$$
(\chi_i*\chi_j)(h) = \int_{K_{\BB R}} \chi_i(k) \chi_j(k^{-1}h) \,dk.
$$
Then
\begin{multline*}
\chi_i(k) \chi_j(k^{-1}h)
= \sum_{m,n} (\tau_i(k)x_m,x_m)_i \cdot (\tau_j(k^{-1}h)y_n,y_n)_j \\
= \sum_{m,n,p} (\tau_i(k)x_m,x_m)_i \cdot (\tau_j(k^{-1})y_p,y_n)_j
\cdot (\tau_j(h)y_n,y_p)_j \\
= \sum_{m,n,p} (\tau_i(k)x_m,x_m)_i \cdot \overline{(\tau_j(k)y_n,y_p)_j}
\cdot (\tau_j(h)y_n,y_p)_j,
\end{multline*}
so Schur Orthogonality Relations imply
$$
\int_{K_{\BB R}} (\tau_i(k)x_m,x_m)_i \cdot \overline{(\tau_j(k)y_n,y_p)_j} \,dk
\quad \text{and} \quad \chi_i*\chi_j
$$
are $0$ when $i \ne j$.
When $i=j$, we can take the basis $\{y_n\}$ to be the same as $\{x_m\}$,
then, by Schur Orthogonality Relations again,
\begin{multline*}
\int_{K_{\BB R}} \chi_i(k) \chi_i(k^{-1}h) \,dk  \\
= \frac1{\dim U_i} \sum_{m,n,p} (x_m,x_n)_i \cdot \overline{(x_m,x_p)_i}
\cdot (\tau_i(h)x_n,x_p)_i  \\
= \frac1{\dim U_i} \sum_{p} (\tau_i(h)x_p,x_p)_i = \frac{\chi_i(h)}{\dim U_i}.
\end{multline*}
\end{itemize}
\end{proof}

\begin{lem}
\begin{itemize}
\item[{\rm a)}]
$$
\pi_{K_{\BB R}}(\phi_i) \circ \pi_{K_{\BB R}}(\phi_j) =
\begin{cases}
0 & \text{if $i \ne j$;} \\
\pi_{K_{\BB R}}(\phi_i) & \text{if $i = j$;}
\end{cases}
$$

\item[{\rm b)}]
The image of $\pi_{K_{\BB R}}(\phi_i)$ is $V(i)$
and $\pi_{K_{\BB R}}(\phi_i)$ restricted to $V(i)$ is the identity map.
In other words, $\pi_{K_{\BB R}}(\phi_i)$ is a projection onto $V(i)$;

\item[{\rm c)}]
$V_{fini}$ is dense in $V$ (we make no assumption on admissibility of $V$).
\end{itemize}
\end{lem}

\begin{proof}
\begin{itemize}
\item[{\rm a)}]
We have:
$$
\pi_{K_{\BB R}}(\phi_i) \circ \pi_{K_{\BB R}}(\phi_j) =\pi_{K_{\BB R}}(\phi_i*\phi_j)
= \begin{cases}
0 & \text{if $i \ne j$;} \\
\pi_{K_{\BB R}}(\phi_i) & \text{if $i = j$.}
\end{cases}
$$

\item[{\rm b)}]
Recall that
$$
V(i) = \text{the linear span of all images of
$T \in \operatorname{Hom}_{K_{\BB R}}(U_i,V)$}.
$$
Then, for all $x \in U_i$,
\begin{multline*}
\bigl( \pi_{K_{\BB R}}(\phi_i) \circ T \bigr) x
= \int_{K_{\BB R}} \phi_i(k) \bigl( \pi_{K_{\BB R}}(k) \circ T \bigr) x \,dk  \\
= T \circ \int_{K_{\BB R}} \phi_i(k) \tau_i(k)x \,dk
= \bigl( T \circ \tau_i(\phi_i) \bigr) x = Tx,
\end{multline*}
since $\tau_i(\phi_i) = Id_{U_i}$.
This proves that the image of $\pi_{K_{\BB R}}(\phi_i)$ contains $V(i)$ and
that $\pi_{K_{\BB R}}(\phi_i)$ restricted to $V(i)$ is the identity map.

On the other hand, the collection of functions
$\{ l_{K_{\BB R}}(k) \phi_i \}$, $k \in K_{\BB R}$,
is a linear combination of matrix coefficients
$$
\overline{(\tau_i(k^{-1}h)x_m, x_m)_i} = (\tau_i(k)x_m, \tau_i(h)x_m)_i,
$$
where $\{x_m\}$ is a basis for $U_i$ and $(\,\cdot\,,\,\cdot\,)_i$ is a
$K_{\BB R}$-invariant inner product, hence lie in
$$
\text{$\BB C$-span of $\{ (x_n, \tau_i(h)x_m)_i ;\:
1 \le m,n \le \dim U_i \}$},
$$
which is a finite-dimensional vector subspace in ${\cal C}^{\infty}(K_{\BB R})$.
In particular, for all $v \in V$, the collection of vectors
$$
\pi(k) \bigl( \pi_{K_{\BB R}}(\phi_i) v \bigr)
= \pi_{K_{\BB R}} \bigl( l_{K_{\BB R}}(k) \phi_i \bigr)v,
\qquad k \in K_{\BB R},
$$
spans a finite-dimensional subspace of $V$.
Hence
$$
\pi_{K_{\BB R}}(\phi_i) v \in V_{fini} = \bigoplus_{j \in \hat K_{\BB R}} V(j).
$$
By part (a),
$$
\pi_{K_{\BB R}}(\phi_i) \bigr|_{V(j)} =0 \quad \text{if $i \ne j$}.
$$
Since $\pi_{K_{\BB R}}(\phi_i) \circ \pi_{K_{\BB R}}(\phi_i) = \pi_{K_{\BB R}}(\phi_i)$,
it follows that \\
$\im \pi_{K_{\BB R}}(\phi_i) \subset V(i)$.

\item[{\rm c)}]
Recall the Stone-Weierstrass Theorem:
Let $X$ be a compact Hausdorff space and let ${\cal C}^0(X)$ denote the space
of continuous complex-valued functions on $X$.
Suppose $A$ is a subalgebra of ${\cal C}^0(X)$ such that it contains
a non-zero constant function, closed under conjugation and separates points.
Then $A$ is dense in ${\cal C}^0(X)$.

By the Stone-Weierstrass Theorem, $l(K_{\BB R})$-finite functions are dense in
${\cal C}^0(K_{\BB R})$.
Indeed, the $l(K_{\BB R})$-finite functions include matrix coefficients
of finite-dimensional representations of $K_{\BB R}$.
The fact that they separate points follows from existence of a faithful
finite-dimensional representation of $K_{\BB R}$ (which in turn can be deduced
from Peter-Weyl Theorem).

Using the approximate identity argument, one can show that the set
$$
\{ \pi_{K_{\BB R}}(f)v ;\: f \in {\cal C}^0(K_{\BB R}),\: v \in V \}
$$
is dense in $V$. Therefore, the set
\begin{equation}  \label{l-finite}
\{ \pi_{K_{\BB R}}(f)v ;\: f \in {\cal C}^0(K_{\BB R}),\:
\text{$f$ is $l(K_{\BB R})$-finite}, \:v \in V \}
\end{equation}
is also dense in $V$.
But any vector in the set (\ref{l-finite}) is automatically in $V_{fini}$.
\end{itemize}
\end{proof}

\begin{cor}
For each $i \in \hat K_{\BB R}$, $V^{\infty} \cap V(i)$ is dense in $V(i)$.
\end{cor}

\begin{proof}
Let $v \in V(i)$. Choose an approximate identity
$\{f_n\} \subset {\cal C}^{\infty}_c(G_{\BB R})$.
Then $\pi(f_n)v \to v$, hence
$$
\pi_{K_{\BB R}}(\phi_i) \circ \pi(f_n)v \to \pi_{K_{\BB R}}(\phi_i) v  =v.
$$
But
$$
\pi_{K_{\BB R}}(\phi_i) \circ \pi(f_n) = \pi \bigl( l_{K_{\BB R}}(\phi_i) f_n \bigr)
\text{ and } l_{K_{\BB R}}(\phi_i) f_n \in {\cal C}^{\infty}_c(G_{\BB R}).
$$
Hence each $\pi_{K_{\BB R}}(\phi_i) \circ \pi(f_n)v \in V^{\infty} \cap V(i)$.
\end{proof}

Recall that a representation $(\pi,V)$ is admissible if $\dim V(i) <\infty$
for all $i \in \hat K_{\BB R}$.

\begin{cor}
If $(\pi,V)$ is admissible, $V_{fini} \subset V^{\infty}$.
\end{cor}

Recall that a topological vector space is separable if it has a countable
dense subset. If $(\pi,V)$ is admissible, combining vector space bases
for each $V(i)$, $i \in \hat K_{\BB R}$, results in a countable set of linearly
independent vectors in $V$ such that their linear combinations are dense in $V$.
Thus, if $(\pi,V)$ is admissible, $V$ is automatically separable.

Finally, we comment that, in order to prove that all $K_{\BB R}$-finite vectors
are weakly analytic, one can show that the functions
$g \mapsto \langle l, \pi(g)v \rangle$, for $v\in V_{fini}$ and $l \in V^*$,
satisfy elliptic differential equations with $C^\omega$ coefficients,
which implies they are real analytic.
To construct such an operator, use the Casimir elements
$\Omega_G \subset {\cal U}(\g g)$ (its symbol is hyperbolic) and
$\Omega_K \subset {\cal U}(\g k) \subset {\cal U}(\g g)$
(its symbol is semidefinite, but degenerate) and show that
$\Omega_G - 2 \Omega_K \in {\cal U}(\g g)$ is elliptic.
Then argue that some polynomial of $\Omega_G - 2 \Omega_K$ annihilates
$\langle l, \pi(g)v \rangle$.

\subsection{Harish-Chandra Modules}

For $v \in V^{\infty}$ and $X \in \g g_{\BB R}$, {\em define}
$$
\pi(X)v =_{\text{def}} \frac{d}{dt} \pi \bigl(\exp(tX)\bigr)v \Bigr|_{t=0},
$$
then $\pi(X): V^{\infty} \to V^{\infty}$.
This defines a representation of $\g g_{\BB R}$. on $V^{\infty}$.
Complexifying, we obtain a representation of $\g g$ on $V^{\infty}$.
The action of $\g g$ extends to the universal enveloping algebra, thus
$V^{\infty}$ is a ${\cal U}(\g g)$-module.

\begin{lem}
Whether or not $(\pi,V)$ is admissible, $V^{\infty} \cap V_{fini}$ is
$\g g$-invariant.
\end{lem}

\begin{proof}
Let $\g g \otimes V^{\infty} \to V^{\infty}$ be the action map
$$
X \otimes v \mapsto \frac{d}{dt} \pi \bigl( \exp(tX)\bigr)v \Bigr|_{t=0},
$$
This map is $K_{\BB R}$-equivariant, i.e., for all $k \in K_{\BB R}$,
\begin{multline*}
\operatorname{Ad}(k)X \otimes \pi(k)v \mapsto
\frac{d}{dt}\pi\bigl(\exp(t \operatorname{Ad}(k)X)\bigr) \pi(k)v \Bigr|_{t=0} \\
= \frac{d}{dt} \pi(k) \circ \pi \bigl(\exp(tX)\bigr)v \Bigr|_{t=0}  \\
= \lim_{\epsilon \to 0} \frac{\pi(k) \circ \pi \bigl(\exp(\epsilon X)\bigr)v
- \pi(k)v}{\epsilon} \\
= \pi(k) \lim_{\epsilon \to 0} \frac{\pi \bigl( \exp(\epsilon X)\bigr)v - v}
{\epsilon}
= \pi(k) \frac{d}{dt} \pi \bigl( \exp(tX) \bigr)v \Bigr|_{t=0}.
\end{multline*}
Hence, if $v \in V^{\infty}$ is $K_{\BB R}$-finite (i.e. lies in a
finite-dimensional $K_{\BB R}$-invariant subspace), then so is $\pi(X)v$.
\end{proof}

\begin{cor}
If $(\pi,V)$ is admissible, $V_{fini}$ is a ${\cal U}(\g g)$-module.
\end{cor}

From now on we assume that the representation $(\pi,V)$ is
\underline{admissible}.
Not only ${\cal U}(\g g)$ acts on $V_{fini}$, but also $K_{\BB R}$.
As a $K_{\BB R}$-representation,
$$
V_{fini} = \bigoplus_{i \in \hat K_{\BB R}} V(i),
\qquad \text{each $V(i)$ is finite-dimensional}.
$$
Recall that $K$ is a complexification of $K_{\BB R}$
and the restriction from $K$ to $K_{\BB R}$ gives a bijection
$$
\left\{ \begin{matrix} \text{finite-dimensional} \\ \text{holomorphic} \\
\text{representations of $K$} \end{matrix}\right\}
\simeq
\left\{ \begin{matrix} \text{finite-dimensional} \\ \text{continuous complex} \\
\text{representations of $K_{\BB R}$} \end{matrix}\right\}
$$
We conclude that $K$ also acts on $V_{fini}$.

\noindent
\underline{Caution}: $K$ does {\em not} act on $V$!

Even though $V_{fini}$ has no natural Hausdorff topology --
it is not closed in $V$ unless $\dim V < \infty$ -- it makes sense to say
that $K$ acts holomorphically on $V_{fini}$:
like $K_{\BB R}$, $K$ acts {\em locally finitely},
in the sense that every vector lies in a finite dimensional invariant subspace;
the invariant finite dimensional subspaces do carry natural Hausdorff
topologies, and $K$ does act holomorphically on them.
The Lie algebra $\g k$ has two natural actions on $V_{fini}$,
by differentiation of the $K$-action, and via the inclusion
$\g k \subset \g g$ and the ${\cal U}(\g g)$-module structure.
These two actions coincide, essentially by construction.
Moreover, for all $X \in {\cal U}(\g g)$, $v \in V_{fini}$ and $k \in K$,
\begin{equation}\label{adjoint}
\pi(k)(Xv) = (\operatorname{Ad}(k)X)(\pi(k)v),
\end{equation}
as can be deduced from the well known formula
$\exp(\operatorname{Ad}k X) = k \exp(X) k^{-1}$,
for $X \in \g g_{\BB R}$, $k \in K_{\BB R}$. 

\begin{df}  \label{(g,K)-mod_def}
A {\em $(\g g, K)$-module} is a complex vector space $M$, equipped with the
structure of ${\cal U}(\g g)$-module and with a linear action of $K$ such that:
\begin{itemize}
\item[{\rm a)}]
The action of $K$ is locally finite, i.e., every $m \in M$ lies in a
finite-dimensional $K$-invariant subspace on which $K$ acts holomorphically;
\item[{\rm b)}]
When the $K$-action is differentiated, the resulting action of $\g k$ agrees
with the action of $\g k$ on $M$ via $\g k \hookrightarrow \g g$ and the
${\cal U}(\g g)$-module structure.
\item[{\rm c)}]
The identity (\ref{adjoint}) holds for all $k\in K$, $X \in {\cal U}(\g g)$,
$v\in M$.
\end{itemize}
\end{df}

\begin{df}
A {\em Harish-Chandra module} is a $(\g g, K)$-module $M$ which is
finitely generated over ${\cal U}(\g g)$ and admissible, in the sense that
every irreducible $K$-representation occurs in $M$ with finite multiplicity. 
\end{df}

\begin{rem}
If $K$ is connected (which is the case if $G_{\BB R}$ is connected),
the compatibility condition (c) follows from condition (b).
Indeed, it is sufficient to verify (\ref{adjoint}) on the infinitesimal level,
in which case it becomes, by (b),
$$
ZXv = [Z,X]v + XZv, \qquad \forall X \in {\cal U}(\g g), \: Z \in \g k,
$$
which is true.

But in the case of non-connected group $G_{\BB R}$, each maximal compact
subgroup $K_{\BB R} \subset G_{\BB R}$ meets every connected component of
$G_{\BB R}$, hence also disconnected. In this case the compatibility condition
(b) does {\em not} imply (\ref{adjoint}), and the compatibility condition (c)
must be stated separately.
\end{rem}

The discussion leading up to the definition shows that the space of
$K_{\BB R}$-finite vectors $V_{fini}$ of an admissible representation
$(\pi, V)$ is an admissible $(\g g,K)$-module.
In order to show that $V_{fini}$ is a Harish-Chandra module, we need to show
that $V_{fini}$ is finitely-generated over ${\cal U}(\g g)$.
This part requires that $V$ has finite length and a few more steps.
Recall a part of Theorem \ref{densityof}:

\begin{thm}  \label{weakly-analytic-thm}
$V_{fini} \subset V^{\omega}$. In other words, for each $v \in V_{fini}$ and
$l \in V^*$, the matrix coefficient function $G_{\BB R} \to \BB C$,
$g \mapsto \langle l, \pi(g)v \rangle$, is real-analytic.
\end{thm}

\begin{cor}
Let $(\pi,V)$ be an admissible representation of $G_{\BB R}$ and let
$W$ be a submodule of the $(\g g,K)$-module $V_{fini}$.
Then $\overline{W}$ -- the closure of $W$ in $V$ -- is $G_{\BB R}$-invariant
(hence a subrepresentation).
\end{cor}

\begin{proof}
Suppose $w \in W$, we need to show that $\pi(g)w$ lies in $\overline{W}$,
for all $g \in G_{\BB R}$.
Equivalently, we must show that the matrix coefficient
$g \mapsto \langle l, \pi(g)w \rangle$ is identically zero for all
$l$ in the annihilator of $W$ in $V^*$.
By Theorem \ref{weakly-analytic-thm}, this function is real-analytic.
$W$ being $\g g$-invariant implies that all derivatives of this function
at the identity element $e$ vanish. Hence the function is zero.
\end{proof}

\begin{lem}
Let $(\pi,V)$ be an admissible representation of $G_{\BB R}$,
then we have a bijection
$$
\left\{ \begin{matrix} \text{closed $G_{\BB R}$-invariant} \\
\text{subspaces of $V$} \end{matrix} \right\}
\simeq
\left\{ \begin{matrix} \text{$(\g g,K)$-submodules} \\
\text{of $V_{fini}$} \end{matrix} \right\}
$$
\begin{align*}
\text{subrepresentation $\tilde V \subset V$} \quad &\mapsto \quad
\tilde V_{fini} \subset V_{fini} \\
\text{$(\g g,K)$-submodule $W \subset V_{fini}$} \quad &\mapsto \quad
\text{closure $\overline{W} \subset V$.}
\end{align*}
\end{lem}

\begin{proof}
If $\tilde V \subset V$ is a closed $G_{\BB R}$-invariant subspace, then
$\overline{\tilde V_{fini}} = \tilde V$, since $\tilde V_{fini}$ is dense in
$\tilde V$.

Conversely, if $W \subset V_{fini}$ is a $(\g g,K)$-submodule, then
$(\overline{W})_{fini} = W$.
Indeed, we certainly have $(\overline{W})_{fini} \supset W$.
If $(\overline{W})_{fini} \ne W$, there exists an $i \in \hat K_{\BB R}$
such that $\overline{W}(i) \supsetneq W(i)$.
Let $w \in \overline{W}(i)$, then there exists a sequence
$\{ w_n \}$ in $W$ converging to $w$. Apply $\pi_{K_{\BB R}}(\phi_i)$ to get
$$
\pi_{K_{\BB R}}(\phi_i)w_n \to \pi_{K_{\BB R}}(\phi_i)w =w, \qquad
\pi_{K_{\BB R}}(\phi_i)w_n \in W(i)
$$
(here we used that $W$ is $K$-invariant). Hence
$$
\overline{W}(i) \subset \overline{W(i)} = W(i),
$$
since $W(i)$ is finite-dimensional.
\end{proof}

\begin{cor}
A representation $(\pi,V)$ is irreducible if and only if $V_{fini}$
is irreducible as $(\g g,K)$-module.
\end{cor}

\begin{cor}
Let $(\pi,V)$ be an admissible representation of $G_{\BB R}$ of finite length,
then $V_{fini}$ is a Harish-Chandra module.
\end{cor}

\begin{proof}
We need to show that $V_{fini}$ is finitely generated as a
${\cal U}(\g g)$-module. Observe that if $w_1,\dots,w_n \in V_{fini}$, then
the $(\g g, K)$-submodule $W$ generated by $w_1,\dots,w_n$ is finitely
generated as a ${\cal U}(\g g)$-module.
Indeed, by the compatibility condition (c)  of Definition \ref{(g,K)-mod_def},
as a ${\cal U}(\g g)$-module, $W$ is generated by
$$
\{ \pi(k_1)w_1,\dots,\pi(k_n)w_n ;\: k_1,\dots,k_n \in K \},
$$
and those vectors span a finite-dimensional vector subspace $W_0 \subset W$.
Then a basis of $W_0$ generates $W$ over ${\cal U}(\g g)$.

If $V_{fini}$ is {\em not} finitely generated, we get an infinite chain of
${\cal U}(\g g)$-modules and  hence an infinite chain of $(\g g,K)$-submodules
$$
W_1 \subsetneq W_2 \subsetneq W_3 \subsetneq \dots \subset V_{fini}.
$$
Take their closures in $V$:
$$
\overline{W_1} \subset \overline{W_2} \subset \overline{W_3} \subset \dots
\subset V.
$$
By previous lemma, all inclusions are proper, so we  get an infinite chain of
closed $G_{\BB R}$-invariant subspaces of $V$.
But this contradicts to $V$ being of finite length.
\end{proof}

From now on, we write $\HC(V)$ for the space of $K_{\BB R}$-finite vectors of
an admissible representation $(\pi,V)$ and call $\HC(V)$
{\em the Harish-Chandra module of} $\pi$.
The next statement formalizes the properties of Harish-Chandra
modules we have mentioned so far: 

\begin{thm}[Harish-Chandra] \label{HCf}
The association
$$
V \mapsto \HC(V) = V_{fini}
$$
establishes a covariant, exact\footnote{A functor is {\em exact}
if it preserves exact sequences.},
faithful\footnote{A functor $F: {\cal C} \to {\cal C}'$ is {\em faithful}
if, for all $X,Y \in {\cal C}$, the map
$\operatorname{Hom}_{\cal C}(X,Y) \to \operatorname{Hom}_{{\cal C}'}(F(X),F(Y))$
is injective.} functor
\begin{multline*}
\Bigl\{\begin{matrix}
\text{category of admissible $G_{\BB R}$-representations of finite}
\\
\text{length and continuous linear $G_{\BB R}$-equivariant maps}
\end{matrix}\Bigr\} \\
\overset{\HC}{\longrightarrow}   \
\Bigl\{\begin{matrix}
\text{category of Harish-Chandra modules} \\
\text{and $(\g g, K)$-equivariant linear maps}
\end{matrix}\Bigr\}.
\end{multline*}
\end{thm}

This functor is {\em not} full\footnote{A functor $F: {\cal C} \to {\cal C}'$
is {\em full} if, for all $X,Y \in {\cal C}$, the map
$\operatorname{Hom}_{\cal C}(X,Y) \to \operatorname{Hom}_{{\cal C}'}(F(X),F(Y))$
is surjective.}, however.
As we will see in Subsection \ref{inf-equiv-subsection}, one can easily
construct two admissible representations of finite length $(\pi_1,V_1)$ and
$(\pi_2,V_2)$ that have isomorphic Harish-Chandra modules, but there is no
continuous $G_{\BB R}$-equivariant map $V_1 \to V_2$ inducing the isomorphism map
$\HC(V_1) \simeq \HC(V_2)$.

\subsection{Infinitesimal Equivalence}  \label{inf-equiv-subsection}

\begin{df}
Two admissible representations of finite length $(\pi_1,V_1)$ and $(\pi_2,V_2)$
are {\em infinitesimally equivalent} or {\em infinitesimally isomorphic}
if $\HC(V_1) \simeq \HC(V_2)$.
\end{df}

Loosely speaking, infinitesimal equivalence means that the two representations
are the same except for the choice of topology.
This is a good notion of equivalence because, for example,
if two irreducible unitary representations are infinitesimally equivalent,
they are isomorphic as unitary representations.

\separate

\underline{Example}:
Let $G_{\BB R}=SU(1,1)$. Recall that
\begin{equation*}
SU(1,1) = \left\{ \begin{pmatrix} a & b \\ \bar b & \bar a \end{pmatrix} ; \:
a,b \in \BB C, \: |a|^2 - |b|^2 =1 \right\},
\end{equation*}
$G_{\BB R}$ has $G=SL(2,\BB C)$ as complexification,
and is conjugate in $G$ to $SL(2,\BB R)$.
As maximal compact subgroup, we choose the diagonal subgroup,
in which case its complexification also consists of diagonal matrices:
\begin{align*}
K_{\BB R} &= \left\{ k_\theta = \begin{pmatrix} e^{i\theta} & 0 \\ 0 & e^{-i\theta}
\end{pmatrix} ;\: \theta \in \BB R\,\right\} \simeq SO(2) \simeq U(1), \\
K &= \left\{ \begin{pmatrix} a & 0 \\ 0 & a^{-1} \end{pmatrix} ;\:
a \in \BB C^{\times} \right\} \simeq \BB C^{\times}.
\end{align*}
The group $SU(1,1)$ acts transitively by fractional linear transformations
on the open unit disc
$$
\BB D = \{ z \in \BB C ;\: |z|<1 \}.
$$
Since the isotropy subgroup at the origin is $K_{\BB R}$,
$$
\BB D \simeq SU(1,1)/K_{\BB R}.
$$
We denote the space of holomorphic functions on $\BB D$ by $H^{-\omega}(\BB D)$.
The group $SU(1,1)$ acts on $H^{-\omega}(\BB D)$ by left translations:
\begin{equation}  \label{lD-action}
\bigl( \ell(g)f \bigr)(z) = f(g^{-1} \cdot z), \qquad
%g \in SU(1,1), \:
f \in H^{-\omega}(\BB D), \: z \in \BB D,
\end{equation}
and on the subspace
\begin{equation*}
H^2(\BB D) =_{\text{def}}
\begin{matrix} \text{space of holomorphic functions on $\BB D$} \\
\text{with $L^2$ boundary values,} \end{matrix}
\end{equation*}
topologized by the inclusion $H^2(\BB D) \hookrightarrow L^2(S^1)$.
One can show that both actions are representations, i.e., they are continuous
with respect to the natural topologies of $H^{-\omega}(\BB D)$ and $H^2(\BB D)$.

Note that
$$
\bigl( \ell(k_{\theta}) f \bigr)(z) = f(e^{-2i\theta}\cdot z), \qquad
\ell(k_{\theta})z^n = e^{-2in\theta} z^n.
$$
Hence, $f \in H^2(\BB D)$ is $K_{\BB R}$-finite if and only if
$f$ has a finite Taylor series at the origin, i.e., if and only if
$f$ is a polynomial:
$$
H^2(\BB D)_{fini} = \BB C[z].
$$
In particular, $(\ell,H^2(\BB D))$ is admissible.
This representation is not irreducible, since $H^2(\BB D)$ contains
the constant functions $\BB C$ as an obviously closed invariant subspace.
It does have finite length; in fact, the quotient $H^2(\BB D)/\BB C$
is irreducible, as follows from a simple infinitesimal calculation in
the Harish-Chandra module $\HC(H^2(\BB D))=\BB C[z]$.

Besides $V=H^2(\BB D)$, the action (\ref{lD-action}) on each of the
following spaces, equipped with the natural topology in each case,
defines a representation of $SU(1,1)$:
\begin{itemize}
\item[{\rm a)}]
$H^p(\BB D)$ = space of holomorphic functions on $\BB D$ with
$L^p$ boundary values, $1 \le p < \infty$;
\item[{\rm b)}]
$H^\infty(\BB D)$ = space of holomorphic functions on $\BB D$ with
${\cal C}^{\infty}$ boundary values;
\item[{\rm c)}]
$H^{-\infty}(\BB D)$ = space of holomorphic functions on $\BB D$
with distribution boundary values;
\item[{\rm d)}]
$H^\omega(\BB D)$ = space of holomorphic functions on $\BB D$
with real analytic boundary values;
\item[{\rm e)}]
$H^{-\omega}(\BB D)$ = space of all holomorphic functions on $\BB D$.
\end{itemize}
Taking boundary values, one obtains inclusions
$H^p(\BB D)\hookrightarrow L^p(S^1)$,
which are equivariant with respect to the action of $SU(1,1)$ on
$L^p(S^1)$ by linear fractional transformations.
The latter fails to be continuous when $p=\infty$, but that is not the case
for the image of $H^\infty(\BB D)$ in $L^\infty(S^1)$.
One can show that $H^\infty(\BB D)$ is the space of ${\cal C}^\infty$ vectors for
the Hilbert space representation $(\ell,H^2(\BB D))$.

Arguing as in the case of $H^2(\BB D)$, one finds that the representation
$\ell$ of $SU(1,1)$ on each of the spaces a)-e) has $\BB C[z]$ as
Harish-Chandra module, so all of them are infinitesimally equivalent.
This is the typical situation, not just for $SU(1,1)$,
but for all groups $G_{\BB R}$ of the type we are considering:
every infinite-dimensional admissible representation $(\pi,V)$
of finite length is infinitesimally equivalent to an infinite
family of pairwise non-isomorphic representations.

\subsection{A Few Words about Globalization}  \label{globalization}

One can ask the following very natural question:
``Does every Harish-Chandra module arise as the space of $K_{\BB R}$-finite
vectors of some admissible representation of $G_{\BB R}$ of finite length?''
In other words: ``Does every Harish-Chandra module $M$ have a
{\em globalization}, i.e. an admissible $G_{\BB R}$-representation
$(\pi,V)$ of finite length, such that $\HC(V)=M$?''
The answer to this question is affirmative and is due to W.Casselman
\cite{Ca1}, but the proof is quite complicated and indirect.

Then it is natural to ask if a globalization can be chosen in a functorial
manner -- in other words, whether the functor $\HC$ in Theorem \ref{HCf}
has a right inverse.
Such functorial globalizations do exist. Four of them are of particular
interest, the $C^{\infty}$ and $C^{-\infty}$ globalizations due to
W.Casselman and N.Wallach \cite{Ca2, Wal}, as well as the minimal and
the maximal globalizations due to M.Kashiwara and W.Schmid \cite{Sch, KSch}.
All four are {\em topologically exact}, i.e., they map exact
sequences of Harish-Chandra modules into exact sequences of representations
in which every morphism has {\em closed range}.
The main technical obstacle in constructing the canonical globalizations
is to establish this closed range property.

We describe without proofs the maximal globalization functor due to
M.Kashiwara and W.Schmid \cite{KSch}.

\begin{df}
Let $M$ be a Harish-Chandra module. By the {\em maximal globalization} we mean
a representation $\operatorname{MG}(M)$ of $G_{\BB R}$ such that
$\HC(\operatorname{MG}(M))=M$ and, for any other globalization $V$ of $M$,
the identity map $M \to M$ extends to a $G_{\BB R}$-equivariant continuous
linear map $V \to \operatorname{MG}(M)$.

Similarly, one can define the {\em minimal globalization}.
\end{df}

We start with a Harish-Chandra module $M$.
Let $M^*$ denote the vector space of all complex linear functions on $M$.
Since $M$ does not have any topology, neither does $M^*$. Define
$$
M' =_{\text{def}} (M^*)_{fini},
$$
then $\g g$ and $K$ act on $M'$ making it a $(\g g, K)$-module and, in fact,
a Harish-Chandra module.
We call $M'$ the dual Harish-Chandra module of $M$.

Define
$$
\widetilde{\operatorname{MG}}(M) =_{\text{def}}
\operatorname{Hom}_{(\g g,K_{\BB R})}(M', {\cal C}^{\infty}(G_{\BB R})),
$$
where the $(\g g,K_{\BB R})$-action on ${\cal C}^{\infty}(G_{\BB R})$ is induced
by the action of $G_{\BB R}$ by multiplications on the left:
$$
\bigl( l(g)f \bigr)(h) = f(g^{-1}h).
$$
Then we have a linear action of $G_{\BB R}$ on
$\widetilde{\operatorname{MG}}(M)$ induced by the action of $G_{\BB R}$ on
${\cal C}^{\infty}(G_{\BB R})$ by multiplications on the right:
$$
\bigl( r(g)f \bigr)(h) = f(hg).
$$
Moreover, $\widetilde{\operatorname{MG}}(M)$ has a natural Fr\'echet space
topology. Indeed, fix a vector space basis $\{m_i\}$ of $M'$,
it is automatically countable.
Pick a diffeomorphism between an open set $U$ of $\BB R^n$ and an open set in
$G_{\BB R}$, and let $C \subset U$ be a compact subset.
Then smooth functions on $G_{\BB R}$ can be ``restricted'' to $U$ and
it makes sense to talk about their partial derivatives.
Finally, let $\alpha$ be a multiindex.
Then the topology of $\widetilde{\operatorname{MG}}(M)$ is induced by seminorms
$$
p_{C,\alpha,i}(f) = \sup_{x \in C} \left|
\frac{\partial^{|\alpha|} f(m_i)}{\partial x^{\alpha}}(x) \right|,
\qquad f \in \widetilde{\operatorname{MG}}(M),
$$
with the compact sets $C$ ranging over a countable family covering $G_{\BB R}$.
The linear action of $G_{\BB R}$ on $\widetilde{\operatorname{MG}}(M)$
is continuous with respect to this topology.

The space $\widetilde{\operatorname{MG}}(M)$ remains unchanged if one replaces
${\cal C}^{\infty}(G_{\BB R})$ by the space of real-analytic functions
${\cal C}^{\omega}(G_{\BB R})$:

\begin{lem}
The inclusion
${\cal C}^{\omega}(G_{\BB R}) \hookrightarrow {\cal C}^{\infty}(G_{\BB R})$
induces a topological isomorphism
$$
\operatorname{Hom}_{(\g g,K_{\BB R})}(M', {\cal C}^{\omega}(G_{\BB R}))
\simeq \operatorname{Hom}_{(\g g,K_{\BB R})}(M', {\cal C}^{\infty}(G_{\BB R})).
$$
\end{lem}

\begin{thm}[M.Kashiwara and W.Schmid, 1994]
The representation $\widetilde{\operatorname{MG}}(M)$ of $G_{\BB R}$
is the maximal globalization of the Harish-Chandra module $M$.
The functor $M \mapsto \widetilde{\operatorname{MG}}(M)$
is the right adjoint\footnote{
Let ${\cal C}$, ${\cal D}$ be two categories, and let
$F: {\cal D} \rightarrow {\cal C}$, $G: {\cal C} \rightarrow {\cal D}$
be two functors. Suppose there is a family of bijections
$$
\operatorname{Hom}_{\cal C}(FY,X) \cong \operatorname{Hom}_{\cal D}(Y,GX)
$$
which is natural in the variables $X$ and $Y$.
Then the functor $F$ is called a left adjoint functor,
while $G$ is called a right adjoint functor.}
of $\HC$ and is topologically exact.
\end{thm}

\subsection{Unitary Representations and Infinitesimal Equivalence}

\begin{thm} [Harish-Chandra]
If two irreducible unitary representations are infinitesimally
equivalent\footnote{By Theorem \ref{admissible} irreducible unitary
representations are automatically admissible, so it makes sense to talk about
their underlying Harish-Chandra modules.},
they are isomorphic as unitary representations.
\end{thm}

\begin{proof}
Let $(\pi_1,V_1)$ and $(\pi_2,V_2)$ be two irreducible unitary infinitesimally
equivalent representations with $G_{\BB R}$-invariant inner products
$(\,\cdot\,,\,\cdot\,)_1$ and $(\,\cdot\,,\,\cdot\,)_2$,
and let $M$ be the common Harish-Chandra module.
Restricting, we get two (possibly different) inner products on $M$,
also denoted by $(\,\cdot\,,\,\cdot\,)_1$ and $(\,\cdot\,,\,\cdot\,)_2$.
For all $X \in \g g_{\BB R}$, $u,v \in M$ and $\alpha=1,2$, we have
$$
0 = \frac{d}{dt} \bigl(\pi_{\alpha}(\exp tX)u, \pi_{\alpha}(\exp tX)v\bigr)_{\alpha}
\Bigr|_{t=0} = (Xu,v)_{\alpha} + (u,Xv)_{\alpha}.
$$
Complexifying, we get
$$
(Xu,v)_{\alpha} + (u,\overline{X}v)_{\alpha} = 0,
\qquad X \in \g g, \quad \alpha=1,2,
$$
where $\overline{X}$ denotes the complex conjugate of $X$ relative to
$\g g_{\BB R} \subset \g g$.

Also, for both inner products, applying Schur's lemma for $K_{\BB R}$, we get
$$
M(i) \perp M(j), \qquad i \ne j.
$$
Hence $(\,\cdot\,,\,\cdot\,)_1$ and $(\,\cdot\,,\,\cdot\,)_2$ define
conjugate-linear isomorphisms
$$
M= \bigoplus_{i \in \hat K_{\BB R}} M(i) \simeq \bigoplus_{i \in \hat K_{\BB R}} M(i)^*.
$$
Compose one isomorphism with the inverse of the other to get a linear
isomorphism $M \simeq M$ which, by construction, relates the two inner products
and is $(\g g,K)$-equivariant.

This implies $V_1 \simeq V_2$ as Hilbert spaces, and this isomorphism
commutes with the actions of $K_{\BB R}$ on $V_{\alpha}$ and $\g g$ on
$(V_{\alpha})_{fini}$, $\alpha=1,2$.
We need to show that this isomorphism commutes with the actions of $G_{\BB R}$
as well. For this purpose we use Theorem \ref{weakly-analytic-thm},
which says that every $K_{\BB R}$-finite vector is weakly analytic.
Suppose that this isomorphism $V_1 \simeq V_2$ identifies
\begin{align*}
(V_1)_{fini} \ni u_1 &\longleftrightarrow u_2 \in (V_2)_{fini}, \\
(V_1)_{fini} \ni v_1 &\longleftrightarrow v_2 \in (V_2)_{fini}.
\end{align*}
Then the matrix coefficient functions
$$
(\pi_1(g)u_1,v_1)_1 \quad \text{and} \quad (\pi_2(g)u_2,v_2)_2,
\qquad g \in G_{\BB R},
$$
are real-analytic and have the same derivatives at the identity element,
hence coincide and
$$
V_1 \ni \pi_1(g)u_1 \longleftrightarrow \pi_2(g)u_2 \in V_2.
\qquad \forall g \in G_{\BB R}.
$$
By continuity, the isomorphism $V_1 \simeq V_2$ commutes with the actions of
$G_{\BB R}$.
\end{proof}

\section{Construction of Representations of $SU(1,1)$}
\label{V-chi-construction}

In this section we construct representations of $SU(1,1)$ following
the book \cite{Vi}. This construction is very similar to that of irreducible
finite-dimensional representations of $SL(2,\BB C)$ given in
Example \ref{F_d-construction}.

\subsection{Construction of Representations of $SL(2,\BB R)$}

The construction requires a parameter
$$
\chi=(l,\epsilon), \quad \text{where} \quad
l \in \BB C \quad \text{and} \quad \epsilon \in \{0,1/2\}.
$$
The $SL(2,\BB R)$ realization can be constructed as follows.
Let $SL(2,\BB R)$ act on $\BB R^2 \setminus \{0\}$ by matrix multiplication:
\begin{multline*}
\begin{pmatrix} a & b \\ c & d \end{pmatrix}
\begin{pmatrix} x_1 \\ x_2 \end{pmatrix}
= \begin{pmatrix} ax_1+bx_2 \\ cx_1+dx_2 \end{pmatrix}, \\
\begin{pmatrix} a & b \\ c & d \end{pmatrix} \in SL(2,\BB R), \quad
\begin{pmatrix} x_1 \\ x_2 \end{pmatrix} \in \BB R^2 \setminus \{0\}.
\end{multline*}
Let $\tilde V_{\chi}$ denote the space of $\BB C$-valued functions
$\phi(x_1,x_2)$ on $\BB R^2 \setminus \{0\}$ with the following properties:
\begin{itemize}
\item
$\phi(x_1,x_2)$ is smooth;
\item
$\phi(ax_1,ax_2) = a^{2l} \cdot \phi(x_1,x_2)$, for all $a \in \BB R$, $a>0$,
i.e. $\phi$ is homogeneous of homogeneity degree $2l$;
\item
$\phi$ is even if $\epsilon=0$ and odd if $\epsilon=1/2$:
$$
\phi(-x_1,-x_2)=(-1)^{2\epsilon} \cdot \phi(x_1,x_2).
$$
\end{itemize}
Then $SL(2,\BB R)$ acts on $\tilde V_{\chi}$ by
\begin{multline*}
\bigl( \tilde\pi_{\chi}(g) \phi \bigr)(x_1,x_2)
= \phi \bigl( g^{-1} \cdot (x_1,x_2) \bigr), \\
g \in SL(2,\BB R), \quad \phi \in \tilde V_{\chi}, \quad
(x_1,x_2) \in \BB R^2 \setminus \{0\}.
\end{multline*}
It is easy to see that $\tilde V_{\chi}$ remains invariant under this action,
so this action is well-defined.

The space $\tilde V_{\chi}$ is a closed subspace of the Fr\'echet space of all
smooth functions on $\BB R^2 \setminus \{0\}$, hence inherits Fr\'echet space
topology. The $SL(2,\BB R)$-action is continuous with respect to this topology,
thus we get a representation $(\tilde\pi_{\chi},\tilde V_{\chi})$ of
$SL(2,\BB R)$.

\subsection{Construction of Representations of $SU(1,1)$}

If we want to do any computations with the representation
$(\tilde\pi_{\chi},\tilde V_{\chi})$ of $SL(2,\BB R)$,
it is much more convenient to rewrite it as a representation of $SU(1,1)$.

Recall that
$$
SU(1,1) = \biggl\{ \begin{pmatrix} a & b \\ \bar b & \bar a \end{pmatrix}
\in SL(2,\BB C);\: a,b \in \BB C,\: |a|^2-|b|^2=1 \biggr\}.
$$
We need to replace $\BB R^2$ with a real 2-dimensional $SU(1,1)$-invariant
subspace of $\BB C^2$. Note that the real subspace
$$
\left\{ \begin{pmatrix} z \\ \bar z \end{pmatrix} \in \BB C^2 ;\: z \in \BB C
\right\}
$$
is preserved by $SU(1,1)$:
$$
\begin{pmatrix} a & b \\ \bar b & \bar a \end{pmatrix}
\begin{pmatrix} z \\ \bar z \end{pmatrix}
=
\begin{pmatrix} az+b \bar z \\ \bar b z + \bar a \bar z \end{pmatrix}
=
\begin{pmatrix} az+b \bar z \\ \overline{az+b \bar z} \end{pmatrix}.
$$

As before, we use a parameter
$$
\chi=(l,\epsilon), \quad \text{where} \quad
l \in \BB C \quad \text{and} \quad \epsilon \in \{0,1/2\}.
$$
Let $V_{\chi}$ denote the space of $\BB C$-valued functions
$\phi(z)$ on $\BB C \setminus \{0\}$ with the following properties:
\begin{itemize}
\item
$\phi(z)$ is a smooth function of $x$ and $y$, where $z=x+iy$;
\item
$\phi(az) = a^{2l} \cdot \phi(z)$, for all $a \in \BB R$, $a>0$,
i.e. $\phi$ is homogeneous of homogeneity degree $2l$;
\item
$\phi$ is even if $\epsilon=0$ and odd if $\epsilon=1/2$:
$$
\phi(-z)=(-1)^{2\epsilon} \cdot \phi(z).
$$
\end{itemize}
Then $SU(1,1)$ acts on $V_{\chi}$ by
\begin{multline*}
\bigl( \pi_{\chi}(g) \phi \bigr)(z) = \phi( g^{-1} \cdot z), \\
g \in SU(1,1), \quad \phi \in V_{\chi}, \quad
z \in \BB C \setminus \{0\}.
\end{multline*}
Explicitly, if $g = \begin{pmatrix} a & b \\ \bar b & \bar a \end{pmatrix}$,
then $g^{-1} = \begin{pmatrix} \bar a & -b \\ -\bar b & a \end{pmatrix}$, and
$$
\bigl( \pi_{\chi}(g) \phi \bigr)(z) = \phi(g^{-1} \cdot z)
= \phi(\bar az-b \bar z),
\quad \phi \in V_{\chi}, \: z \in \BB C \setminus \{0\}.
$$
It is easy to see that $V_{\chi}$ remains invariant under this action,
so this action is well-defined. As before, $V_{\chi}$ is a Fr\'echet space
with topology inherited from the space of all smooth functions on
$\BB C \setminus \{0\}$, and the $SU(1,1)$-action is continuous with respect
to this topology. Thus we get a representation $(\pi_{\chi},V_{\chi})$ of
$SU(1,1)$.

Now, let us turn our attention to the space $V_{\chi}$.
Every homogeneous function on $\BB C \setminus \{0\}$ is
completely determined by its values on the unit circle
$S^1=\{z \in \BB C ;\: |z|=1\}$:
$$
\phi(z)= |z|^{2l} \cdot \phi(z/|z|), \qquad z/|z| \in S^1.
$$
Conversely, every smooth function on $S^1$ satisfying
$\phi(-z)=(-1)^{2\epsilon} \cdot \phi(z)$, $z \in S^1$, extends to a homogeneous
function of degree $2l$ which is an element of $V_{\chi}$.

It will be convenient to realize the space $V_{\chi}$ in another way on the
circle. Namely, for $\epsilon=0$ with each function $\phi(z)$ we associate
a function $f(e^{i\theta})$, defined by
$$
f(e^{i\theta}) =_{\text{def}} \phi(e^{i\theta/2}),
\qquad \theta \in \BB R.
$$
Since $\phi(z)$ is even, the function $f(e^{i\theta})$ is uniquely defined and
smooth.
If $\epsilon=1/2$ we take
$$
f(e^{i\theta}) =_{\text{def}} e^{i\theta/2} \cdot \phi(e^{i\theta/2}),
\qquad \theta \in \BB R.
$$
This function is uniquely defined and smooth, since $\phi(z)$ is odd.
In this way, for any $\chi(l,\epsilon)$ the space $V_{\chi}$
can be realized as the space ${\cal C}^{\infty}(S^1)$.

Let us find the expression for $\pi_{\chi}(g)$ under this identification of
$V_{\chi}$ with ${\cal C}^{\infty}(S^1)$. Assume first that $\epsilon=0$.
Let $g = \begin{pmatrix} a & b \\ \bar b & \bar a \end{pmatrix}$, we have:
\begin{multline*}
\bigl( \pi_{\chi}(g) f \bigr)(e^{i\theta})
= \bigl( \pi_{\chi}(g) \phi \bigr)(e^{i\theta/2})
= \phi(\bar a e^{i\theta/2} - b e^{-i\theta/2}) \\
= \bigl| \bar a e^{i\theta/2} - b e^{-i\theta/2} \bigr|^{2l} \cdot
\phi \biggl( \frac{\bar a e^{i\theta/2} - b e^{-i\theta/2}}
{| \bar a e^{i\theta/2} - b e^{-i\theta/2}|} \biggr).
\end{multline*}
We denote
$$
\frac{\bar a e^{i\theta/2} - b e^{-i\theta/2}}{|\bar a e^{i\theta/2} - b e^{-i\theta/2}|}
\quad \text{by} \quad  e^{i\psi/2}, \quad \text{then} \quad
e^{i\psi} = \frac{\bar a e^{i\theta} - b}{-\bar b e^{i\theta} +a}.
$$
Indeed,
$$
e^{i\psi} = \frac{(\bar a e^{i\theta/2} - b e^{-i\theta/2})^2}
{|\bar a e^{i\theta/2} - b e^{-i\theta/2}|^2}
= \frac{\bar a e^{i\theta/2} - b e^{-i\theta/2}}{a e^{-i\theta/2} - \bar b e^{i\theta/2}}
= \frac{\bar a e^{i\theta} - b}{-\bar b e^{i\theta} + a}.
$$
Since, in addition,
$$
\bigl| \bar a e^{i\theta/2} - b e^{-i\theta/2} \bigr| = |-\bar b e^{i\theta} + a|,
$$
it follows that
\begin{multline*}
\bigl( \pi_{\chi}(g) f \bigr)(e^{i\theta})
= \bigl| \bar a e^{i\theta/2} - b e^{-i\theta/2} \bigr|^{2l} \cdot \phi(e^{i\psi/2}) \\
= |-\bar b e^{i\theta} + a|^{2l} \cdot f( e^{i\psi})
= |-\bar b e^{i\theta} + a|^{2l} \cdot
f \biggl( \frac{\bar a e^{i\theta} - b}{-\bar b e^{i\theta} +a} \biggr).
\end{multline*}
Thus we have proved that for $\chi=(l,0)$ the representation $\pi_{\chi}(g)$
is realized in the space ${\cal C}^{\infty}(S^1)$ of complex-valued smooth
functions on the circle, and is given by the formula
$$
\bigl( \pi_{\chi}(g) f \bigr)(e^{i\theta}) = |-\bar b e^{i\theta} + a|^{2l} \cdot
f \biggl( \frac{\bar a e^{i\theta} - b}{-\bar b e^{i\theta} +a} \biggr).
$$

One similarly proves that, for $\chi=(l,1/2)$ the operators $\pi_{\chi}(g)$
are given by the formula
$$
\bigl( \pi_{\chi}(g) f \bigr)(e^{i\theta})
= |-\bar b e^{i\theta} + a|^{2l-1} \cdot (-\bar b e^{i\theta} + a) \cdot
f \biggl( \frac{\bar a e^{i\theta} - b}{-\bar b e^{i\theta} +a} \biggr).
$$
These two formulas can be replaced by a single expression:
$$
\bigl( \pi_{\chi}(g) f \bigr)(e^{i\theta})
= (-\bar b e^{i\theta} + a)^{l+\epsilon} \cdot (-b e^{-i\theta} + \bar a)^{l-\epsilon}
\cdot f \biggl( \frac{\bar a e^{i\theta} - b}{-\bar b e^{i\theta} +a} \biggr).
$$
(Strictly speaking, this expression is ambiguous, since it involves raising
a complex number to a complex power.)

\subsection{Action of the Lie Algebra}

In this subsection we differentiate $(\pi_{\chi}, {\cal C}^{\infty}(S^1))$
and compute the actions of the Lie algebra $\mathfrak{su}(1,1)$ and its
complexification $\mathfrak{sl}(2,\BB C)$.
Note that every element of ${\cal C}^{\infty}(S^1)$ is a smooth vector.

We choose a basis of $\mathfrak{su}(1,1)$
$$
X= \begin{pmatrix} i & 0 \\ 0 & -i \end{pmatrix}, \qquad
Y= \begin{pmatrix} 0 & 1 \\ 1 & 0 \end{pmatrix}, \qquad
Z= \begin{pmatrix} 0 & i \\ -i & 0 \end{pmatrix}.
$$

Then
$$
\exp(tX)= \begin{pmatrix} e^{it} & 0 \\ 0 & e^{-it} \end{pmatrix}, \qquad
\exp(tY)= \begin{pmatrix} \cosh t & \sinh t \\ \sinh t & \cosh t \end{pmatrix},
$$
$$
\exp(tZ)= \begin{pmatrix} \cosh t & i\sinh t \\
-i\sinh t & \cosh t \end{pmatrix}.
$$
We have:
$$
\bigl( \pi_{\chi}(\exp(tX)) f \bigr)(e^{i\theta}) =
e^{2\epsilon it} \cdot f(e^{i(\theta-2t)}),
$$
\begin{multline*}
\bigl( \pi_{\chi}(X) f \bigr)(e^{i\theta})
= \frac{d}{dt} \bigl( \pi_{\chi}(\exp(tX)) f \bigr)(e^{i\theta}) \Bigr|_{t=0} \\
= 2 \Bigl( i\epsilon - \frac{d}{d\theta} \Bigr) f(e^{i\theta});
\end{multline*}
\begin{multline*}
\bigl( \pi_{\chi}(\exp(tY)) f \bigr)(e^{i\theta})
= (-\sinh t e^{i\theta} + \cosh t)^{l+\epsilon} \\
\times (-\sinh t e^{-i\theta} + \cosh t)^{l-\epsilon} \cdot
f \biggl( \frac{\cosh t e^{i\theta} - \sinh t}
{-\sinh t e^{i\theta} + \cosh t} \biggr),
\end{multline*}
\begin{multline*}
\bigl( \pi_{\chi}(Y) f \bigr)(e^{i\theta})
= \frac{d}{dt} \bigl( \pi_{\chi}(\exp(tY)) f \bigr)(e^{i\theta}) \Bigr|_{t=0} \\
= - \Bigl( (l+\epsilon)e^{i\theta} + (l-\epsilon)e^{-i\theta}
-2\sin\theta  \frac{d}{d\theta} \Bigr) f(e^{i\theta});
\end{multline*}
\begin{multline*}
\bigl( \pi_{\chi}(\exp(tZ)) f \bigr)(e^{i\theta})
= (i\sinh t e^{i\theta} + \cosh t)^{l+\epsilon} \\
\times (-i\sinh t e^{-i\theta} + \cosh t)^{l-\epsilon} \cdot
f \biggl( \frac{\cosh t e^{i\theta} - i\sinh t}
{i\sinh t e^{i\theta} + \cosh t} \biggr),
\end{multline*}
\begin{multline*}
\bigl( \pi_{\chi}(Z) f \bigr)(e^{i\theta})
= \frac{d}{dt} \bigl( \pi_{\chi}(\exp(tZ)) f \bigr)(e^{i\theta}) \Bigr|_{t=0} \\
= \Bigl( i(l+\epsilon)e^{i\theta} - i(l-\epsilon)e^{-i\theta}
-2\cos\theta  \frac{d}{d\theta} \Bigr) f(e^{i\theta}).
\end{multline*}

Finally, we rewrite our results in terms of the more familiar  basis of
$\mathfrak{sl}(2,\BB C) \simeq \BB C \otimes \mathfrak{su}(1,1)$
$$
E= \begin{pmatrix} 0 & 1 \\ 0 & 0 \end{pmatrix}, \qquad
F= \begin{pmatrix} 0 & 0 \\ 1 & 0 \end{pmatrix}, \qquad
H= \begin{pmatrix} 1 & 0 \\ 0 & -1 \end{pmatrix}.
$$
Since
\begin{equation}  \label{EFG-rel}
H=-iX, \qquad E=\frac12(Y-iZ), \qquad F=\frac12(Y+iZ),
\end{equation}
we obtain:
\begin{align*}
\pi_{\chi}(H) &= 2 \Bigl( i\frac{d}{d\theta} + \epsilon \Bigr), \\
\pi_{\chi}(E) &= e^{-i\theta} \Bigl( i\frac{d}{d\theta} - (l-\epsilon) \Bigr), \\
\pi_{\chi}(F) &= -e^{i\theta} \Bigl( i\frac{d}{d\theta} + (l+\epsilon) \Bigr).
\end{align*}
%\begin{align*}
%\bigl( \pi_{\chi}(H) f \bigr)(e^{i\theta})
%&= 2 \Bigl( \epsilon + i\frac{d}{d\theta} \Bigr) f(e^{i\theta}), \\
%\bigl( \pi_{\chi}(E) f \bigr)(e^{i\theta})
%&= e^{-i\theta} \Bigl( i\frac{d}{d\theta} - (l-\epsilon) \Bigr) f(e^{i\theta}), \\
%\bigl( \pi_{\chi}(F) f \bigr)(e^{i\theta})
%&= -e^{i\theta} \Bigl( i\frac{d}{d\theta} + (l+\epsilon) \Bigr) f(e^{i\theta}).
%\end{align*}

\subsection{The Harish-Chandra Module of $(\pi_{\chi}, {\cal C}^{\infty}(S^1))$}

As usual, we choose the diagonal subgroup
$$
K_{\BB R} = \left\{ k_t = \begin{pmatrix} e^{it} & 0 \\ 0 & e^{-it}
\end{pmatrix} ;\: t \in \BB R\,\right\} \simeq SO(2) \simeq U(1)
$$
as maximal compact subgroup of $SU(1,1)$. Then
$$
\bigl( \pi_{\chi}(k_t) f \bigr)(z) = e^{2\epsilon it} \cdot f(e^{-2it} \cdot z),
\qquad z \in S^1.
$$
The eigenvectors of this action are functions $z^n = e^{in\theta}$, $n \in \BB Z$,
$$
\pi_{\chi}(k_t) z^n = e^{2i(\epsilon-n)t} \cdot z^n.
$$
Since every smooth function on $S^1$ has a Fourier series expansion, we see that
$f(e^{i\theta}) \in {\cal C}^{\infty}(S^1)$ is $K_{\BB R}$-finite if and only if
it is a finite linear combination of $z^n = e^{in\theta}$, $n \in \BB Z$:
$$
{\cal C}^{\infty}(S^1)_{fini} = \bigoplus_{n \in \BB Z} \BB C z^n
= \bigoplus_{n \in \BB Z} \BB C e^{in\theta}.
$$
In particular, each representation $(\pi_{\chi},V_{\chi})$ is admissible.

Pick a basis of ${\cal C}^{\infty}(S^1)_{fini}$
$$
v_n = z^{-n} = e^{-in\theta}, \qquad n \in \BB Z,
$$
and let us find the actions of $\pi_{\chi}(H)$, $\pi_{\chi}(E)$, $\pi_{\chi}(F)$
in that basis. We have:
$$
\pi_{\chi}(H)v_n
= 2 \Bigl( i\frac{d}{d\theta} + \epsilon \Bigr) e^{-in\theta}
= 2(n+\epsilon) e^{-in\theta} = 2(n+\epsilon) v_n,
$$
\begin{multline*}
\pi_{\chi}(E)v_n
= e^{-i\theta} \Bigl( i\frac{d}{d\theta} - (l-\epsilon) \Bigr) e^{-in\theta} \\
= (n-l + \epsilon) e^{-i(n+1)\theta} = (n-l + \epsilon)v_{n+1},
\end{multline*}
\begin{multline*}
\pi_{\chi}(F)v_n
= -e^{i\theta} \Bigl( i\frac{d}{d\theta} + (l+\epsilon) \Bigr) e^{-in\theta} \\
= -(n+l + \epsilon) e^{-i(n-1)\theta} = -(n+l + \epsilon)v_{n-1}.
\end{multline*}
Thus we have proved:

\begin{lem}
The Lie algebra $\mathfrak{sl}(2,\BB C)$ acts on ${\cal C}^{\infty}(S^1)_{fini}$
as follows:
\begin{align*}
\pi_{\chi}(H)v_n &= 2(n+\epsilon) v_n, \\
\pi_{\chi}(E)v_n &= (n-l + \epsilon)v_{n+1}, \\
\pi_{\chi}(F)v_n &= -(n+l + \epsilon)v_{n-1}.
\end{align*}
\end{lem}

Note that the coefficients $(n \pm l + \epsilon)$ are never zero unless
$l+\epsilon \in \BB Z$. This means that the Harish-Chandra module
${\cal C}^{\infty}(S^1)_{fini}$ is irreducible if and only if
$l+\epsilon \notin \BB Z$. Since a representation is irreducible if and only
if its underlying Harish-Chandra module is irreducible, we have proved:

\begin{prop}
The representation $(\pi_{\chi}, {\cal C}^{\infty}(S^1))$ of $SU(1,1)$ is
irreducible if and only if $l+\epsilon \notin \BB Z$.
\end{prop}

In particular, for a generic parameter $\chi$,
$(\pi_{\chi}, {\cal C}^{\infty}(S^1))$ is irreducible.

\subsection{Decomposition of $(\pi_{\chi}, {\cal C}^{\infty}(S^1))$
into Irreducible Components}

In this subsection we study how the representation
$(\pi_{\chi}, {\cal C}^{\infty}(S^1))$ of $SU(1,1)$ decomposes into irreducible
components. We have seen that if $l+\epsilon \notin \BB Z$, then
$(\pi_{\chi}, {\cal C}^{\infty}(S^1))$ is irreducible.
Thus we study the case $l+\epsilon \in \BB Z$.

Note that
\begin{align*}
\pi_{\chi}(F)v_{-(l+\epsilon)} &= \pi_{\chi}(F) e^{i(l+\epsilon)\theta} = 0
\quad \text{and} \\
\pi_{\chi}(E)v_{l-\epsilon} &= \pi_{\chi}(E) e^{-i(l-\epsilon)\theta} = 0.
\end{align*}
Define $(\g g, K)$ submodules of ${\cal C}^{\infty}(S^1)_{fini}$
\begin{align*}
M^+ &= \text{ $\BB C$-span of }
\{ v_n;\: n \ge -(l+\epsilon) \} \quad \text{and} \\
M^- &= \text{ $\BB C$-span of } \{ v_n;\: n \le l-\epsilon \}.
\end{align*}
Here is a diagram illustrating modules $M^+$, $M^-$ and their weights:
\begin{align*}
M^+&: \qquad \begin{matrix}
[v_{-l-\epsilon} & & v_{1-l-\epsilon} & & v_{2-l-\epsilon} & & \dots  \\
-2l & & 2-2l & & 4-2l & & \dots \end{matrix}  \\
M^-&: \qquad \begin{matrix}
\dots & & v_{l-2-\epsilon} & & v_{l-1-\epsilon} & & v_{l-\epsilon}] \\
\dots & & 2l-4 & & 2l-2 & & 2l \end{matrix}
\end{align*}
Their closures in ${\cal C}^{\infty}(S^1)_{fini}$ are subrepresentations of
$(\pi_{\chi}, {\cal C}^{\infty}(S^1))$.

If $l < 0$, the submodules $M^+$ and $M^-$ are irreducible,
$M^+ \cap M^- = \{0\}$. In the notations of Section \ref{sl(2,R)-rep-section},
as a representation of $\mathfrak{sl}(2,\BB C)$, $M^+$ is the irreducible
lowest weight module $V_{-2l}$ of lowest weight $-2l$.
Similarly, as a representation of $\mathfrak{sl}(2,\BB C)$, $M^-$ is the
irreducible highest weight module $\bar V_{2l}$ of highest weight $2l$.
In this case, the complete list of proper
$(\g g, K)$ submodules of ${\cal C}^{\infty}(S^1)_{fini}$ is
$$
M^+, \quad M^-, \quad M^+ \oplus M^- \:
(\text{if $(l,\epsilon) \ne (-1/2,1/2)$}).
$$
If $(l,\epsilon) \ne (-1/2,1/2)$, the quotient module
$$
{\cal C}^{\infty}(S^1)_{fini} /(M^+ \oplus M^-)
$$
is irreducible, has dimension $-(2l+1)$, and will be denoted by $F_{|2(l+1)|}$,
which is consistent with notations of Section \ref{sl(2,R)-rep-section}
$$
F_{|2(l+1)|}: \qquad \begin{matrix} [v_{l+1-\epsilon} & & v_{l+2-\epsilon} & & \dots
& & v_{-l-2-\epsilon} & & v_{-l-1-\epsilon}]  \\
2l+2 & & 2l+4 & & \dots & & -2l-4 & & -2l-2 \end{matrix}
$$
We have the following exact sequence of $(\g g, K)$ modules
\begin{equation}  \label{exact_seq}
0 \to M^+ \oplus M^- \to {\cal C}^{\infty}(S^1)_{fini} \to F_{|2(l+1)|} \to 0.
\end{equation}
Note that this exact sequence does {\em not} split, i.e. $F_{|2(l+1)|}$
cannot be realized as a submodule of ${\cal C}^{\infty}(S^1)_{fini}$.
Taking closures in ${\cal C}^{\infty}(S^1)$ results in a similar exact sequence
of representations of $SU(1,1)$.

In the exceptional case $l=-1/2$, $\epsilon=1/2$ we have only two
representations:
\begin{equation}  \label{dir_sum}
{\cal C}^{\infty}(S^1)_{fini} = M^+ \oplus M^-
\end{equation}
and $M^+$, $M^-$ are irreducible.
The corresponding exact sequence of $(\g g, K)$ modules is
$$
0 \to M^+ \oplus M^- \to {\cal C}^{\infty}(S^1)_{fini} \to 0.
$$

If $l \ge 0$, then $M^+$ and $M^-$ have a non-trivial
intersection which is irreducible, has dimension $2l+1$, and will be denoted
by $F_{2l}$, which is consistent with notations of
Section \ref{sl(2,R)-rep-section}:
$$
F_{2l} = M^+ \cap M^-: \qquad \begin{matrix}
[v_{-l-\epsilon} & & v_{-l+1-\epsilon} & & \dots & & v_{l-1-\epsilon} & & v_{l-\epsilon}]  \\
-2l & & -2l+2 & & \dots & & 2l-2 & & 2l \end{matrix}
$$
The complete list of proper $(\g g, K)$ submodules of
${\cal C}^{\infty}(S^1)_{fini}$ is
$$
F_{2l}, \quad M^+, \quad M^-.
$$
The quotient module
$$
{\cal C}^{\infty}(S^1)_{fini} /F_{2l}
$$
is a direct sum of two irreducible modules $M^+/F_{2l}$ and $M^-/F_{2l}$.
As representations of $\mathfrak{sl}(2,\BB C)$, $M^+/F_{2l}$ and $M^-/F_{2l}$
are respectively the irreducible lowest and highest weight modules $V_{2l+1}$
of lowest weight $2l+1$ and $\bar V_{-(2l+1)}$ of highest weight $-(2l+1)$.
We have the following exact sequence of $(\g g, K)$ modules
$$
0 \to F_{2l} \to {\cal C}^{\infty}(S^1)_{fini} \to (M^+/F_{2l}) \oplus (M^-/F_{2l})
\to 0.
$$
Similarly to (\ref{exact_seq}), this sequence does not split.
Taking closures in ${\cal C}^{\infty}(S^1)$ results in a similar exact sequence
of representations of $SU(1,1)$.

As a consequence of this discussion we obtain:

\begin{prop}
Each representation $(\pi_{\chi},{\cal C}^{\infty}(S^1))$ is admissible and has
finite length.
\end{prop}

We conclude this subsection with the following result:

\begin{thm}  \label{SU(1,1)-subrep-thm}
Every irreducible representation of $SU(1,1)$ is infinitesimally equivalent
to an irreducible subrepresentation of $(\pi_{\chi}, {\cal C}^{\infty}(S^1))$,
for some $\chi$.
\end{thm}

\begin{proof}
Since $K \simeq \BB C^{\times}$ is connected, the $K$-action on
a $(\g g, K)$ module is completely determined by the $\g g$-action.
We will show later (Corollary \ref{condition-satisfied}) that if $(\pi,V)$
is an admissible irreducible representation of $SU(1,1)$, then the action of
$\mathfrak{sl}(2,\BB C)$ on $V_{fini}$ satisfies Condition 2 of
Corollary \ref{equiv-ond-cor}. Recall that all irreducible
$\mathfrak{sl}(2,\BB C)$-modules satisfying one of the two equivalent
conditions of Corollary \ref{equiv-ond-cor} were classified in
Theorem \ref{classification-thm}.
Note that of these irreducible $\mathfrak{sl}(2,\BB C)$-modules only those
with $\pi(H)$ having integer eigenvalues can possibly arise from
$(\g g, K)$ modules (homework).
On the other hand, all of these arise as submodules of
${\cal C}^{\infty}(S^1)_{fini}$ for appropriate values of $\chi$.
Indeed, the finite-dimensional $\mathfrak{sl}(2,\BB C)$-modules $F_d$ appear
when $l=d/2$, $l+\epsilon \in \BB Z$.
The irreducible lowest weight modules $V_{\lambda}$ and highest weight modules
$\bar V_{\lambda}$, $\lambda \in \BB Z$, appear as $M^+$ with $l=-\lambda/2$ and
$M^-$ with $l=\lambda/2$, $l+\epsilon \in \BB Z$.
Finally, the irreducible modules $P(\lambda,\mu)$, $\lambda \in \BB Z$,
appear as ${\cal C}^{\infty}(S^1)_{fini}$ when $l+\epsilon \notin \BB Z$
(homework).
\end{proof}

The following corollary is a precise way of saying when an irreducible
representation of $\mathfrak{su}(1,1)$ ``lifts'' to a representation of
$SU(1,1)$.

\begin{cor}
An irreducible $\mathfrak{sl}(2,\BB C)$-module satisfying one of the two
equivalent conditions of Corollary \ref{equiv-ond-cor} appears as
$\mathfrak{sl}(2,\BB C)$ acting on the space of $SO(2)$-finite vectors of some
admissible irreducible representation of $SU(1,1)$ if and only if $H$
has integer eigenvalues.
\end{cor}

We conclude this subsection with a statement concerning isomorphisms
between various representations $(\pi_{\chi},{\cal C}^{\infty}(S^1))$.

\begin{prop}
Consider two parameters $\chi_1=(l_1,\epsilon_1)$ and $\chi_2=(l_2,\epsilon_2)$.
Then $(\pi_{\chi_1}, {\cal C}^{\infty}(S^1))$ and
$(\pi_{\chi_2}, {\cal C}^{\infty}(S^1))$ are isomorphic as representations
if and only if they are infinitesimally equivalent,
which happens if and only if $\chi_1=\chi_2$ or
$$
\epsilon_1=\epsilon_2, \qquad l_1=-1-l_2 \quad \text{and} \quad 
l_1+\epsilon_1 \notin \BB Z.
$$
\end{prop}

If $l_1+\epsilon_1 \in \BB Z$, let $\chi_2=(\epsilon_1,-1-l_1)$
so that $l_1=-1-l_2$, and assume that $l_1 \ne -1/2$ so that $l_1 \ne l_2$.
Then, as we saw at the beginning of the subsection,
$(\pi_{\chi_1}, {\cal C}^{\infty}(S^1))$ and $(\pi_{\chi_2}, {\cal C}^{\infty}(S^1))$
are reducible, have isomorphic irreducible components,
but not isomorphic nor infinitesimally equivalent.

\begin{proof}
If $l_1+\epsilon_1 \in \BB Z$ the proposition follows from the decomposition
of $(\pi_{\chi_1},{\cal C}^{\infty}(S^1))$ into irreducible components that we did
at the beginning of the subsection.
If $l_1+\epsilon_1 \notin \BB Z$ the proposition can be deduced from
the argument showing that the irreducible modules $P(\lambda,\mu)$,
$\lambda \in \BB Z$, appear among $(\pi_{\chi}, {\cal C}^{\infty}(S^1)_{fini})$
for appropriate choices of $\chi$ and the criterion for isomorphism between
two irreducible $\mathfrak{sl}(2,\BB C)$-modules $P(\lambda,\mu)$ and
$P(\lambda',\mu')$ (Proposition \ref{P-class}).
\end{proof}

\subsection{Unitary Representations of $SU(1,1)$}

Let us return to the case $l+\epsilon \in \BB Z$, $l<0$.
Let ${\cal D}_l^-$ denote the closure of $M^-$ in ${\cal C}^{\infty}(S^1)$;
it contains functions
$$
e^{-i(l-\epsilon)\theta} \cdot f(e^{i\theta}), \quad
\begin{matrix} \text{$f(z)$ is holomorphic on a} \\
\text{neighborhood of $\{z \in \BB C;\: |z| \le 1 \}$.} \end{matrix}
$$
Similarly, let ${\cal D}_l^+$ denote the closure of $M^+$ in
${\cal C}^{\infty}(S^1)$; it contains functions
$$
e^{i(l+\epsilon)\theta} \cdot f(e^{i\theta}), \quad
\begin{matrix} \text{$f(z)$ is antiholomorphic on a} \\
\text{neighborhood of $\{z \in \BB C;\: |z| \le 1 \}$.} \end{matrix}
$$
If $l \le -1$, the repserentations $(\pi_{\chi}, {\cal D}_l^-)$ and
$(\pi_{\chi}, {\cal D}_l^+)$ are called respectively the holomorphic and
antiholomorphic discrete series.
These names reflect the fact that historically they were first realized in
the spaces of holomorphic and antiholomorphic functions.
If $l=-1/2$, the representations $(\pi_{\chi}, {\cal D}_l^-)$ and
$(\pi_{\chi}, {\cal D}_l^+)$ are called the limits of discrete series.

The following is a complete list of irreducible unitary representations of
$SU(1,1)$ (or, more precisely, representations infinitesimally equivalent to
irreducible unitary ones):
\begin{itemize}
\item
The trivial one-dimensional representation;
\item
The holomorphic and antiholomorphic discrete series 
$(\pi_{\chi}, {\cal D}_l^-)$ and $(\pi_{\chi}, {\cal D}_l^+)$,
$l=-1,-3/2,-2,\dots$, $l+\epsilon \in \BB Z$;
\item
The limits of discrete series $(\pi_{\chi}, {\cal D}_{-1/2}^-)$ and
$(\pi_{\chi}, {\cal D}_{-1/2}^+)$, $l=-1/2$, $\epsilon=1/2$;
\item
The spherical unitary principal series $(\pi_{\chi}, {\cal C}^{\infty}(S^1))$,
$l=-1/2+i\lambda$, $\lambda \in \BB R$, $\epsilon=0$;
\item
The nonspherical unitary principal series $(\pi_{\chi}, {\cal C}^{\infty}(S^1))$,
$l=-1/2+i\lambda$, $\lambda \in \BB R^{\times}$, $\epsilon=1/2$;
\item
The complementary series $(\pi_{\chi}, {\cal C}^{\infty}(S^1))$,
$l \in (-1,-1/2) \cup (-1/2,0)$, $\epsilon=0$.
\end{itemize}
Note that in the case of nonspherical unitary principal series
we exclude the value $\chi=(-1/2,1/2)$ because in this case
$$
(\pi_{(-1/2,1/2)}, {\cal C}^{\infty}(S^1)) = 
(\pi_{(-1/2,1/2)}, {\cal D}_{-1/2}^-) \oplus (\pi_{(-1/2,1/2)}, {\cal D}_{-1/2}^+).
$$
In the case of complementary series we exclude $l=-1/2$ because
$(\pi_{(-1/2,0)}, {\cal C}^{\infty}(S^1))$
has already appeared in the spherical unitary principal series.

Of these unitary representations, the following appear in the decomposition
of $L^2(SU(1,1))$ into irreducible components:
\begin{itemize}
\item
The holomorphic and antiholomorphic discrete series 
$(\pi_{\chi}, {\cal D}_l^-)$ and $(\pi_{\chi}, {\cal D}_l^+)$,
$l=-1,-3/2,-2,\dots$, $l+\epsilon \in \BB Z$;
\item
The spherical unitary principal series $(\pi_{\chi}, {\cal C}^{\infty}(S^1))$,
$l=-1/2+i\lambda$, $\lambda \in \BB R$, $\epsilon=0$;
\item
The nonspherical unitary principal series $(\pi_{\chi}, {\cal C}^{\infty}(S^1))$,
$l=-1/2+i\lambda$, $\lambda \in \BB R^{\times}$, $\epsilon=1/2$.
\end{itemize}
The unitary principal series are sometimes called continuous series, so that
$L^2(SU(1,1))$ has a discrete component and a continuous component.

Let us show some examples of invariant inner products.
The first proposition covers the spherical and nonspherical unitary principal
series as well as the limits of discrete series cases.

\begin{prop}
Suppose that $\re l =-1/2$, then $(\pi_{\chi}, {\cal C}^{\infty}(S^1))$ has an
$SU(1,1)$-invariant inner product
$$
(f_1,f_2) = \frac1{2\pi} \int_0^{2\pi}
f_1(e^{i\theta}) \cdot \overline{f_2(e^{i\theta})} \,d\theta,
\qquad f_1,f_2 \in {\cal C}^{\infty}(S^1),
$$
so that $\{ \dots, v_{-1}, v_0, v_1, v_2, \dots \}$ is an orthonormal
pre-Hilbert space basis of ${\cal C}^{\infty}(S^1)$.
Taking the closure of ${\cal C}^{\infty}(S^1)$ with respect to this inner product
produces an infinitesimally equivalent unitary representation in a Hilbert
space.
\end{prop}

\begin{proof}
We need to show that
$$
(\pi_{\chi}(g)f_1,\pi_{\chi}(g)f_2) = (f_1,f_2), \qquad \forall g \in SU(1,1).
$$
Since $SU(1,1)$ is connected, it is sufficient to prove an infinitesimal
version of this property:
$$
(\pi_{\chi}(X)v_m,v_n) + (v_m,\pi_{\chi}(X)v_n) =0, \quad
\forall X \in \mathfrak{su}(1,1), \: m,n \in \BB Z,
$$
or
$$
(\pi_{\chi}(X)v_m,v_n) + (v_m,\pi_{\chi}(\overline{X})v_n) =0, \quad
\forall X \in \mathfrak{sl}(2,\BB C), \: m,n \in \BB Z,
$$
where $\overline{X}$ denotes the complex conjugate of $X$ relative
$\mathfrak{su}(1,1) \subset \mathfrak{sl}(2,\BB C)$.
From (\ref{EFG-rel}), $\overline{H}=-H$ and $\overline{E}=F$.
Thus, it is sufficient to verify
$$
(\pi_{\chi}(H)v_m,v_n) - (v_m,\pi_{\chi}(H)v_n) =0, \qquad \forall m,n \in \BB Z,
$$
and
$$
(\pi_{\chi}(E)v_m,v_n) + (v_m,\pi_{\chi}(F)v_n) =0, \qquad \forall m,n \in \BB Z.
$$
The first equation amounts to
$$
2(m+\epsilon)(v_m,v_n) - 2(n+\epsilon)(v_m,v_n) =0, \qquad \forall m,n \in \BB Z,
$$
which is certainly true.
And the second equation amounts to
$$
(m-l+\epsilon)(v_{m+1},v_n) - \overline{(n+l+\epsilon)}(v_m,v_{n-1}) =0,
\quad \forall m,n \in \BB Z.
$$
The two inner products are non-zero if and only if $m=n-1$,
in which case we get
$$
(n-1-l+\epsilon) - (n+ \bar l +\epsilon) = -(1 + l + \bar l) =0,
$$
since $\re l = -1/2$.
\end{proof}

The second proposition covers the complementary unitary series case.

\begin{prop}
Suppose that $l \in (-1,0)$ and $\epsilon=0$, then
$(\pi_{\chi}, {\cal C}^{\infty}(S^1)_{fini})$ has an
$\mathfrak{su}(1,1)$-invariant inner product defined so that
$\{ \dots, v_{-1}, v_0, v_1, v_2, \dots \}$ form an orthogonal basis of
${\cal C}^{\infty}(S^1)_{fini}$ and
$$
(v_n,v_n) =c_n, \qquad n \in \BB Z,
$$
where the coefficients $c_n$ are defined by the rule
$$
c_0=1, \qquad c_{n+1} = \frac{n+l+1}{n-l} c_n, \qquad n \in \BB Z.
$$
Taking the closure of ${\cal C}^{\infty}(S^1)_{fini}$ with respect to this
inner product produces an infinitesimally equivalent unitary representation
in a Hilbert space.
\end{prop}

\begin{proof}
Since $l \in (-1,0)$, each $c_n>0$ and the inner product is positive definite.

As in the previous proof, it is sufficient to show that
$$
(\pi_{\chi}(H)v_m,v_n) - (v_m,\pi_{\chi}(H)v_n) =0, \qquad \forall m,n \in \BB Z,
$$
and
$$
(\pi_{\chi}(E)v_m,v_n) + (v_m,\pi_{\chi}(F)v_n) =0, \qquad \forall m,n \in \BB Z.
$$
The first equation amounts to
$$
2m(v_m,v_n) - 2n(v_m,v_n) =0, \qquad \forall m,n \in \BB Z,
$$
which is certainly true.
And the second equation amounts to
$$
(m-l)(v_{m+1},v_n) - (n+l)(v_m,v_{n-1}) =0,
\qquad \forall m,n \in \BB Z.
$$
The two inner products are non-zero if and only if $m=n-1$,
in which case we get
$$
(m-l)c_{m+1} - (m+1+l)c_m = 0 \quad \Longleftrightarrow \quad
c_{m+1} = \frac{m+l+1}{m-l} c_m.
$$
\end{proof}

In the cases of holomorphic and antiholomorphic discrete series a similar
argument can be used to prove the following:

\begin{prop}
Suppose that $l+\epsilon \in \BB Z$, $l < 0$ and $(l,\epsilon) \ne (-1/2,1/2)$.
Then $(\pi_{\chi}, M^+)$ has an $\mathfrak{su}(1,1)$-invariant inner product
defined so that
$\{ v_{-l-\epsilon}, v_{1-l-\epsilon}, v_{2-l-\epsilon}, v_{3-l-\epsilon}, \dots \}$
form an orthogonal basis of $M^+$ and
$$
(v_n,v_n)^+ = c_n^+, \qquad n \in \BB Z, \quad n \ge -l-\epsilon,
$$
where the coefficients $c_n^+$ are defined by the rule
$$
c_{-l-\epsilon}^+=1, \quad c_{n+1}^+ = \frac{n+l+\epsilon+1}{n-l+\epsilon} c_n^+,
\qquad n \ge -l-\epsilon.
$$
Similarly, $(\pi_{\chi}, M^-)$ has an $\mathfrak{su}(1,1)$-invariant inner
product defined so that
$\{ v_{l-\epsilon}, v_{l-1-\epsilon}, v_{l-2-\epsilon}, v_{l-3-\epsilon}, \dots \}$
form an orthogonal basis of $M^-$ and
$$
(v_n,v_n)^- = c_n^-, \qquad n \in \BB Z, \quad n \le l-\epsilon,
$$
where the coefficients $c_n^-$ are defined by the rule
$$
c_{l-\epsilon}^-=1, \quad c_{n-1}^- = \frac{n-l+\epsilon-1}{n+l+\epsilon} c_n^-,
\qquad n \le l-\epsilon.
$$
Taking the closures of $M^+$ and $M^-$ with respect to their inner products
produces infinitesimally equivalent unitary representations in Hilbert
spaces.
\end{prop}

We know from homework that the only finite-dimensional unitary representations
of $SU(1,1)$ are the trivial representations.
To show that there are no more irreducible unitary representations of $SU(1,1)$
one can use Theorem \ref{SU(1,1)-subrep-thm} to reduce the problem to finding
a (possibly indefinite) $\mathfrak{su}(1,1)$-invariant sesquilinear product on
$(\pi_{\chi}, {\cal C}^{\infty}(S^1)_{fini})$ with $l+\epsilon \notin \BB Z$
(the irreducible case).
Then argue that $\{ \dots, v_{-1}, v_0, v_1, v_2, \dots \}$ must form an
orthogonal basis of ${\cal C}^{\infty}(S^1)_{fini}$. Hence
$$
(v_n,v_n) =c_n, \qquad n \in \BB Z,
$$
for some choice of coefficients $c_n$.
By the  $\mathfrak{su}(1,1)$-invariance, these coefficients must satisfy
$$
(n-l+\epsilon)c_{n+1} = (n+ \bar l + \epsilon +1) c_n, \qquad n \in \BB Z.
$$
This determines the product $(\,.\,,\,.\,)$ up to proportionality.
Finally, one can determine when the proportionality coefficient can be
chosen so that the product is positive definite.

\section{Notes on Schur's Lemma}

\subsection{Classical Version}

Recall the ``classical'' version of Schur's Lemma:

\begin{lem}
Suppose $(\pi,V)$ is an irreducible representation of a group $G_{\BB R}$
(or a Lie algebra $\g g_{\BB R}$) on a complex finite-dimensional vector space
$V$. If $T:V \to V$ is a linear map such that $T$ commutes with all
$\pi(g)$, $g \in G_{\BB R}$,
(respectively with all $\pi(X)$, $X \in \g g_{\BB R}$),
%$$
%\pi(g) \circ T = T \circ \pi(g), \qquad \forall g \in G_{\BB R},
%$$
then $T$ is multiplication by a scalar.
\end{lem}

\begin{proof}
Let $\lambda \in \BB C$ be an eigenvalue of $T$. Then
$T - \lambda \cdot Id_V$ also commutes with $\pi(g)$, for all $g \in G_{\BB R}$,
(respectively with $\pi(X)$, for all $X \in \g g_{\BB R}$).
Hence the $\lambda$-eigenspace of $T$ is a non-zero invariant subspace of $V$.
By the irreduciblity of $V$, this $\lambda$-eigenspace has to be all of $V$.
Therefore, $T$ is $\lambda \cdot Id_V$.
\end{proof}

Note that the proof depends on the existence of eigenvalues and eigenvectors
of the linear map $T:V \to V$. But if $\dim V = \infty$,
$T$ might not have any eigenvectors (e.g. ``shift'' operator).
So $\dim V < \infty$ is essential for this proof.
Here we discuss two infinite-dimensional analogues of Schur's Lemma.

\subsection{Schur's Lemma for Unitary Representations}

Here is a version of Schur's Lemma for unitary representations reproduced
from \cite{Kn1}.

\begin{prop}
A unitary representation $(\pi,V)$ of a topological group $G$ on a Hilbert
space $V$ is irreducible if and only if the only bounded linear operators
on $V$ commuting with all $\pi(g)$, $g \in G$, are the scalar operators.
\end{prop}

\begin{proof}
If $V$ is reducible with $U \subset V$ a non-trivial closed invariant subspace,
then the orthogonal projection on $U$ is a non-scalar bounded linear operator
on $V$ commuting with all $\pi(g)$'s.

Conversely, let $T:V \to V$ be a non-scalar bounded linear operator commuting
with all $\pi(g)$'s. Since $(\pi,V)$ is unitary, all $\pi(g)$, $g \in G$,
commute with the adjoint operator $T^*$ as well as self-adjoint operators
$$
A= \frac12 (T+T^*) \quad \text{and} \quad B= \frac1{2i}(T-T^*).
$$
Since $T=A+iB$, at least one of $A$, $B$ is not scalar, say $A$
for concreteness.
Using the Spectral Theorem, we can get a projection $P$ which commutes with
all operators commuting with $A$. Intuitively this involves taking a proper
subset $S$ of $\sigma(A)$ -- the spectrum of $A$ -- and letting $P$ be the
orthogonal projection onto the ``direct sum of eigenspaces with eigenvalues
in $S$''. In particular, $P$ commutes with all $\pi(g)$, $g \in G$.
Then $V= \ker P \oplus \im P$.
\end{proof}

\subsection{Dixmier's Lemma}

In this subsection we follow \cite{Wal1}.

\begin{lem}  \label{spectrum-lem}
Let $V$ be a countable dimensional vector space over $\BB C$.
If $T \in \operatorname{End}(V)$, then there exists a scalar $c \in \BB C$
such that $T - c \cdot Id_V$ is not invertible on $V$.
\end{lem}

\begin{proof}
Suppose $T - c \cdot Id_V$ is invertible for all $c \in \BB C$.
Then $P(T)$ is invertible for all polynomials $P \in \BB C[x]$, $P \ne 0$.
Thus, if $R=P/Q$ is a rational function with $P,Q \in \BB C[x]$,
we can define $R(T)$ by $P(T) \cdot (Q(T))^{-1}$.
This rule defines a map $\BB C(x) \to \operatorname{End}(V)$.
If $v \in V$, $v \ne 0$, and $R \in \BB C(x)$, $R \ne 0$, then
$R(T)v \ne 0$. Hence the map
$$
\BB C(x) \to V, \quad R \mapsto R(T)v,
$$
is injective. Since $\BB C(x)$ has uncountable dimension over $\BB C$
(for example, the rational functions $(x-a)^{-1}$, $a \in \BB C$,
are linearly independent), we get a contradiction.
\end{proof}

\begin{ex}
Let $V$ be a countable dimensional vector space with basis $\{e_n\}$,
$n \in \BB Z$. (Thus $V = \bigoplus_{n \in \BB Z} \BB C e_n$.)
Consider the shift operator $T: V \to V$ defined by $T(e_n)=e_{n+1}$.
Then $T - c \cdot Id_V$ is not invertible on $V$ whenever $c \in \BB C^{\times}$.
We have: $\ker(T - c \cdot Id_V)=0$ (i.e. $T$ has no eigenvectors) and
$\im(T - c \cdot Id_V) \subsetneq V$ for all $c \in \BB C^{\times}$.
For example, if $c=1$,
$$
\im(T - Id_V) = \Bigl\{ v=\sum_n a_ne_n \in V ;\: \sum_n a_n=0 \Bigr\}.
$$
\end{ex}

\begin{df}
Let $V$ be a vector space over $\BB C$ and $S \subset \operatorname{End}(V)$
a subset. Then $S$ acts {\em irreducibly} if whenever $W \subset V$ is a
subspace such that $SW \subset W$, then $W=V$ or $W=\{0\}$.
\end{df}

\begin{lem}[Dixmier]
Suppose $V$ is a vector space over $\BB C$ of countable dimension and that
$S \subset \operatorname{End}(V)$ acts irreducibly.
If $T \in \operatorname{End}(V)$ commutes with every element of $S$,
then $T$ is a scalar multiple of the identity operator.
\end{lem}

\begin{proof}
By Lemma \ref{spectrum-lem}, there exists a $c \in \BB C$ such that
$T - c \cdot Id_V$ is not invertible on $V$.
Then $\ker(T - c \cdot Id_V)$ and $\im(T - c \cdot Id_V)$ are preserved by $S$,
hence these spaces are either $\{0\}$ or $V$.
If $\ker(T - c \cdot Id_V)=\{0\}$ and $\im(T - c \cdot Id_V)=V$,
then $T$ is invertible (there is no topology or continuity involved here).
Hence $\ker(T - c \cdot Id_V)=V$ or $\im(T - c \cdot Id_V)=\{0\}$ and
in both cases $T=c \cdot Id_V$.
\end{proof}

Note that if $V$ is an admissible $(\g g, K)$-module, then $V$ has countable
dimension, since the set $\hat K_{\BB R}$ is countable.
The following lemma essentially tells us that we can drop the admissibility
assumption if $V$ is irreducible.

\begin{lem}
Let $V$ be an irreducible $(\g g, K)$-module, then $V$ has countable dimension.
\end{lem}

\begin{proof}
Pick a $v \in V$, $v \ne 0$, and define
$$
W_v = \text{$\BB C$-span of $\{ kv ;\: k \in K \}$}.
$$
Since the action of $K$ on $V$ is locally finite, $\dim W_v < \infty$.
Then ${\cal U}(\g g)W_v$ is a $(\g g, K)$-invariant subspace of $V$,
and by the Poincar\'e-Birkhoff-Witt Theorem the dimension of
${\cal U}(\g g)W_v$ is countable.
Since $V$ is irreducible, $V={\cal U}(\g g)W_v$ and $V$ has countable dimension.
\end{proof}

\begin{cor}
Let $V$ be an irreducible $(\g g,K)$-module, then
$$
\operatorname{Hom}_{(\g g,K)}(V,V)=\BB C \cdot Id_V.
$$
\end{cor}

\begin{proof}
Dixmier's Lemma applies, since $V$ has countable dimension.
\end{proof}

\subsection{Infinitesimal Characters}

Let ${\cal ZU}(\g g)$ denote the center of ${\cal U}(\g g)$ and introduce
$$
{\cal Z}_{{\cal U}(\g g)}(G_{\BB R}) =_{\text{def}} \{ a \in {\cal U}(\g g);\:
Ad(g)a=a, \forall g \in G_{\BB R} \}.
$$
Clearly, ${\cal Z}_{{\cal U}(\g g)}(G_{\BB R}) \subset {\cal ZU}(\g g)$ and
$$
{\cal Z}_{{\cal U}(\g g)}(G_{\BB R}) = {\cal ZU}(\g g) \quad
\text{if $G_{\BB R}$ is connected}.
$$
%if $G_{\BB R}$ is connected,
%${\cal Z}_{{\cal U}(\g g)}(G_{\BB R}) = {\cal ZU}(\g g)$.
Applying Dixmier's Lemma we obtain:

\begin{lem}
Let $V$ be an irreducible $(\g g, K)$-module. Then each
$z \in {\cal Z}_{{\cal U}(\g g)}(G_{\BB R})$ acts on $V$ by a scalar multiple of
identity.
\end{lem}

\begin{cor}  \label{inf-char-cor}
Let $(\pi,V)$ be an irreducible admissible representation of $G_{\BB R}$,
then each $z \in {\cal Z}_{{\cal U}(\g g)}(G_{\BB R})$ acts on $V^{\infty}$
by a scalar multiple of identity.
\end{cor}

\begin{proof}
Apply the above lemma to $V_{fini}$, which is dense in $V^{\infty}$.
\end{proof}

\begin{cor}  \label{condition-satisfied}
Let $(\pi,V)$ be an irreducible admissible representation of $SU(1,1)$,
then the Casimir element acts on $V_{fini}$ by multiplication by a scalar.
In particular, the action of $\mathfrak{sl}(2,\BB C)$ on $V_{fini}$ satisfies
condition 2 of Corollary \ref{equiv-ond-cor}.
\end{cor}

\begin{df}
If $(\pi,V)$ is a representation of $G_{\BB R}$ and $\chi$ is an algebra
homomorphism of ${\cal Z}_{{\cal U}(\g g)}(G_{\BB R})$ into $\BB C$ such that
$$
\pi(z)v=\chi(z)v, \qquad \forall z \in {\cal Z}_{{\cal U}(\g g)}(G_{\BB R}),\: 
v \in V^{\infty},
$$
then $\chi$ is called the {\em infinitesimal character} of $(\pi,V)$.
\end{df}

By Corollary \ref{inf-char-cor}, every admissible irreducible representation
of $G_{\BB R}$ has an infinitesimal character.

\noindent
\underline{Caution}: Even if $G_{\BB R}$ is connected, some non-equivalent
irreducible representations may have the same infinitesimal character.
For example, representations $(\pi_{\chi}, V_{\chi})$ of $SU(1,1)$ always have
infinitesimal character. Some of these are reducible, and their irreducible
components have the same infinitesimal character.

\section{Iwasawa and $KAK$ Decompositions}

In this section we describe Iwasawa and $KAK$ decompositions. 
For proofs and details see the book \cite{Kn2}.
We assume that $G_{\BB R}$ is a closed linear connected semisimple
real Lie group.

\subsection{Iwasawa Decomposition} \label{Iwasawa-subsection}

An important class of representations are the so-called principal series
representations. The Subrepresentation Theorem says that every irreducible
admissible representation of $G_{\BB R}$ is infinitesimally equivalent to
a subrepresentation of some principal series representation.
To construct these we need Iwasawa decomposition, which is sometimes
called the $KAN$ decomposition.

Recall the Cartan involution and the Cartan decomposition:
$$
\theta: \g g_{\BB R} \to \g g_{\BB R}, \qquad \theta^2=Id,
$$
$$
\g g_{\BB R} = \g k_{\BB R} \oplus \g p_{\BB R}, \qquad
\theta \bigr|_{\g k_{\BB R}} = Id, \qquad \theta \bigr|_{\g p_{\BB R}} = -Id,
$$
$$
[\g k_{\BB R}, \g k_{\BB R}] \subset \g k_{\BB R}, \qquad
[\g p_{\BB R}, \g p_{\BB R}] \subset \g k_{\BB R}, \qquad
[\g k_{\BB R}, \g p_{\BB R}] \subset \g p_{\BB R}.
$$
Let $K_{\BB R} \subset G_{\BB R}$ be the connected Lie subgroup with Lie algebra
$\g k_{\BB R}$, then $K_{\BB R}$ is a maximal compact subgroup of $G_{\BB R}$
(Corollary \ref{max-compact-cor}),
and it acts on $\g g_{\BB R}$ and $\g p_{\BB R}$ by $Ad$.
Every $X \in \g p_{\BB R} \subset \g g_{\BB R}$ is diagonalizable
with real eigenvalues, hence the action of $\operatorname{ad}X$ on
$\g g_{\BB R}$ is diagonalizable and has real eigenvalues as well.

Choose a maximal abelian subspace $\g a_{\BB R} \subset \g p_{\BB R}$.
Clearly, $\dim \g a_{\BB R} \ge 1$, $\g a_{\BB R}$ is $\theta$-invariant and,
for all $X \in \g a_{\BB R}$, $\operatorname{ad}X$ acts on $\g g_{\BB R}$
semisimply with real eigenvalues (homework). Hence
$$
Z_{\g g_{\BB R}}(\g a_{\BB R}) = \{ Y \in \g g_{\BB R} ;\: [Y,X]=0, \:
\forall X \in \g a_{\BB R} \} = \g m_{\BB R} \oplus \g a_{\BB R},
$$
where
$$
\g m_{\BB R} = Z_{\g k_{\BB R}}(\g a_{\BB R}) = \{ Y \in \g k_{\BB R} ;\: [Y,X]=0, \:
\forall X \in \g a_{\BB R} \}.
$$
Let
\begin{align*}
A_{\BB R} &=_{\text{def}} \text{connected subgroup of $G_{\BB R}$ with
Lie algebra $\g a_{\BB R}$}, \\
M_{\BB R} &=_{\text{def}} Z_{K_{\BB R}}(A_{\BB R}) = Z_{K_{\BB R}}(\g a_{\BB R}).
\end{align*}
The group $M_{\BB R}$ need not be connected, and so
$M_{\BB R} \ne \exp \g m_{\BB R}$ in general, but its Lie algebra is $\g m_{\BB R}$.
Complexify:
$$
\g a =_{\text{def}} \BB C \otimes \g a_{\BB R}, \qquad
\g m =_{\text{def}} \BB C \otimes \g m_{\BB R}, \qquad
\g a, \g m \subset \g g= \BB C \otimes \g g_{\BB R}.
$$

Consider the adjoint action of $\g a$ on $\g g$ (we already know that
it is semisimple). We have:
$$
Z_{\g g}(\g a) = \g m \oplus \g a
\quad \text{and} \quad
\g g = \g m \oplus \g a \oplus \Bigl(\bigoplus_{\alpha \in \Phi} \g g^{\alpha}\Bigr),
$$
where $\Phi = \Phi(\g g, \g a) \subset \g a^* \setminus \{0\}$ is the
``reduction of relative roots of the pair $(\g g, \g a)$'' and
$$
\g g^{\alpha} = \{ Y \in \g g; (\operatorname{ad} X) Y = \alpha(X) \cdot Y,\:
\forall X \in \g a \}.
$$
(Caution: $\g a$ need not be a Cartan subalgebra of $\g g$, but it is a
somewhat similar object.)
Both $\theta$ and complex conjugation (with respect to $\g g_{\BB R}$)
act on $\g g$. We know that $\theta = -Id$ on $\g a$ and the roots are
real-valued on $\g a_{\BB R}$. Therefore, for all $\alpha \in \Phi$,
$$
\theta\alpha=-\alpha, \qquad
%\bar\alpha=\alpha, \qquad
\theta(\g g^{\alpha})= \g g^{-\alpha}, \qquad
\overline{\g g^{\alpha}}=\g g^{\alpha}.
$$
All properties of root space decomposition of a complex semisimple Lie
algebra carry over to this setting, except
\begin{enumerate}
\item
$\dim \g g^{\alpha}$ can be strictly bigger than 1;
\item
It is possible to have non-trivially proportional roots
$\pm\alpha$ and $\pm2\alpha$ only.
\end{enumerate}
(This is why non-reduced root systems are important.)

A {\em Weyl chamber} is a connected component of
$$
\g a_{\BB R} \setminus \bigcup_{\alpha \in \Phi(\g g, \g a)}
\{ X \in \g a_{\BB R} ;\: \alpha(X)=0 \}
$$
($\g a_{\BB R}$ without a finite number of hyperplanes).
Each Weyl chamber is a convex cone. Choose one and call it the
``highest Weyl chamber''.
Let $\Phi^+ \subset \Phi(\g g,\g a)$ consist of those roots
$\alpha \in \Phi(\g g,\g a)$ that are positive on the highest Weyl chamber.
Then the highest Weyl chamber is precisely
$$
\{ X \in \g a_{\BB R} ;\: \alpha(X)>0 \: \forall \alpha \in \Phi^+ \}
$$
and
$$
\Phi(\g g,\g a) = \Phi^+ \sqcup
\{ \alpha \in \Phi(\g g,\g a);\: -\alpha \in \Phi^+ \}.
$$
Consider an analogue of the Weyl group
\begin{equation}  \label{WeylGr}
W = W(G_{\BB R},A_{\BB R}) = N_{K_{\BB R}}(A_{\BB R}) / Z_{K_{\BB R}}(A_{\BB R})
= N_{K_{\BB R}}(A_{\BB R}) / M_{\BB R}.
\end{equation}
This is a finite group generated by reflections about root hyperplanes,
it acts simply transitively on the set of Weyl chambers.

Define
$$
\g n = \bigoplus_{\alpha \in \Phi^+} \g g^{-\alpha}, \qquad \g n \subset \g g,
$$
this is a nilpotent subalgebra because
$[\g g^{\alpha_1}, \g g^{\alpha_2}] \subset \g g^{\alpha_1+\alpha_2}$.
Since the roots are real-valued on $\g a_{\BB R}$, 
$\g n = \BB C \otimes \g n_{\BB R}$, where $\g n_{\BB R} = \g g_{\BB R} \cap \g n$,
i.e. $\g n$ is defined over $\BB R$. Let
\begin{align*}
N_{\BB R} &=
\text{connected subgroup of $G_{\BB R}$ with Lie algebra $\g n_{\BB R}$} \\
&= \exp(\g n_{\BB R}).
\end{align*}
Note that $\exp: \g n_{\BB R} \to N_{\BB R}$ is a polynomial map.

\begin{thm}[Iwasawa Decomposition]  \label{Iwasawa}
\begin{itemize}
\item[{\rm a)}]
Any two maximal abelian subspaces of $\g p_{\BB R}$ are conjugate under
$K_{\BB R}$;
\item[{\rm b)}]
$M_{\BB R}$, $A_{\BB R}$ and $N_{\BB R}$ are closed subgroups of $G_{\BB R}$,
the groups $A_{\BB R}$ and $M_{\BB R}$ commute, both normalize $N_{\BB R}$;
\item[{\rm c)}]
The map $K_{\BB R} \times A_{\BB R} \times N_{\BB R} \to G_{\BB R}$,
$(k,a,n) \mapsto kan$, is a diffeomorphism of ${\cal C}^{\infty}$-manifolds.
\end{itemize}
\end{thm}

(Sometimes this decomposition is called the $KAN$ decomposition.)

\begin{ex}
Let $G_{\BB R} = SL(n,\BB R)$, $K_{\BB R}=SO(n)$,
\begin{align*}
\g k_{\BB R} &= \{ X \in \operatorname{End}(\BB R^n) ;\: X^T=-X,\: \tr X =0 \}, \\
\g p_{\BB R} &= \{ X \in \operatorname{End}(\BB R^n) ;\: X^T=X,\: \tr X =0 \}, \\
\g a_{\BB R} &= \Bigl\{ X \in \operatorname{End}(\BB R^n) ;\:
\begin{matrix}\text{$X$ is diagonal with respect to} \\
\text{the standard basis and $\tr X =0$} \end{matrix} \Bigr\}, \\
A_{\BB R} &= \Bigl\{ a \in SL(n,\BB R);\:
\begin{matrix} \text{$a$ is diagonal with} \\
\text{positive diagonal entries} \end{matrix} \Bigr\}, \\
M_{\BB R} &= \left\{ \begin{pmatrix} \epsilon_1 & 0 & 0 \\ 0 & \ddots & 0 \\
0 & 0 & \epsilon_n \end{pmatrix} ;\: \epsilon_j = \pm 1 ,\:
\prod_j \epsilon_j =1 \right\}.
\end{align*}
In this case, $\dim M_{\BB R}=0$, $\g m_{\BB R}=0$ and
$\g a_{\BB R}$ is a Cartan subalgebra of $\g g_{\BB R}$
(meaning that $\g a$ is a Cartan subalgebra of $\g g$).
This implies that
$\Phi(\g g_{\BB R},\g a_{\BB R}) = \Phi(\g g,\g a)$ is the full root system of
$\g{sl}(n,\BB C)$.
In particular, there are no non-trivially proportional roots,
$\dim \g g^{\alpha}=1$ for all $\alpha \in \Phi$ and $W$ is the full Weyl group.
This situation is referred to by saying that $G_{\BB R}$ is {\em split}.
Each complex semisimple Lie group has a split real form, which is unique
up to conjugation. One can think of it as the opposite extreme to the
compact real form. For example, $SL(n,\BB C)$ has real forms:
compact $SU(n)$, split $SL(n,\BB R)$ as well as others
(such as $SU(p,q)$, $p+q=n$) if $n>2$. Write
$$
\g a_{\BB R} = \left\{ \begin{pmatrix} x_1 & 0 & 0 \\ 0 & \ddots & 0 \\
0 & 0 & x_n \end{pmatrix} ;\: x_j \in \BB R,\:
\sum_j x_j =0 \right\}.
$$
The group $W$ acting on $\g a_{\BB R}$ is isomorphic to the symmetric group $S_n$
acting on $\g a_{\BB R}$ by permutation of the $x_j$'s.
The root system is
$$
\Phi = \{ \pm(x^*_j-x^*_k) ;\: 1 \le j < k \le n \},
$$
where $x_1^*, \dots, x_n^* \in (\BB R^n)^*$ denote the dual basis of $\BB R^n$
so that $\langle x_j^*,e_k \rangle = \delta_{jk}$.
We can choose the highest Weyl chamber to be
$$
\left\{ \begin{pmatrix} x_1 & 0 & 0 \\ 0 & \ddots & 0 \\
0 & 0 & x_n \end{pmatrix} ;\: x_j \in \BB R,\: x_1 > x_2 > \dots > x_n,\:
\sum_j x_j =0 \right\},
$$
then
\begin{align*}
\Phi^+ &= \{ x_j^*-x_k^* ;\: 1 \le j < k \le n \}, \\
\g n_{\BB R} &= \text{Lie algebra of strictly lower-triangular matrices}, \\
N_{\BB R} &= \text{Lie group of lower-triangular unipotent matrices}.
\end{align*}
The Iwasawa decomposition for $GL(n,\BB R)$ (requires a slight modification
for $SL(n,\BB R)$) asserts that any invertible $n \times n$ real matrix can be
expressed uniquely as
{\Large
$$
\begin{pmatrix} \text{orthogonal} \\ \text{matrix} \end{pmatrix}
\begin{pmatrix} \text{diagonal matrix with} \\ \text{positive diagonal entries}
\end{pmatrix}
\begin{pmatrix} \text{lower triangular} \\ \text{unipotent matrix} \end{pmatrix}
$$}
and these three factors are real-analytic functions.
This can be proved using the Gramm-Schmidt orthogonalization process,
and the general case can be reduced to $GL(n,\BB R)$.
\end{ex}

Recall that by Theorem \ref{Iwasawa} the groups $M_{\BB R}$, $A_{\BB R}$
normalize $N_{\BB R}$, the group $A_{\BB R} \cdot N_{\BB R}$ is closed in
$G_{\BB R}$, and $M_{\BB R}$ is compact. Hence
$$
P_{\BB R} =_{\text{def}} M_{\BB R} A_{\BB R} N_{\BB R}
$$
is a closed subgroup of $G_{\BB R}$ called a {\em minimal parabolic subgroup}.
Then $N_{\BB R} \subset P_{\BB R}$ is a closed normal subgroup.

\noindent
\underline{Caution}: Beware of the following clash of notations.
The Lie algebra of $P_{\BB R}$ is
$\g m_{\BB R} \oplus \g a_{\BB R} \oplus \g n_{\BB R}$ and \underline{not}
the subspace $\g p_{\BB R}$ from the Cartan decomposition
$\g g_{\BB R} = \g k_{\BB R} \oplus \g p_{\BB R}$.
Unless, $[\g p_{\BB R}, \g p_{\BB R}] =0$, $\g p_{\BB R}$ is not even a Lie algebra.

\begin{lem}
We have: $P_{\BB R} \cap K_{\BB R} = M_{\BB R}$.
\end{lem}

\begin{proof}
Since $M_{\BB R} \subset P_{\BB R} \cap K_{\BB R}$, it is enough to show that
$K_{\BB R} \cap (A_{\BB R} N_{\BB R}) =\{e\}$,
but this follows from Iwasawa decomposition.
%But the intersection $K_{\BB R} \cap (A_{\BB R} N_{\BB R})$ is a compact subgroup
%of $A_{\BB R}N_{\BB R}$, hence $\{e\}$.
\end{proof}

From this lemma and Iwasawa decomposition we immediately obtain:

\begin{cor}
The homogeneous space
$G_{\BB R} / P_{\BB R} \simeq K_{\BB R} / M_{\BB R}$ is compact.
\end{cor}

\subsection{$KAK$ Decomposition}

By part (a) of Theorem \ref{Iwasawa},
$$
\g p_{\BB R} = \bigcup_{k \in K_{\BB R}} Ad(k) \g a_{\BB R}.
$$
Then it is immediate from the global Cartan decomposition
$G_{\BB R} = K_{\BB R} \cdot \exp \g p_{\BB R}$ (Proposition \ref{Cartan-decomp})
that $G_{\BB R} = K_{\BB R} A_{\BB R} K_{\BB R}$ in the sense that every element
$g \in G_{\BB R}$ can be expressed as $g=k_1ak_2$ with
$k_1,k_2 \in K_{\BB R}$ and $a \in A_{\BB R}$.
The following theorem describes the degree of non-uniqueness of this
decomposition.
Recall that the analogue of the Weyl group $W(G_{\BB R},A_{\BB R})$ was defined
by (\ref{WeylGr}).

\begin{thm}[$KAK$ Decomposition]
Every element $g \in G_{\BB R}$ can be expressed as $g=k_1ak_2$ with
$k_1,k_2 \in K_{\BB R}$ and $a \in A_{\BB R}$. In this decomposition,
$a$ is uniquely determined up to conjugation by an element of
$W(G_{\BB R},A_{\BB R})$.
If $a \in A_{\BB R}$ is fixed as $\exp X$ with $X \in \g a_{\BB R}$ and if
$\alpha(X) \ne 0$ for all $\alpha \in \Phi(\g g, \g a)$, then $k_1$
is unique up to multiplication on the right by an element of $M_{\BB R}$.
\end{thm}

Note that the roles of $k_1$ and $k_2$ in this decomposition are symmetric,
since $g^{-1} = (k_2)^{-1}a^{-1}(k_1)^{-1}$.
Hence, if $a \in A_{\BB R}$ is fixed as $\exp X$ with $X \in \g a_{\BB R}$ and
if $\alpha(X) \ne 0$ for all $\alpha \in \Phi(\g g, \g a)$, then $k_2$
is unique up to multiplication on the left by an element of $M_{\BB R}$.

\section{Principal Series Representations}

\subsection{$G_{\BB R}$-equivariant Vector Bundles}

The principal series representations of $G_{\BB R}$ are realized in the space
of sections of certain $G_{\BB R}$-equivariant vector bundles over a homogeneous
space $G_{\BB R}/P_{\BB R}$. For this reason we review equivariant vector bundles.

Let $G$ be any Lie group and suppose that it acts on a manifold $M$.
For concreteness, let us suppose that the manifold, the action map and
the vector bundles are all smooth.
But keep in mind that everything we do in this subsection applies to the
complex analytic setting as well.

A $G$-equivariant vector bundle on $M$ is a vector bundle ${\cal E} \to M$ 
with $G$-action on ${\cal E}$ by bundle maps lying over the $G$-action on $M$.
Write ${\cal E}_m$ for the fiber of ${\cal E}$ at $m \in M$, then,
for each $g \in G$,
$$
g: {\cal E}_m \to {\cal E}_{g \cdot m} \quad
\text{is a linear isomorphism of vector spaces}.
$$

Now suppose that $G$ acts on $M$ transitively. This means that $M$ is a
homogeneous space $G/H$ for some closed subgroup $H \subset G$.
We have a bijection:
$$
\left\{ \begin{matrix} \text{$G$-equivariant complex vector} \\
\text{bundles over $G/H$ of rank $n$}  \end{matrix}\right\}
\simeq
\left\{ \begin{matrix} \text{complex representations} \\
\text{of $H$ of dimension $n$} \end{matrix}\right\}
$$
In one direction, if ${\cal E} \to G/H$ is a $G$-equivariant complex vector
bundle of rank $n$, then $H$ acts on the $n$-dimensional vector space
$E={\cal E}_{eH}$ linearly, hence we get a representation $(\tau,E)$ of $H$.
Moreover, we have an isomorphism
\begin{equation}  \label{vbundle}
\begin{matrix}
{\cal E} & \simeq & G \times_H E & =_{\text{def}} & G \times E / \sim \\
\downarrow & & \downarrow \\
G/H & = & G/H \end{matrix}
\hskip-10pt (gh,v) \sim (g,\tau(h)v).
\end{equation}
To see this, note that the map
$$
G \times {\cal E}_{eH} \to {\cal E}, \qquad
(g,v) \mapsto gv \in {\cal E}_{gH} \subset {\cal E},
$$
is onto and then determine which points get identified with which.
%show that the points that get identified under this map
%are precisely $(gh,v)$ and $(g,\tau(h)v)$ for all $h \in H$.
Conversely, any representation $(\tau, E)$ of $H$  determines a
$G$-equivariant vector bundle by (\ref{vbundle}),
$\operatorname{rank} {\cal E} = \dim E$.
And, of course, we have a similar bijection between real vector bundles over
$G/H$ and real representations of $H$.

Now, let ${\cal E} \to G/H$ be the vector bundle corresponding to a
representation $(\tau, E)$ of $H$. Then, by (\ref{vbundle}), the global
sections of ${\cal E}$ can be described as follows:
\begin{multline*}
{\cal C}^{\infty}(G/H, {\cal E}) \simeq \\
\{ f: G \xrightarrow{\text{smooth }} E ;\: f(gh)=\tau(h^{-1})f(g),\:
\forall g \in G,\: h \in H \}
\end{multline*}
(because $(g,v)=(gh,\tau(h^{-1})v)$ in $G \times_H E$).

\subsection{Construction of the Principal Series Representations}

In this subsection we construct the principal series representations.
These are obtained by {\em inducing} representations from a minimal
parabolic subgroup $P_{\BB R}$ to $G_{\BB R}$.
In other words, one starts with a finite-dimensional representation
of $P_{\BB R}$, forms the corresponding $G_{\BB R}$-equivariant vector
bundle ${\cal E} \to G_{\BB R}/P_{\BB R}$ and lets $G_{\BB R}$ act on
${\cal C}^{\infty}(G_{\BB R}/P_{\BB R}, {\cal E})$
-- the space of global sections of that bundle.
When the group $G_{\BB R}$ is $SU(1,1)$, the principal series representations
are precisely the representations $(\pi_{\chi},V_{\chi})$ constructed in
Section \ref{V-chi-construction}.
We use the same notations as in Subsection \ref{Iwasawa-subsection}
where we discussed the Iwasawa decomposition.

Note that any $\nu \in \g a^*$ lifts to a multiplicative character
(i.e. one-dimensional representation)
$$
e^{\nu} : A_{\BB R} \to \BB C^{\times} \quad \text{by the rule} \quad
e^{\nu}(\exp X) =_{\text{def}} e^{\nu(X)}, \quad X \in \g a_{\BB R},
$$
then
$$
e^{\nu}(a_1) \cdot e^{\nu}(a_2) = e^{\nu}(a_1a_2), \qquad
\forall a_1, a_2 \in A_{\BB R}.
$$
Using the decomposition $P_{\BB R} = M_{\BB R} A_{\BB R} N_{\BB R}$,
we can extend $e^{\nu}$ to a one-dimensional representation
$e^{\nu}: P_{\BB R} \to \BB C^{\times}$ with $e^{\nu}$ being trivial on
$M_{\BB R}$ and $N_{\BB R}$:
$$
e^{\nu}(man) = e^{\nu}(a), \qquad
m \in M_{\BB R},\: a \in A_{\BB R},\: n \in N_{\BB R}.
$$
Note that, since $M_{\BB R}$ and $A_{\BB R}$ normalize $N_{\BB R}$ and
$M_{\BB R}$ commutes with $A_{\BB R}$, it is still true that
$$
e^{\nu}(p_1) \cdot e^{\nu}(p_2) = e^{\nu}(p_1p_2), \qquad
\forall p_1, p_2 \in P_{\BB R}.
$$
This one-dimensional representation determines a $G_{\BB R}$-equivariant
line bundle ${\cal L}_{\nu} \to G_{\BB R}/P_{\BB R}$.

In particular, we can apply above construction to $\nu=-2\rho$, where
$$
\rho = \frac12 \sum_{\alpha \in \Phi^+} (\dim \g g^{\alpha}) \alpha
\quad \in \g a_{\BB R}^* \subset \g a^*,
$$
and obtain a line bundle ${\cal L}_{-2\rho} \to G_{\BB R}/P_{\BB R}$.
Since $\rho \in \g a_{\BB R}^*$, this line bundle has a real structure, i.e.
can be obtained by complexifying a certain $G_{\BB R}$-equivariant {\em real}
line bundle.

\begin{lem}  \label{cotangent-iso}
As a $G_{\BB R}$-equivariant line bundle with real structure,
${\cal L}_{-2\rho} \to G_{\BB R}/P_{\BB R}$ is isomorphic to the top exterior power
of the complexified cotangent bundle of $G_{\BB R}/P_{\BB R}$,
$\BB C \otimes \Lambda^{\text{top}} T^* G_{\BB R}/P_{\BB R}$.
\end{lem}

\begin{proof}
%Since the tangent space of $G_{\BB R}/P_{\BB R}$ at $eP_{\BB R}$ is
%$\g g_{\BB R} / Lie(P_{\BB R})$, the cotangent space at $eP_{\BB R}$ is
%the annihilator of $Lie (P_{\BB R})$ in $\g g_{\BB R}^*$:
%$$
%T^*_{eP_{\BB R}} G_{\BB R}/P_{\BB R} = (\g g_{\BB R} / Lie(P_{\BB R}))^*
%= \{ l \in \g g_{\BB R}^* ;\: l \bigr|_{Lie(P_{\BB R})}=0 \}.
%$$
As a $P_{\BB R}$-module,
the tangent space of $G_{\BB R}/P_{\BB R}$ at $eP_{\BB R}$ is
%$T_{eP_{\BB R}} G_{\BB R}/P_{\BB R}$ is
$$
\g g_{\BB R} / Lie(P_{\BB R}) =
(\g m_{\BB R} \oplus \g a_{\BB R} \oplus \g n_{\BB R} \oplus \theta \g n_{\BB R})
/ (\g m_{\BB R} \oplus \g a_{\BB R} \oplus \g n_{\BB R}),
$$
and
$$
T_{eP_{\BB R}} G_{\BB R}/P_{\BB R} = \g g_{\BB R} / Lie(P_{\BB R})
\simeq \theta \g n_{\BB R}
$$
as $M_{\BB R}A_{\BB R}$-modules.
Identifying $\g g_{\BB R}$ with its dual $\g g_{\BB R}^*$ via the Killing form
and noting that $\g n_{\BB R}$ and $\theta \g n_{\BB R}$ are dual to each other,
we conclude that
$$
T^*_{eP_{\BB R}} G_{\BB R}/P_{\BB R} \simeq \g n_{\BB R}
$$
as $M_{\BB R}A_{\BB R}$-modules. Then we have an isomorphism of $P_{\BB R}$-modules
$$
\Lambda^{\text{top}} T^*_{eP_{\BB R}} G_{\BB R}/P_{\BB R}
\simeq \Lambda^{\text{top}} \g n_{\BB R},
$$
since $N_{\BB R}$ acts trivially on both sides
(the fact that $N_{\BB R}$ acts trivially on each side boils down to the
property of root spaces
$[\g g^{\alpha_1}, \g g^{\alpha_2}] \subset \g g^{\alpha_1+\alpha_2}$.)

Since we have a bijection between $G_{\BB R}$-equivariant real line bundles 
over $G_{\BB R}/P_{\BB R}$ and one-dimensional real representations of
$P_{\BB R}$, it is sufficient to show isomorphism of $P_{\BB R}$-modules
$$
\Lambda^{\text{top}} T^*_{eP_{\BB R}} G_{\BB R}/P_{\BB R}
\simeq \Lambda^{\text{top}} \g n_{\BB R}
\qquad \text{and} \qquad (e^{-2\rho},\BB R).
$$
It is certainly clear that the actions of $A_{\BB R}$ and $N_{\BB R}$
on both sides are the same.
Since the trivial representation is the only real one-dimensional
representation of a compact connected Lie group, $M_{\BB R}^0$ -- the connected
component of the identity element in $M_{\BB R}$ -- acts trivially on both sides.
It remains to show that every connected component of $M_{\BB R}$ contains
a representative $m$ which acts trivially on $\Lambda^{\text{top}} \g n_{\BB R}$.
This part requires a bit more work and we omit it.
\end{proof}

\begin{cor}
The homogeneous space $G_{\BB R}/P_{\BB R}$ is orientable.
\end{cor}

\begin{proof}
Recall that a manifold $M$ is orientable if and only if 
the top exterior power of the cotangent bundle $\Lambda^{\text{top}} T^*M$
has a non-vanishing global section. In our case $M=G_{\BB R}/P_{\BB R}$ and
$\Lambda^{\text{top}} T^*M$ is isomorphic to the real subbundle of
${\cal L}_{-2\rho}$.
The group $K_{\BB R}$ acts on $G_{\BB R}/P_{\BB R} \simeq K_{\BB R}/M_{\BB R}$
transitively, and it is easy to see that there is a real $K_{\BB R}$-invariant
non-vanishing section of ${\cal L}_{-2\rho}$ -- the one corresponding to a
nonzero constant map $K_{\BB R} \to \BB R$.
\end{proof}

Now we have everything we need to construct the principal series
representations of $G_{\BB R}$.
Let $(\tau,E)$ be an irreducible representation of $M_{\BB R}$ (over $\BB C$),
and pick an element $\nu \in \g a^*$.
Define a representation $(\tau_{\nu},E)$ of $P_{\BB R}$ as follows.
Informally, it can be described by saying that
$A_{\BB R}$ acts on $E$ by $e^{\nu-\rho} \cdot Id_E$,
$M_{\BB R}$ acts by $\tau$ and $N_{\BB R}$ acts trivially.
More precisely, we first consider a representation of $M_{\BB R} \times A_{\BB R}$
which is the exterior tensor product representation $\tau \boxtimes e^{\nu-\rho}$
on $E \otimes \BB C \simeq E$, and then define the representation
$(\tau_{\nu},E)$ of $P_{\BB R}$ by extending $\tau \boxtimes e^{\nu-\rho}$
via the isomorphism
$P_{\BB R}/N_{\BB R} \simeq M_{\BB R} A_{\BB R} \simeq M_{\BB R} \times A_{\BB R}$.

By the above construction, the representation $(\tau_{\nu},E)$ leads to a
$G_{\BB R}$-equivariant vector bundle ${\cal E}_{\nu} \to G_{\BB R}/P_{\BB R}$.
Its global sections are
\begin{multline*}
{\cal C}^{\infty}(G_{\BB R}/P_{\BB R}, {\cal E}_{\nu}) \simeq \\
\{ f: G_{\BB R} \xrightarrow{{\cal C}^{\infty}} E ;\: f(gp)=\tau_{\nu}(p^{-1})f(g),\:
\forall g \in G_{\BB R},\: p \in P_{\BB R} \}.
\end{multline*}
The group $G_{\BB R}$ acts on ${\cal C}^{\infty}(G_{\BB R}/P_{\BB R}, {\cal E}_{\nu})$
by left translations; we denote this action by $\pi_{E,\nu}$ and call
it a {\em principal series representation}:
$$
(\pi_{E,\nu}(\tilde g) f)(g) = f(\tilde g^{-1}g), \qquad
\tilde g, g \in G_{\BB R},\:
f \in {\cal C}^{\infty}(G_{\BB R}/P_{\BB R}, {\cal E}_{\nu}).
$$

\subsection{Admissibility of the Principal Series Representations}

Recall that we have a diffeomorphism of $K_{\BB R}$-spaces
$G_{\BB R}/P_{\BB R} \simeq K_{\BB R}/M_{\BB R}$.
This induces isomorphisms of vector spaces:
\begin{multline}  \label{iso}
{\cal C}^{\infty}(G_{\BB R}/P_{\BB R}, {\cal E}_{\nu}) \\
\simeq \{ f: G_{\BB R} \xrightarrow{{\cal C}^{\infty}} E ;\:
f(gp)=\tau_{\nu}(p^{-1})f(g),\: \forall g \in G_{\BB R},\: p \in P_{\BB R} \} \\
\simeq \{ f: K_{\BB R} \xrightarrow{{\cal C}^{\infty}} E ;\:
f(km)=\tau(m^{-1})f(k),\: \forall k \in K_{\BB R},\: m \in M_{\BB R} \} \\
\simeq {\cal C}^{\infty}(K_{\BB R}/M_{\BB R}, {\cal E}_{\nu}).
\end{multline}
The last two spaces are (in obvious way) only $K_{\BB R}$-modules.

We would like to know how $\pi_{E,\nu}$ acts on 
${\cal C}^{\infty}(K_{\BB R}/M_{\BB R}, {\cal E}_{\nu})$.
For an element $g \in G_{\BB R}$, let $k(g) \in K_{\BB R}$, $a(g) \in A_{\BB R}$
and $n(g) \in N_{\BB R}$ be the Iwasawa components of $g$.
Suppose that $f: K_{\BB R} \to E$ is a smooth function such that
$$
f(um)=\tau(m^{-1})f(u),\: \forall u \in K_{\BB R},\: m \in M_{\BB R}.
$$
We can think of $f$ as a function on $G_{\BB R}$ via the isomorphism
(\ref{iso}), then, for all $g \in G_{\BB R}$,
\begin{multline*}
(\pi_{E,\nu}(g)f)(u) = f(g^{-1}u)
= f \bigl( k(g^{-1}u) \cdot a(g^{-1}u) \cdot n(g^{-1}u) \bigr)  \\
= e^{\rho-\nu} \bigl( a(g^{-1}u) \bigr) \cdot f \bigl( k(g^{-1}u) \bigr).
\end{multline*}
We conclude that $\pi_{E,\nu}(g)$ acts on
${\cal C}^{\infty}(K_{\BB R}/M_{\BB R}, {\cal E}_{\nu})$ by the $G_{\BB R}$-action
on $G_{\BB R}/P_{\BB R} \simeq K_{\BB R}/M_{\BB R}$, $(g,u) \mapsto k(gu)$, and
multiplication by a real analytic function
$$
G_{\BB R} \times K_{\BB R} \to \BB C, \qquad
(g,u) \mapsto e^{\rho-\nu} \bigl( a(g^{-1}u) \bigr),
$$
called the {\em factor of automorphy}.

The $K_{\BB R}/M_{\BB R}$ picture allows us to identify the space of
$K_{\BB R}$-finite vectors in
$(\pi_{E,\nu}, {\cal C}^{\infty}(G_{\BB R}/P_{\BB R}, {\cal E}_{\nu}))$.
By Peter-Weyl Theorem, the space of $K_{\BB R} \times K_{\BB R}$-finite vectors
in ${\cal C}^{\infty}(K_{\BB R})$ is
$$
{\cal C}^{\infty}(K_{\BB R})_{fini} \simeq
\bigoplus_{i \in \hat K_{\BB R}} (\tau_i \boxtimes \tau_i^*, U_i \otimes U_i^*).
$$
Therefore, the space of $K_{\BB R}$-finite vectors
%$$
%{\cal C}^{\infty}(K_{\BB R}/M_{\BB R}, {\cal E}_{\nu})_{fini} \simeq
%\bigoplus_{i \in \hat K_{\BB R}} U_i \otimes (U_i^* \otimes E)^{M_{\BB R}},
%$$
$$
{\cal C}^{\infty}(K_{\BB R}/M_{\BB R}, {\cal E}_{\nu})_{fini} \simeq
\bigoplus_{i \in \hat K_{\BB R}} U_i \otimes \operatorname{Hom}_{M_{\BB R}} (U_i,E),
$$
where $\operatorname{Hom}_{M_{\BB R}} (U_i,E)$
%$$
%(U_i^* \otimes E)^{M_{\BB R}} = \operatorname{Hom}_{M_{\BB R}} (U_i,E)
%$$
denotes the space of $M_{\BB R}$-equivariant
linear maps $U_i \to E$ ($M_{\BB R}$ acts via $\tau_i |_{M_{\BB R}}$ on $U_i$
and via $\tau$ on $E$) and $K_{\BB R}$ acts via $\tau_i$ on $U_i$
and trivially on $\operatorname{Hom}_{M_{\BB R}} (U_i,E)$.
Assuming $(\tau,E)$ is irreducible, we can obtain a very crude estimate for
the dimension of the $i$-isotypic component of
${\cal C}^{\infty}(K_{\BB R}/M_{\BB R}, {\cal E}_{\nu})$:
\begin{multline}  \label{dim-est}
\dim {\cal C}^{\infty}(K_{\BB R}/M_{\BB R}, {\cal E}_{\nu})(i)
= \dim U_i \cdot \dim \operatorname{Hom}_{M_{\BB R}} (U_i,E) \\
= \dim U_i \cdot \bigl( \text{multiplicity of $(\tau,E)$ in
$(\tau_i |_{M_{\BB R}}, U_i)$} \bigr) \\
\le (\dim U_i)^2.
\end{multline}
In particular, the principal series representations are admissible.

It is also true that the principal series representations have finite length.
Proving this requires character theory, see the last paragraph of
Subsection \ref{regularity-subsection}.
Moreover, if the representation $(\tau,E)$ of $M_{\BB R}$ is irreducible
and $\nu \in \g a^*$ is ``sufficiently generic'', the resulting principal
series representation $\pi_{E,\nu}$ of $G_{\BB R}$ is irreducible.
The description of $\pi_{E,\nu}$ as a representation of $G_{\BB R}$ on the vector
space
$$
\{ f: K_{\BB R} \xrightarrow{{\cal C}^{\infty}} E ;\:
f(km)=\tau(m^{-1})f(k),\: \forall k \in K_{\BB R},\: m \in M_{\BB R} \}
$$
(in terms of factor of automorphy) shows that $\pi_{E,\nu}$ induces (continuous!)
representations on various related spaces
$$
{\cal C}^k(K_{\BB R}/M_{\BB R}, {\cal E}_{\nu}),
\qquad 0 \le k \le \infty,\: k=-\infty,\omega,
$$
and
$$
L^p(K_{\BB R}/M_{\BB R}, {\cal E}_{\nu}), \qquad 1 \le p < \infty,
$$
consisting of functions $f: K_{\BB R} \to E$ that are $L^p$ with respect to
the Haar measure on $K_{\BB R}$ (and satisfy
$f(km)=\tau(m^{-1})f(k)$ for all $k \in K_{\BB R}$, $m \in M_{\BB R}$ as before).
All of these representations have the same $K_{\BB R}$-structure, i.e. their
spaces of $K_{\BB R}$-finite vectors are the same as $(\g g, K)$-modules.
So all these representations are infinitesimally equivalent to each other.

\subsection{Subrepresentation Theorem}

Of particular interest is the following highly non-trivial result
due to W.Casselman \cite{Ca0}:

\begin{thm}[Subrepresentation Theorem]
For every irreducible Harish-Chandra module $W$, there exist an irreducible
representation $(\tau,E)$ of $M_{\BB R}$ and a $\nu \in \g a^*$ with an embedding
$$
W \hookrightarrow {\cal C}^{\infty}(G_{\BB R}/P_{\BB R}, {\cal E}_{\nu})_{fini}.
$$
\end{thm}

\begin{cor}
Every admissible irreducible representation of $G_{\BB R}$ is infinitesimally
equivalent to a subrepresentation of some principal series representation.
\end{cor}

The following corollary plays the key role in constructing globalizations
of Harish-Chandra modules discussed in Subsection \ref{globalization}.

\begin{cor}
Every Harish-Chandra module can be realized as the underlying
Harish-Chandra module of an admissible representation of finite length.
In other words, every Harish-Chandra module has a globalization.
\end{cor}

The Subrepresentation Theorem replaces a much older result due to
Harish-Chandra:

\begin{thm}[Subquotient Theorem]
Any irreducible Harish-Chandra module arises as a subquotient
(i.e. quotient of a submodule) of the underlying Harish-Chandra module
of some principal series representation.
\end{thm}

\subsection{Unitary Principal Series Representations}

In this subsection we construct the unitary principal series representations.
Suppose that $\nu \in i \g a_{\BB R}^*$, then $e^{\nu}: A_{\BB R} \to \BB C$ is
unitary (but $e^{\nu-\rho}$ is not).
Fix an $M_{\BB R}$-invariant positive definite inner product on $E$,
then get a $G_{\BB R}$-invariant pairing
\begin{multline*}
{\cal C}^{\infty}(G_{\BB R}/P_{\BB R}, {\cal E}_{\nu}) \times
{\cal C}^{\infty}(G_{\BB R}/P_{\BB R}, {\cal E}_{\nu})  \\
\to {\cal C}^{\infty}(G_{\BB R}/P_{\BB R}, {\cal L}_{-2\rho}) \simeq
\Omega^{\text{top}}(G_{\BB R}/P_{\BB R}) \to \BB C.
\end{multline*}
In this composition, the first map is taking inner product on the fibers of
${\cal E}_{\nu}$ pointwise over $G_{\BB R}/P_{\BB R}$,
the second map is the isomorphism from Lemma \ref{cotangent-iso} and
the last map is given by integration over $G_{\BB R}/P_{\BB R}$.
This composition results in a $G_{\BB R}$-invariant positive definite
hermitian pairing.
Complete the representation spaces to get action $\pi_{E,\nu}$ on
$$
L^2(G_{\BB R}/P_{\BB R}, {\cal E}_{\nu}) \qquad
\text{for \underline{$\nu \in i \g a_{\BB R}^*$}.}
$$
This representation is denoted by
$$
\operatorname{Ind}_{P_{\BB R}}^{G_{\BB R}} (\tau \boxtimes e^{\nu})
$$
and is often called an {\em induced representation},
the shift by $\rho$ is implicit.
Thus, for $\nu \in i \g a_{\BB R}^*$, we get the
{\em unitary principal series representations}.
When $\nu$ is in $\g a^*$ and not necessarily in $i \g a_{\BB R}^*$,
we get the ``non-unitary'' principal series representations,
or -- much better -- ``not necessarily unitary'' principal series
representations, since some of these may be unitary for totally non-obvious
reasons.

Now, let $\nu \in \g a^*$ and let $(\tau^*,E^*)$ be the representation of
$M_{\BB R}$ dual to $(\tau,E)$. Form a representation $(\tau_{-\nu}^*,E^*)$
of $P_{\BB R}$ and let ${\cal E}^*_{-\nu}$ be the corresponding
$G_{\BB R}$-equivariant vector bundle over $G_{\BB R}/P_{\BB R}$.
Then the same argument as before proves that the bilinear pairing
$E \times E^* \to \BB C$ induces a $G_{\BB R}$-invariant non-degenerate
bilinear pairing
$$
{\cal C}^{\infty}(G_{\BB R}/P_{\BB R}, {\cal E}_{\nu}) \times
{\cal C}^{\infty}(G_{\BB R}/P_{\BB R}, {\cal E}^*_{-\nu}) \to \BB C.
$$
Completing with respect to appropriate topology yields representations
$\pi_{E,\nu}$ on, for example,
$$
{\cal C}^{\infty}(G_{\BB R}/P_{\BB R}, {\cal E}_{\nu}), \qquad
L^p(G_{\BB R}/P_{\BB R}, {\cal E}_{\nu}), \quad 1 < p < \infty,
$$
with dual representations $\pi_{E^*,-\nu}$ on
$$
{\cal C}^{-\infty}(G_{\BB R}/P_{\BB R}, {\cal E}^*_{-\nu}), \qquad
L^q(G_{\BB R}/P_{\BB R}, {\cal E}^*_{-\nu}), \quad p^{-1}+q^{-1}=1,
$$
where ${\cal C}^{-\infty}$ denotes the space of distributions.

\section{Group Characters}

\subsection{Hilbert-Schmidt Operators}

Note that if the vector space $V$ is infinite dimensional, then 
taking the trace of a linear transformation $T \in \operatorname{End}(V)$
is problematic. In this subsection we discuss Hilbert-Schmidt operators.
These will be used in the next subsection to define trace class operators.
Let $V$ be a Hilbert space and fix an orthonormal basis $\{v_i\}$ of $V$.

\begin{df}
A bounded operator $T \in \operatorname{End}(V)$ is a
{\em Hilbert-Schmidt operator} if $\sum_i \|Tv_i\|^2 < \infty$.
\end{df}

We will see shortly that this definition as well as the next one
do not depend on the choice of basis.
We denote the space of Hilbert-Schmidt operators by
$\operatorname{End}(V)_{HS}$.

\begin{df}
Let $T_1$ and $T_2$ be two Hilbert-Schmidt operators, then their
{\em Hilbert-Schmidt inner product} is
$$
\langle T_1,T_2\rangle_{HS} =_{\text{def}}
\sum_i \langle T_1v_i,T_2v_i \rangle = \tr (T_1^*T_2).
$$
This inner product induces the {\em Hilbert-Schmidt norm}
$$
\|T\|_{HS}^2 = \langle T,T \rangle_{HS} = \sum_i \|Tv_i\|^2.
$$
\end{df}

This inner product turns the space of Hilbert-Schmidt operators into a Hilbert
space (in particular, it is complete).

Note that $\|Tv_i\|^2 = \sum_j | \langle Tv_i,v_j \rangle |^2$,
hence we obtain the following expression for the Hilbert-Schmidt norm:
\begin{equation}  \label{HS-norm}
\|T\|_{HS}^2 = \sum_i \|Tv_i\|^2 = \sum_{ij} | \langle Tv_i,v_j \rangle |^2.
\end{equation}
The last sum is the sum of squares of absolute values of matrix coefficients
of $T$ in basis $\{v_i\}$.

\begin{lem}
If $T \in \operatorname{End}(V)_{HS}$, then its adjoint $T^*$ also is a
Hilbert-Schmidt operator with the same norm. Moreover, if
$T_1,T_2 \in \operatorname{End}(V)_{HS}$, then
\begin{equation}  \label{T^*}
\langle T_1,T_2\rangle_{HS} = \overline{\langle T_1^*,T_2^*\rangle_{HS}}.
\end{equation}
\end{lem}

\begin{proof}
From (\ref{HS-norm}) we get:
\begin{multline*}
\|T\|_{HS}^2 = \sum_{ij} | \langle Tv_i,v_j \rangle |^2
= \sum_{ij} | \langle v_i,T^*v_j \rangle |^2  \\
= \sum_{ij} | \langle T^*v_j,v_i \rangle |^2 = \|T^*\|_{HS}^2.
\end{multline*}
The second part follows from the first part and the polarization identity:
\begin{multline*}
4 \langle T_1,T_2\rangle_{HS} \\
= \|T_1+T_2\|_{HS}^2 - \|T_1-T_2\|_{HS}^2 + i \|T_1+iT_2\|_{HS}^2
- i \|T_1-iT_2\|_{HS}^2  \\
= \|T_1^*+T_2^*\|_{HS}^2 - \|T_1^*-T_2^*\|_{HS}^2 + i \|T_1^*-iT_2^*\|_{HS}^2
- i \|T_1^*+iT_2^*\|_{HS}^2  \\
= 4 \overline{\langle T_1^*,T_2^*\rangle_{HS}}.
\end{multline*}
\end{proof}

\begin{lem}
Let $T \in \operatorname{End}(V)_{HS}$, $A \in \operatorname{End}(V)$ with
operator norm $\|A\|$. Then $AT, TA \in \operatorname{End}(V)_{HS}$ and
$$
\|AT\|_{HS} \le \|A\| \cdot \|T\|_{HS}, \qquad
\|TA\|_{HS} \le \|A\| \cdot \|T\|_{HS}.
$$
\end{lem}

\begin{rem}
This lemma can be restated as follows. The Hilbert-Schmidt operators form
a two-sided ideal in the Banach algebra of bounded linear operators on $V$.
\end{rem}

\begin{proof}
We have:
$$
\|AT\|_{HS}^2 = \sum_i \|ATv_i\|^2 \le \|A\|^2 \sum_i \|Tv_i\|^2
= \|A\|^2 \cdot \|T\|_{HS}^2.
$$
Finally, $TA \in \operatorname{End}(V)_{HS}$ if and only if
$(TA)^* \in \operatorname{End}(V)_{HS}$ and
\begin{multline*}
\|TA\|_{HS}=\|(TA)^*\|_{HS} = \|A^*T^*\|_{HS} \\
\le \|A^*\| \cdot \|T^*\|_{HS} = \|A\| \cdot \|T\|_{HS}.
\end{multline*}
\end{proof}

\begin{cor}
For $U \in \operatorname{Aut}(V)$ unitary, $\|UTU^{-1}\|_{HS} = \|T\|_{HS}$.
\end{cor}

\begin{proof}
Since $U$ is unitary, $\|U\|=\|U^{-1}\|=1$. By above lemma,
$\|UTU^{-1}\|_{HS} \le \|T\|_{HS}$. By symmetry, $\|T\|_{HS} \le \|UTU^{-1}\|_{HS}$,
and we have equality $\|UTU^{-1}\|_{HS} = \|T\|_{HS}$.
\end{proof}

Now, let $\{\tilde v_i\}$ be another orthonormal basis of $V$,
then there exists a unique linear transformation $U$ such that
$U\tilde v_i = v_i$ for all $i$, then $U^{-1} v_i = \tilde v_i$ for all $i$
and $U$ is automatically unitary. We have:
\begin{multline*}
\|T\|_{HS}^2 = \|UTU^{-1}\|_{HS}^2 = \sum_i \|UTU^{-1}v_i\|^2 \\
= \sum_i \|UT \tilde v_i\|^2 = \sum_i \|T \tilde v_i\|^2.
\end{multline*}
This calculation shows that we can take any orthonormal basis of $V$
without affecting the Hilbert-Schmidt norm or inner product.

\subsection{Trace Class Operators}

In this subsection we define trace class operators. As the name suggests,
there is a well defined notion of trace for these operators.

\begin{df}
A {\em trace class operator} on $V$ is a finite linear combination of
compositions $S_1S_2$, where $S_1, S_2 \in \operatorname{End}(V)_{HS}$.
In other words,
$$
\{\text{trace class operators}\} = \bigl( \operatorname{End}(V)_{HS} \bigr)^2
$$
in the sense of ideals in the Banach algebra of bounded linear operators on $V$.
\end{df}

Let $\{v_i\}$ be an orthonormal basis of $V$ and suppose that $T=S_1S_2$
with $S_1, S_2 \in \operatorname{End}(V)_{HS}$. Then
$$
\langle S_2, S_1^* \rangle_{HS} = \sum_i \langle S_2v_i, S_1^*v_i \rangle
= \sum_i \langle S_1S_2v_i,v_i \rangle = \sum_i \langle Tv_i,v_i \rangle
$$
is absolutely convergent and does not depend on the choice of
orthonormal basis $\{v_i\}$.
We define $\tr T = \langle S_2, S_1^* \rangle_{HS}$ if $T=S_1S_2$ and extend
this definition by linearity to all trace class operators.
Then $\tr T$ is well defined, i.e. independent of the decomposition of $T$
into a linear combination of products of pairs of Hilbert-Schmidt operators
and the choice of orthonormal basis. This definition is consistent with
the notion of trace in finite dimensional case.

\begin{lem}  \label{trace-lem}
If $T \in \operatorname{End}(V)$ is of trace class and
$A \in \operatorname{End}(V)$ is bounded, then $AT$, $TA$ are of trace class
and $\tr(AT)=\tr(TA)$. In particular, if $A$ has a bounded inverse, then
$ATA^{-1}$ is of trace class and $\tr(ATA^{-1})=\tr T$.
\end{lem}

\begin{proof}
By linearity, we may assume that $T=S_1S_2$ with
$S_1, S_2 \in \operatorname{End}(V)_{HS}$. Then
$AT=(AS_1)S_2$, $TA=S_1(S_2A)$ with $AS_1$ and
$S_2A \in \operatorname{End}(V)_{HS}$, thus $AT$, $TA$ are of trace class.

Applying (\ref{T^*}), we obtain:
\begin{multline*}
\tr(AT) = \langle S_2, (AS_1)^* \rangle_{HS}
= \overline{\langle S_2^*, AS_1 \rangle_{HS}}
= \overline{\langle A^*S_2^*, S_1 \rangle_{HS}}  \\
= \overline{\langle (S_2A)^*, S_1 \rangle_{HS}}
= \langle S_2A, S_1^* \rangle_{HS} = \tr(TA).
\end{multline*}
\end{proof}

\begin{rem}
%One can show that $T \in \operatorname{End}(V)$ is of trace class if and only
%if the sum
%$$
%\sum_i \langle T^*Tv_i,v_i \rangle^{1/2} = \sum_i \|Tv_i\|
%$$
%converges for some (and hence for any) orthonormal basis $\{v_i\}$ of $V$.
We have the following inclusions of two-sided ideals in the Banach algebra
of bounded linear operators on $V$:
\begin{multline*}
\biggl\{ \begin{matrix} \text{finite rank} \\ \text{operators}
\end{matrix}\biggr\} \subset
\biggl\{ \begin{matrix} \text{trace class} \\ \text{operators}
\end{matrix}\biggr\}  \\
\subset \biggl\{ \begin{matrix} \text{Hilbert-Schmidt} \\ \text{operators}
\end{matrix}\biggr\} \subset
\biggl\{ \begin{matrix} \text{compact} \\ \text{operators}
\end{matrix}\biggr\}.
\end{multline*}
(There are many equivalent definitions of compact operators, here is one.
Let $T: V_1 \to V_2$ be a linear transformation between two topological vector
spaces with $V_1$ a normed vector space.
Then $T$ is said to be {\em compact operator} if the closure of the image
in $V_2$ of the unit ball is compact.)
If $V$ is infinite dimensional these inclusions are all proper.
\end{rem}

\subsection{The Definition of a Character}

Let $(\pi,V)$ be an admissible representation of $G_{\BB R}$ of finite length
on a Hilbert space $V$ (but not necessarily unitary).
Fix a maximal compact subgroup $K_{\BB R} \subset G_{\BB R}$.
We can change the inner product on $V$ without changing the Hilbert space
topology so that $K_{\BB R}$ acts unitarily.
(First average the inner product over the $K_{\BB R}$-action, and then argue
using the uniform boundedness principle that the topology with respect to
the new inner product is the same as before.)
Then
$$
V_{fini} = \bigoplus_{i \in \hat K_{\BB R}} V(i)
$$
becomes an orthogonal decomposition.
Take an orthonormal basis in each $V(i)$, put these bases together to get
a (Hilbert space) orthonormal basis $\{v_k\}$ of $V$.
Fix a bi-invariant Haar measure $dg$ on $G_{\BB R}$.

\begin{thm}[Harish-Chandra]
For each $f \in {\cal C}^{\infty}_c(G_{\BB R})$, the operator
$$
\pi(f) = \int_{G_{\BB R}} f(g)\pi(g)\,dg
$$
is of trace class, and the linear functional
$$
\Theta_{\pi} : f \mapsto \tr \pi(f)
$$
is a conjugate-invariant distribution on $G_{\BB R}$.
\end{thm}

\begin{rem}
Note that $\pi(f)$ (and hence $\tr \pi(f)$) depends on the product $f(g)\,dg$
and not directly on $f$ or $dg$. The expression $f(g)\,dg$ can be interpreted
as a smooth compactly supported measure on $G_{\BB R}$.
Thus the theorem asserts that $\Theta_{\pi}$ is a continuous linear functional
on the space of smooth compactly supported measures on $G_{\BB R}$.

Let $\theta_k: G_{\BB R} \to \BB C$ be the diagonal matrix coefficient functions
$\langle \pi(g)v_k,v_k \rangle$. The proof shows that
$\Theta_{\pi} = \sum_k \theta_k$ as {\em distributions} and the convergence
is in the weak-$*$ topology (i.e.
$\Theta_{\pi}(f) = \sum_k \int_{G_{\BB R}} \theta_k(g) \cdot f(g)\,dg$, for all
$f \in {\cal C}^{\infty}_c(G_{\BB R})$).
In other words, the distribution $\Theta_{\pi}$ extends the notion of the
character of a finite dimensional representation.
For this reason, $\Theta_{\pi}$ is called the {\em Harish-Chandra character}
of $(\pi,V)$.

We will use a formal notation of integration for the pairing between
distributions and smooth functions with compact support and write
$$
\tr \pi(f) = \int_{G_{\BB R}} \Theta_{\pi}(g) \cdot f(g)\,dg
$$
for $\Theta_{\pi}(f)$.
\end{rem}

\begin{proof}
Let $\Omega_{\g k} \in {\cal U}(\g k)$ be the Casimir element of $\g k$.
Pick a Cartan subalgebra $\g t$ of $\g k$ and a system of positive roots
relative to $\g t$. Let
$$
\rho_c = \frac 12 (\text{sum of positive roots}) \quad \in \g t^*.
$$
For each $i \in \hat K_{\BB R}$, let $\mu_i \in \g t^*$ denote the
highest weight of $(\tau_i,U_i)$.
Since $(\tau_i,U_i)$ is irreducible, $\Omega_{\g k}$ acts on $U_i$ by
multiplication by a scalar.
The value of that scalar is easily determined by checking the action of
$\Omega_{\g k}$ on the highest weight vector of $U_i$:
$$
(\mu_i+\rho_c,\mu_i+\rho_c)-(\rho_c,\rho_c)
= (\mu_i+2\rho_c,\mu_i) \ge \|\mu_i\|^2.
$$
Note that $\mu_i$ is a complete invariant of $(\tau_i,U_i)$ and ranges over
a lattice intersected with certain cone (namely, the positive Weyl chamber).
Therefore, for $n$ sufficiently large,
$$
\sum_{i \in \hat K_{\BB R}} \bigl( 1+\|\mu_i\|^2 \bigr)^{-n} <\infty.
$$
Recall that the Weyl dimension formula implies that $\dim U_i$ is expressible
as a polynomial function of $\mu_i$.
We also know from (\ref{dim-est})
(at least for the principal series representations and hence for all
of their subrepresentations) that
$$
\dim V(i) \le (\text{length of $(\pi,V)$}) \cdot (\dim U_i)^2,
$$
which is bounded by a polynomial in $(1+\|\mu_i\|^2)$.
Note that, for $n > 0$, %$n \gg 0$,
$$
\pi(1+\Omega_{\g k})^{-n} :
\bigoplus_{i \in \hat K_{\BB R}} V(i) \to \bigoplus_{i \in \hat K_{\BB R}} V(i)
$$
extends to a bounded operator on $V$ -- in fact, a Hilbert-Schmidt operator --
if $n$ is sufficiently large. For $f \in {\cal C}^{\infty}_c(G_{\BB R})$, write
\begin{equation}  \label{factorization}
\pi(f) = \pi \bigl( r(1+\Omega_{\g k})^{2n} f \bigr) \cdot
\pi(1+\Omega_{\g k})^{-n} \cdot \pi(1+\Omega_{\g k})^{-n}.
\end{equation}
Thus we have expressed $\pi(f)$ as a composition of two Hilbert-Schmidt
operators, hence $\pi(f)$ is of trace class.
This proves that $\tr \pi(f)$ is well-defined and equals
$\sum_k \int_{G_{\BB R}} \theta_k\cdot f(g)\,dg$, where
$\theta_k: G_{\BB R} \to \BB C$ are the diagonal matrix coefficient
functions $\langle \pi(g)v_k,v_k \rangle$.

Note that
$$
\bigl\| \pi \bigl( r(1+\Omega_{\g k})^{2n} f \bigr) \bigr\|
\le C \cdot \sup \bigl| r(1+\Omega_{\g k})^{2n} f \bigr|,
$$
where the constant $C$ depends only on the support of $f$
(follows from the uniform boundedness principle).
This means that in the factorization (\ref{factorization}) of $\pi(f)$ one
Hilbert-Schmidt operator is independent of $f$ and the other has
Hilbert-Schmidt norm bounded by the operator norm of
$\pi \bigl( r(1+\Omega_{\g k})^{2n} f \bigr)$.
We conclude that the map $f \mapsto \tr \pi(f)$ is continuous in the topology
of ${\cal C}^{\infty}_c(G_{\BB R})$, i.e. $\Theta_{\pi}$ is a distribution.

Finally, for $g \in G_{\BB R}$,
we can write $f$ conjugated by $g$ as $l(g)r(g)f$, then
by Lemma \ref{trace-lem} we get
$$
\tr \pi \bigl( l(g)r(g)f \bigr)
= \tr \bigl( \pi(g^{-1}) \pi(f) \pi(g) \bigr) = \tr \pi(f),
$$
which proves that the distribution $\Theta_{\pi}$ is conjugation-invariant.
\end{proof}

\subsection{Properties of Characters}

As before, let $\{v_k\}$ be a (Hilbert space) orthonormal basis of $V$
obtained by combining orthonormal bases for all $V(i)$, $i \in \hat K_{\BB R}$.
By Theorem \ref{densityof}, each diagonal matrix coefficient function
$\theta_k(g) = \langle \pi(g)v_k,v_k \rangle$ is real analytic.
We have already mentioned that $\Theta_{\pi} = \sum_k \theta_k$
in the sense of distributions.
Note that, for each $k$, the Taylor series of $\theta_k$ at the identity
element (or at an arbitrary point of $K_{\BB R}$ when $G_{\BB R}$ is not connected)
is completely determined by the structure of $V_{fini}$ as a Harish-Chandra
module. Thus we conclude:

\begin{lem}
The characters of any two infinitesimally equivalent representations coincide.
\end{lem}

Also observe that if
$$
0 \to (\pi_1,V_1) \to (\pi,V) \to (\pi_2,V_2) \to 0
$$
is a short exact sequence of representations of $G_{\BB R}$, then
$$
\Theta_{\pi} = \Theta_{\pi_1} + \Theta_{\pi_2}
$$
(even if the sequence does not split).
Thus we can say that the character of a representation $(\pi,V)$ is an
invariant of the ``infinitesimal semisimplification of $(\pi,V)$''.

\begin{thm}
Let $\{(\pi_{\alpha},V_{\alpha})\}_{\alpha \in {\cal A}}$ be a collection of
irreducible admissible representations of $G_{\BB R}$, no two of which are
infinitesimally equivalent. Then $\{ \Theta_{\pi_{\alpha}} \}_{\alpha \in {\cal A}}$
is a linearly independent set.
\end{thm}

\begin{cor}
An irreducible admissible representation of $G_{\BB R}$ is uniquely determined
by its character, up to infinitesimal equivalence.
\end{cor}

\begin{cor}
The character $\Theta_{\pi}$ completely determines the irreducible components
of $(\pi,V)$, up to infinitesimal equivalence.
\end{cor}

If $\dim V < \infty$, $(\pi,V)$ is the direct sum of its irreducible
components, hence completely determined by $\Theta_{\pi}$.
If $\dim V = \infty$, the character $\Theta_{\pi}$ determines the irreducible
components of $(\pi,V)$, but cannot tell how $(\pi,V)$ is built out of them.
Hence $\Theta_{\pi}$ only determines the infinitesimal equivalence class of the
semisimplification of $(\pi,V)$.

As usual, we assume that $G_{\BB R}$ is connected and let ${\cal ZU}(\g g)$
denote the center of the universal enveloping algebra ${\cal U}(\g g)$.
%Recall the isomorphism from Proposition \ref{UDiso}:
%$$
%{\cal U}(\g g) \simeq
%\biggl\{ \begin{matrix} \text{algebra of left-invariant} \\
%\text{differential operators on $G_{\BB R}$} \end{matrix} \biggr\},
%$$
%under this isomorphism,
%$$
%{\cal ZU}(\g g) \simeq
%\biggl\{ \begin{matrix} \text{algebra of bi-invariant} \\
%\text{differential operators on $G_{\BB R}$} \end{matrix} \biggr\}.
%$$
Suppose that the representation $(\pi,V)$ is irreducible, admissible.
Then ${\cal ZU}(\g g)$ acts on $V^{\infty}$ by an infinitesimal character
$\chi_{\pi}: {\cal ZU}(\g g) \to \BB C$ (recall Corollary \ref{inf-char-cor}).
This means that, for all $z \in {\cal ZU}(\g g)$ and all $v \in V^{\infty}$,
$\pi(z)v=\chi_{\pi}(z)v$. In particular, for a diagonal matrix coefficient
$\theta_k(g) = \langle \pi(g)v_k,v_k \rangle$, we have
$r(z) \theta_k = \chi_{\pi}(z) \theta_k$.
Since $\Theta_{\pi} = \sum_k \theta_k$, we obtain:

\begin{cor}
For all $z \in {\cal ZU}(\g g)$, we have
$r(z) \Theta_{\pi} = \chi_{\pi}(z) \Theta_{\pi}$.
In other words, the character of an irreducible representation $\Theta_{\pi}$
is a conjugation-invariant eigendistribution.
\end{cor}

\begin{rem}
This is a very important property of characters of irreducible
representations. Other properties of characters -- such as Harish-Chandra
Regularity Theorem --  are proved in a slightly more general setting of
conjugation-invariant eigendistributions.
But not every conjugation-\\invariant eigendistribution is the character
of a representation.
\end{rem}

If $(\pi,V)$ is admissible of finite length (but not necessarily irreducible),
$\Theta_{\pi}$ is a finite sum of conjugation-invariant eigendistributions.
This implies that $\Theta_{\pi}$ is annihilated by an ideal
$I \subset {\cal ZU}(\g g)$ of finite codimension.

\subsection{Harish-Chandra Regularity Theorem}  \label{regularity-subsection}

As before, we denote the complexification of $G_{\BB R}$ by $G$.
For an element $g \in G$, the following three properties are equivalent:
\begin{itemize}
\item
$g$ acts semisimply in every finite-dimensional representation of $G$;
\item
It does so for one faithful finite-dimensional representation;
\item
The operator $\operatorname{Ad}(g)$ on $\g g$ is semisimple.
\end{itemize}
We call such elements of $G$ {\em semisimple}.
Similarly, one can define {\em unipotent} elements of $G$.
The unipotent elements constitute a proper algebraic subvariety in $G$
and a real-analytic subvariety in $G_{\BB R}$.

For an element $g \in G$, we use the following notation for the centralizer
of $g$:
$$
Z_G(g) = \{ h \in G ;\: gh=hg \}.
$$
We have the following analogue of the Jordan Canonical Form Theorem:

\begin{thm}[Borel]
Every $g \in G$ can be written uniquely in the form $g=g_{ss}g_u$, where
$g_{ss} \in G$ is semisimple and $g_u \in G$ is unipotent,
and $g_{ss}g_u=g_ug_{ss}$. Moreover, $g_{ss}$ and $g_u$ lie in the center of
$Z_G(g)$ and
$$
Z_G(g) = Z_G(g_{ss}) \cap Z_G(g_u).
$$
\end{thm}

And the same is true for $G_{\BB R}$ in place of $G$.

We call an element $g \in G$ {\em regular} if $\dim Z_G(g)$ is minimal.
An element $g \in G_{\BB R}$ is regular if it is regular as an element of $G$
or, equivalently, if $\dim Z_{G_{\BB R}}(g)$ is minimal. Set
\begin{align*}
G_{rs} &=_{\text{def}} \text{set of all regular semisimple elements in $G$}, \\
(G_{\BB R})_{rs} &=_{\text{def}} G_{\BB R} \cap G_{rs}.
\end{align*}

Now we give another description of the regular semisimple set $G_{rs}$.
Define a function $D_0$ on $G$ by
$$
\det \bigl( \operatorname{Ad}(g) + (\lambda-1) \bigr) =
\lambda^r D_0(g) + \lambda^{r+1} D_1(g) + \dots,
$$
where $r$ is the smallest power of $\lambda$ such that its coefficient
$D_0(g)$ is not identically zero.
For example, suppose that $H \subset G$ is a Cartan subgroup, then,
for $g \in H$,
$$
\det \bigl( \operatorname{Ad}(g) + (\lambda-1) \bigr) =
\lambda^r \prod_{\alpha \in \Phi} (e^{\alpha}(g)-1+\lambda),
$$
with $r = \dim H$. Hence
\begin{equation}  \label{D_0}
D_0(g) \bigr|_H = \prod_{\alpha \in \Phi} (e^{\alpha}(g)-1).
\end{equation}
Then
$$
G_{rs} = \{ g \in G ;\: D_0(g) \ne 0 \},
$$
\begin{equation}  \label{G_rs}
%G_{rs} = \{ g \in G ;\: D_0(g) \ne 0 \}, \qquad
(G_{\BB R})_{rs} = \{ g \in G_{\BB R} ;\: D_0(g) \ne 0 \}.
\end{equation}
In fact, $(G_{\BB R})_{rs}$ is a Zariski open conjugation-invariant subset of
$G_{\BB R}$, its complement has measure zero.

We are almost ready to state the Harish-Chandra Regularity Theorem.
We identify functions $F$ and distributions obtained by integrating against
these functions $F$:
$$
F \quad \leftrightsquigarrow \quad
f \mapsto \int_{G_{\BB R}} F(g) \cdot f(g) \,dg.
$$
Note that the integral $\int_{G_{\BB R}} F(g) \cdot f(g) \,dg$ is finite for all
test functions $f \in {\cal C}^{\infty}_c(G_{\BB R})$ if and only if $F$ is
{\em locally integrable}, i.e. $\int_C |F(g)|\,dg < \infty$ for all
compact subsets $C \subset G_{\BB R}$.
Not every distribution on $G_{\BB R}$ can be expressed in this form.
For example, the delta distributions and their derivatives cannot be
expressed as integrals against locally integrable functions.

\begin{thm}[Harish-Chandra Regularity Theorem]
Let $\Theta$ be a conjugation-invariant eigendistribution for ${\cal ZU}(\g g)$.
Then
\begin{itemize}
\item[{\rm a)}]
$\Theta \bigr|_{(G_{\BB R})_{rs}}$ is a real-analytic function;
\item[{\rm b)}]
This real-analytic function -- regarded as a measurable function on all of
$G_{\BB R}$ -- is locally integrable;
\item[{\rm c)}]
This locally integrable function {\em is} the distribution $\Theta$.
\end{itemize}
\end{thm}

In part (a), the restriction of $\Theta$ to $(G_{\BB R})_{rs}$ means that we
evaluate the distribution $\Theta$ on smooth functions $f$ with
$\supp f \subset (G_{\BB R})_{rs}$ only.
Intuitively, $\Theta$ is a real-analytic function, possibly with singularities
along the complement of $(G_{\BB R})_{rs}$; part (c) can be interpreted as
``there are no delta functions or their derivatives hiding in
$G_{\BB R} \setminus (G_{\BB R})_{rs}$''.

An accessible proof of this theorem for $G_{\BB R} = SL(2,\BB R)$ which
generalizes to the general case is given in \cite{At}.
Part (a) is proved by showing that the distribution
$\Theta \bigr|_{(G_{\BB R})_{rs}}$ satisfies an elliptic system of partial
differential equations with real analytic coefficients on $(G_{\BB R})_{rs}$.
Therefore, by elliptic regularity, $\Theta \bigr|_{(G_{\BB R})_{rs}}$ must be
a real analytic function as well.
This elliptic system as well as its solutions can be spelled out explicitly,
which is used to prove parts (b) and (c).

The proof of this theorem also shows that the set of Harish-Chandra characters
attached to a single infinitesimal character is finite dimensional.
It is easy to show that the principal series representations {\em have}
infinitesimal characters, and we already know that they are admissible.
We can conclude that the principal series representations have finite length.

\subsection{Characters of Representations of $SU(1,1)$}

In this subsection we give without proofs formulas for characters of
some representations of $SU(1,1)$.

An element in $SL(2,\BB R)$ or $SU(1,1)$ is semisimple if and only if it is
diagonalizable (as a $2 \times 2$ complex matrix).
By (\ref{D_0})-(\ref{G_rs}), an element $g$ in $SL(2,\BB R)$
or $SU(1,1)$ is regular semisimple if it is diagonalizable and its eigenvalues
are not $\pm1$, i.e. $g$ has distinct eigenvalues.
There are two possible scenarios: the eigenvalues of $g$ are real and
the eigenvalues of $g$ are complex with non-zero imaginary part.

If $g \in SL(2,\BB R)$ has distinct real eigenvalues, then $g$ is
$SL(2,\BB R)$-conjugate to
$\bigl(\begin{smallmatrix} a & 0 \\ 0 & a^{-1} \end{smallmatrix}\bigr)$,
for some $a \in \BB R^{\times}$, such elements are called {\em hyperbolic}.
If $g \in SL(2,\BB R)$ has eigenvalues in $\BB C \setminus \BB R$, then
the eigenvalues must be of the form $e^{i\theta}$ and $e^{-i\theta}$, for some
$\theta \in (0,\pi) \cup (\pi,2\pi)$, and $g$ is $SL(2,\BB R)$-conjugate to
$\bigl(\begin{smallmatrix} \cos\theta & -\sin\theta \\
\sin\theta & \cos\theta \end{smallmatrix}\bigr)$,
such elements are called {\em elliptic}.
It is easy to tell apart elliptic and hyperbolic elements:
\begin{align*}
\text{$g \in SL(2,\BB R)$ is elliptic} \quad
&\Longleftrightarrow \quad |\tr g| < 2, \\
\text{$g \in SL(2,\BB R)$ is hyperbolic} \quad
&\Longleftrightarrow \quad |\tr g| > 2.
\end{align*}

Similarly, if $g \in SU(1,1)$ has distinct real eigenvalues, it is called
hyperbolic, and $g$ is $SU(1,1)$-conjugate to
$\pm \bigl(\begin{smallmatrix} \cosh t & \sinh t \\ \sinh t & \cosh t
\end{smallmatrix}\bigr)$, for some $t \in \BB R^{\times}$.
If $g \in SU(1,1)$ has eigenvalues in $\BB C \setminus \BB R$, it is called
elliptic, and $g$ is $SU(1,1)$-conjugate to
$\bigl(\begin{smallmatrix} e^{i\theta} & 0 \\ 0 & e^{-i\theta}
\end{smallmatrix}\bigr)$, for some $\theta \in (0,\pi) \cup (\pi,2\pi)$.
As before,
\begin{align*}
\text{$g \in SU(1,1)$ is elliptic} \quad
&\Longleftrightarrow \quad |\tr g| < 2, \\
\text{$g  \in SU(1,1)$ is hyperbolic} \quad
&\Longleftrightarrow \quad |\tr g| > 2.
\end{align*}
Furthermore, elements
$$
\begin{pmatrix} \cosh t & \sinh t \\ \sinh t & \cosh t \end{pmatrix}
\quad \text{and} \quad
\begin{pmatrix} \cosh t & -\sinh t \\ -\sinh t & \cosh t \end{pmatrix}
$$
are conjugate to each other via
$\bigl(\begin{smallmatrix} i & 0 \\ 0 & -i \end{smallmatrix}\bigr)
\in SU(1,1)$, but
$$
\begin{pmatrix} e^{i\theta} & 0 \\ 0 & e^{-i\theta} \end{pmatrix}
\quad \text{and} \quad
\begin{pmatrix} e^{-i\theta} & 0 \\ 0 & e^{i\theta} \end{pmatrix}
$$
are {\em not} $SU(1,1)$-conjugate.
(And a similar statement holds for $SL(2,\BB R)$.)
Thus, a conjugation-invariant function on $SU(1,1)_{rs}$, such as a character,
is uniquely determined by its values on the following elements
$$
k_{\theta} = \begin{pmatrix} e^{i\theta} & 0 \\ 0 & e^{-i\theta} \end{pmatrix},
\quad \theta \in (0,\pi) \cup (\pi,2\pi),
$$
together with
$$
\pm a_t, \quad \text{where} \quad
a_t = \begin{pmatrix} \cosh t & \sinh t \\ \sinh t & \cosh t \end{pmatrix},
%\text{ and }
%-a_t = -\begin{pmatrix} \cosh t & \sinh t \\ \sinh t & \cosh t \end{pmatrix},
\quad t>0.
$$

We start with finite dimensional representations.
In this case the character is well-defined as a function on all elements of
$SU(1,1)$ and there is no need for invoking distributions.

\begin{prop}
Let $(\pi_d,F_d)$ be the irreducible finite dimensional representation of
$SU(1,1)$ of dimension $d+1$. Then
\begin{align*}
\tr\pi_d(k_{\theta}) &=
\frac{e^{i(d+1)\theta} - e^{-i(d+1)\theta}}{e^{i\theta} - e^{-i\theta}}, \\
\tr\pi_d(a_t) &= \frac{e^{(d+1)t} - e^{-(d+1)t}}{e^t - e^{-t}}, \\
\tr\pi_d(-a_t) &= (-1)^{d+1} \frac{e^{(d+1)t} - e^{-(d+1)t}}{e^t - e^{-t}}.
\end{align*}
\end{prop}

\begin{proof}
Since $k_{\theta} = \exp(i\theta H)$, by Proposition \ref{finite-dim},
$k_{\theta}$ acts on $F_d$ with eigenvalues
$$
e^{-id\theta}, \quad e^{i(2-d)\theta}, \quad e^{i(4-d)\theta}, \quad \dots, \quad
e^{i(d-2)\theta}, \quad e^{id\theta}.
$$
Hence we get a geometric series
\begin{multline*}
\tr\pi_d(k_{\theta}) = e^{-id\theta} + e^{i(2-d)\theta} + e^{i(4-d)\theta} + \dots
+ e^{i(d-2)\theta} + e^{id\theta}  \\
= \frac{e^{i(d+1)\theta} - e^{-i(d+1)\theta}}{e^{i\theta} - e^{-i\theta}}.
\end{multline*}

To find $\tr \pi_d(a_t)$, we notice that finite dimensional representations of
$SU(1,1)$ extend to $SL(2,\BB C)$ and $a_t$ is $SL(2,\BB C)$-conjugate to
$\bigl(\begin{smallmatrix} e^t & 0 \\ 0 & e^{-t} \end{smallmatrix}\bigr)
= \exp(tH)$. Hence
\begin{multline*}
\tr\pi_d(a_t) = \tr \pi_d(\exp(tH))  \\
= e^{-dt} + e^{(2-d)t} + e^{(4-d)t} + \dots + e^{(d-2)t} + e^{dt} \\
= \frac{e^{(d+1)t} - e^{-(d+1)t}}{e^t - e^{-t}}.
\end{multline*}

Finally, $\bigl(\begin{smallmatrix} -1 & 0 \\ 0 & -1 \end{smallmatrix}\bigr)
\in SU(1,1)$
acts on $F_d$ trivially if $d$ is odd and by $-Id$ if $d$ is even. Hence
$$
\tr\pi_d(-a_t) = (-1)^{d+1} \frac{e^{(d+1)t} - e^{-(d+1)t}}{e^t - e^{-t}}.
$$
\end{proof}

Next we consider the holomorphic and antiholomorphic discrete series
representations as well as their limits.

\begin{thm}
Let
$$
(\pi_{\chi}, {\cal D}_l^-) \quad \text{and} \quad (\pi_{\chi}, {\cal D}_l^+),
\qquad l=-\frac12,-1,-\frac32,-2,\dots,
$$
be the holomorphic and antiholomorphic discrete series representations of
$SU(1,1)$ together with their limits constructed in
Section \ref{V-chi-construction}. We denote by $\Theta^-_{\pi_{\chi}}$ and
$\Theta^+_{\pi_{\chi}}$ their respective characters. Then
$$
\Theta^-_{\pi_{\chi}} (k_{\theta}) = \frac{e^{i(2l+1)\theta}}{e^{i\theta}-e^{-i\theta}},
\qquad
\Theta^+_{\pi_{\chi}} (k_{\theta}) = -\frac{e^{-i(2l+1)\theta}}{e^{i\theta}-e^{-i\theta}},
$$
\begin{align*}
\Theta^-_{\pi_{\chi}} (a_t) = \Theta^+_{\pi_{\chi}} (a_t)
&= \frac{e^{(2l+1)|t|}}{|e^t - e^{-t}|}, \\
\Theta^-_{\pi_{\chi}} (-a_t) = \Theta^+_{\pi_{\chi}} (-a_t)
&= (-1)^{2l} \frac{e^{(2l+1)|t|}}{|e^t - e^{-t}|}.
\end{align*}
%\begin{align*}
%\Theta^-_{\pi_{\chi}} (k_{\theta}) &= \frac{e^{i(2l+1)\theta}}{e^{i\theta}-e^{-i\theta}}, \\
%\Theta^-_{\pi_{\chi}} (a_t) &= \frac{e^{(2l+1)|t|}}{|e^t - e^{-t}|}, \\
%\Theta^-_{\pi_{\chi}} (-a_t) &= (-1)^{2l} \frac{e^{(2l+1)|t|}}{|e^t - e^{-t}|}.
%\end{align*}
%\begin{align*}
%\Theta^+_{\pi_{\chi}} (k_{\theta}) &= -\frac{e^{-i(2l+1)\theta}}{e^{i\theta}-e^{-i\theta}}, \\
%\Theta^+_{\pi_{\chi}} (a_t) &= \frac{e^{(2l+1)|t|}}{|e^t - e^{-t}|}, \\
%\Theta^+_{\pi_{\chi}} (- a_t) &= (-1)^{2l} \frac{e^{(2l+1)|t|}}{|e^t-e^{-t}|}.
%\end{align*}
\end{thm}

Note that $k_{\theta} = \exp(i\theta H)$ acts on ${\cal D}_l^-$ with weights
$$
e^{2il\theta}, \quad e^{2i(l-1)\theta}, \quad e^{2i(l-2)\theta}, \quad e^{2i(l-3)\theta},
\dots.
$$
Hence we get a geometric series
\begin{multline*}
\Theta^-_{\pi_{\chi}} (k_{\theta}) =
e^{2il\theta} + e^{2i(l-1)\theta} + e^{2i(l-2)\theta} + e^{2i(l-3)\theta} + \dots \\
= \lim_{n \to \infty}
\frac{e^{i(2l+1)\theta} - e^{i(2l+1-2n)\theta}}{e^{i\theta} - e^{-i\theta}}.
\end{multline*}
Now, $e^{i(2l+1-2n)\theta}$ does not approach zero as a function, but
$\lim_{n \to \infty} e^{i(2l+1-2n)\theta} = 0$ in the sense of distributions.
In other words, let $\tilde F_n$ be a distribution on $SU(1,1)_{rs}$
$$
\tilde F_n(f) = \int_{SU(1,1)} F_n(g) \cdot f(g)\,dg, \qquad
f \in {\cal C}^{\infty}_c(SU(1,1)_{rs}),
$$
where
$$
F_n(g) = \begin{cases}
\frac{e^{i(2l+1-2n)\theta}}{e^{i\theta} - e^{-i\theta}} &
\text{if $g$ is $SU(1,1)$-conjugate to $k_{\theta}$}; \\
0 & \text{otherwise},
\end{cases}
$$
then $\lim_{n \to \infty} \tilde F_n(f) = 0$ for all
$f \in {\cal C}^{\infty}_c(SU(1,1)_{rs})$.
This assertion follows from the fact that the Fourier coefficients
$\hat \phi(n)$ of a smooth (or, more generally, integrable)
function on a circle\footnote{The denominator
  $e^{i\theta} - e^{-i\theta}$ can be combined with $f$,
  and since the support of $f$ lies in $SU(1,1)_{rs}$,
  this results in a smooth function $\phi$.}
satisfy $\lim_{n \to \pm \infty} \hat \phi(n) =0$
(this result is known as the Riemann-Lebesgue lemma).
Similarly, one can compute $\Theta^+_{\pi_{\chi}} (k_{\theta})$.
Computing $\Theta^{\mp}_{\pi_{\chi}} (a_t)$ requires more work.

Finally, we consider the principal series representations.

\begin{thm}
Let
$(\pi_{\chi}, V_{\chi})$, $\chi=(l,\epsilon)$, be the principal series
representations of $SU(1,1)$ constructed in Section \ref{V-chi-construction}.
Then
\begin{align*}
\Theta_{\pi_{\chi}} (k_{\theta}) &= 0, \\
\Theta_{\pi_{\chi}} (a_t) &= \frac{e^{(2l+1)t} + e^{-(2l+1)t}}{|e^t - e^{-t}|}, \\
\Theta_{\pi_{\chi}} (-a_t) &=
(-1)^{2\epsilon} \frac{e^{(2l+1)t} + e^{-(2l+1)t}}{|e^t - e^{-t}|}.
\end{align*}
\end{thm}

Note that when $l=-1,-3/2,-2,-5/2,\dots$ and $l+\epsilon \in \BB Z$,
the character of the principal series representation $(\pi_{\chi}, V_{\chi})$
is the sum of characters of the holomorphic discrete series
$(\pi_{\chi}, {\cal D}_l^-)$, antiholomorphic discrete series representation
$(\pi_{\chi}, {\cal D}_l^+)$ and the finite dimensional representation
$F_{|2(l+1)|}$, as it should be by (\ref{exact_seq}).
Also, when $l=-1/2$ and $\epsilon=1/2$, the character of $(\pi_{\chi}, V_{\chi})$
is the sum of characters of the limits of the holomorphic and antiholomorphic
discrete series $(\pi_{\chi}, {\cal D}_{-1/2}^-)$ and
$(\pi_{\chi}, {\cal D}_{-1/2}^+)$, which agrees with (\ref{dir_sum}).

\end{document}